\documentclass[preprint,3p]{elsarticle}
\usepackage{stmaryrd}
\usepackage{amsmath,cases}
\usepackage{amssymb,amsfonts,bm}
\usepackage{amsthm}
\usepackage{dsfont}
\usepackage{color,graphicx,caption,subcaption}
\usepackage[svgnames]{xcolor}
\usepackage{hyperref}
\usepackage{xspace}
\usepackage{fullpage}
\usepackage{placeins}
\usepackage{enumitem}
\usepackage{epstopdf,caption,graphicx,subcaption,mathtools}
\usepackage{url}
\usepackage{verbatim}
\usepackage{amsthm}
\usepackage{commath}
\usepackage{multirow}
\usepackage[ruled]{algorithm2e}
\usepackage{ulem}
\usepackage{booktabs}
\usepackage{float}
\usepackage{xpatch}
\xpatchcmd{\paragraph}{\normalfont}{{\normalfont\bfseries}}{}{}
\usepackage{tabularx}

\newcommand{\be}{\begin{equation}}
\newcommand{\ee}{\end{equation}}
\newcommand{\ba}{\begin{array}}
\newcommand{\ea}{\end{array}}
\def\bea{\begin{eqnarray}}
\def\eea{\end{eqnarray}}
\def \beas{\begin{eqnarray*}}
\def \eeas{\end{eqnarray*}}
\newtheorem{definition}{Definition}[section]

 \newtheorem{remark}{Remark}[section]
\newtheorem{exmp}{Example}[section]
\newtheorem{thm}{Theorem}[section]
\newtheorem{lemma}{Lemma}[section]
\newtheorem{prop}{Proposition}[section]
 
\newtheorem{corollary}{Corollary}[section]

\DeclareMathOperator*{\argmin}{arg\,min}

\def\<{\langle}
\def\>{\rangle}

\def\ba{{\bf a}}

\usepackage{epsfig,epstopdf,bm}
\input psfig.sty

\begin{document}
\begin{center}
	\large \bf On preconditioned Riemannian gradient methods for minimizing the Gross-Pitaevskii energy functional: algorithms, global convergence and optimal local convergence rate
\end{center}
\vspace{6pt}
\begin{center}
	\large  Zixu Feng  and Qinglin Tang$^{\it *}$
\end{center} 
\vspace{6pt}	

\normalsize
\rm
\begin{center}\it		
	School of Mathematics, Sichuan University, Chengdu 610064,  P. R. China
	
	$^*$E-mail: qinglin\_tang@scu.edu.cn
\end{center}

\footnotesize
\vspace{6pt}
\textbf{Abstract.} In this article, we propose a unified framework to develop and analyze a class of preconditioned Riemannian gradient methods (P-RG) for   minimizing   Gross-Pitaevskii (GP) energy functionals with rotation on a Riemannian manifold. This framework enables one to carry out  a comprehensive analysis of all existing projected Sobolev gradient methods, and more important, to construct a P-RG with improved efficiency for computing minimizers of GP energy functionals. For mild assumptions on the preconditioner, the energy dissipation and global convergence of the P-RG are thoroughly proved. As for the local convergence analysis of the P-RG, it is  much more challenging due to the two invariance properties of the GP energy functional caused by phase shifts and rotations. To address this issue,  assuming  the GP energy functional is a Morse-Bott functional, we first derive the celebrated Polyak-\L ojasiewicz (PL) inequality  around its minimizer. The PL inequality is sharp, therefore allows us to precisely characterize the local convergence rate of the P-RG via condition number $\frac{\mu}{L}$. Here, $\mu$ and $L$ are respectively the lower and upper bound of the spectrum of an combined operator closely related to the preconditioner and Hessian of the GP energy functional on a closed subspace. The precise form of the convergence rate $\sqrt{1-\frac{\mu-\varepsilon}{L+\varepsilon}}$ (with $\varepsilon > 0$ sufficiently small) reveals how to design a preconditioner. Based on this analysis, we construct a quasi-optimal preconditioner and establish its optimal local convergence rate via spectral analysis, i.e., $\frac{L-\mu}{L+\mu}+\varepsilon$, which is the best rate one can possibly achieve for a Riemannian gradient method. To the best of our knowledge, this study represents the first to rigorously derive the local convergence rate of the P-RG for minimizing the Gross-Pitaevskii energy functional with two symmetric structures. Finally, numerical examples related to rapidly rotating Bose-Einstein condensates are carried out to compare the performances of P-RG with different preconditioners and to verify the theoretical findings.
\vspace{6pt}


\vspace{6pt}
\normalsize



\section{Introduction}
\setcounter{equation}{0}

The Gross-Pitaevskii energy functional and the corresponding equation play a crucial role in various domains of quantum physics, particularly in cold atom physics, nonlinear optics, astrophysics, quantum fluids and turbulence \cite{1995Observation,2014Intro,2013Quan, K1995Bose,2000Fuzzy,Klaers2010Bose}. It originates from the description of Bose-Einstein condensates (BECs), a macroscopic quantum phenomenon where a large number of bosons occupy the lowest quantum state at extremely low temperatures. Subsequently, the application of this theory has been extended to other fields. In nonlinear optics, the propagation equations of light pulses in nonlinear media share a similar form with the Gross-Pitaevskii equation, facilitating the study of spatial optical solitons and vortex beams. Moreover, hypothetical dark matter candidates, such as ultra-light axions, or the interiors of neutron stars may exhibit BEC-like coherence on macroscopic scales, suggesting potential applications of the Gross-Pitaevskii equation in astrophysical contexts. Additionally, the Gross-Pitaevskii equation is employed to investigate turbulence phenomena, including the entanglement of vortex lines and energy cascades in quantum fluids. 

The minimizer of the Gross-Pitaevskii energy functional holds significant importance in physics, particularly in describing BECs and other quantum systems.
Mathematically, minimizers of the Gross-Pitaevskii energy functional are defined under the $L^2$ normalization constraint. As outlined in the comprehensive review by Bao et al. \cite{2013Mathematical}, the dimensionless Gross-Pitaevskii energy functional incorporating the rotation term is given by
\begin{align}\label{GP-Energy}
	E(\phi):=\frac{1}{2}\int_{\mathbb{R}^d}\left(\frac{1}{2}|\nabla\phi|^2+V(\bm{x})|\phi|^2
	-\Omega\overline{\phi} \mathcal{L}_z\phi+F(\rho_{\phi})\right)\text{d}\bm{x}.
\end{align}
Here, $\bm{x}\in\mathbb{R}^d\ (d=2,3)$ denotes spatial variables, with $\bm{x}=(x, y)^T$ in two-dimensional or $\bm{x}=(x, y, z)^T$ in three-dimensional. {The function} $V(\bm{x})$ is a real-valued external potential and satisfies $\lim\limits_{|\bm{x}|\to\infty}V(\bm{x})=\infty$. The rotation term is characterized by the angular momentum $\mathcal{L}_z=-i(x\partial_y-y\partial_x)$ and the rotation frequency $\Omega\ge0$. {Let} $\overline{\phi}$ denotes the complex conjugate of $\phi$. The nonlinear interaction term can be written as follows
\begin{align*}
	F(\rho_{\phi})=\int_{0}^{\rho_{\phi}}f(s)\;\text{d}s,\quad\ \rho_{\phi}:=|\phi|^2.
\end{align*}
In the physical literature, the real-valued function $f(s)$ is defined in the forms: $f(s)=\eta s$, $\eta s\log s$, and $\eta s+\eta_{LHY} s^{3/2}$ \cite{2021Rota,2006Deri,2020Supp,2019Rota}. 
The constraint is defined as
\begin{align*}
	N(\phi):=\|\phi\|_{L^2(\mathbb{R}^d)}^2=\int_{\mathbb{R}^d}|\phi|^2\;\text{d}\bm{x}=1.
\end{align*}
The minimizer of the Gross-Pitaevskii energy functional is represented by the macroscopic wave function $\phi_g$, which is defined as follows:
\begin{align}\label{O-P}
	\phi_g(\bm{x}):=\argmin_{\phi\in\mathcal{M}} E(\phi)\quad
	\mbox{with} \quad  \mathcal{M}:=\left\{\phi\in H^1(\mathbb{R}^d)\big|\|\phi\|_{L^2(\mathbb{R}^d)}^2=1\right\}.
\end{align}

Over the past two decades, various iterative solvers have been proposed to compute the minimizer of rotating or non-rotating Gross-Pitaevskii energy functional. These solvers mainly consist of energy minimization methods based on gradient flows \cite{2014Robust, 2017Efficient, 2003Computing, 2006Efficient, 2023Second, 2024On, 2000Ground, 2010A, 2017Computation, 2001Optimi, 2020Sobolev, 2023The, 2010Tackling, 2021Normalized, 2017A, 2024Anew, 2022Exp, 2019Efficient} and some nonlinear eigenvalue solvers \cite{2021The, 2007Dion, 2023The, 2014An}. Despite the large variety of methods, analytical convergence results are scarce, especially for cases involving rotation terms. For the non-rotating case ($\Omega = 0$), the first convergence result was obtained by Faou et al. \cite{2017Convergence}, who proved local convergence for the discrete normalized gradient flow (DNGF) in the cases where \(d = 1\) and $f(s)=\eta s$ with \(\eta \leq 0\). Later, in \cite{2023The}, Henning interpreted DNGF as a special inverse power iteration method and derived its local convergence results for \(d = 1, 2, 3\) and $f(s)=\eta s$ with \(\eta \geq 0\). Some convergence results for a series of time-semidiscretized projected Sobolev gradient flows were obtained in \cite{2024On, 2020Sobolev, 2023The, 2022Exp}, again for \(d = 1, 2, 3\) and $f(s)=\eta s$ with \(\eta \geq 0\). These convergence results rely on a special property of the ground state: the ground state of the nonlinear problem is also the unique ground state of its linearized version (cf. \cite{2010Numerical}), which cannot apply to the rotating cases ($\Omega>0$). To the best of our knowledge, only two studies have demonstrated the convergence of iterative solvers for the rotating cases. These are the $J$-method \cite{2021The} (a particular inverse iteration method originally proposed by Jarlebring et al. \cite{2014An}) and the adaptive Riemannian gradient method \cite{2024Convergence} (also known as the projected Sobolev gradient method, first proposed by Henning et al. \cite{2020Sobolev}). The difficulty of this problem \eqref{O-P} lies in the non-convexity of the constraint functional and the invariance properties of the Gross-Pitaevskii energy functional. 
\noindent $1)$ The first invariance property arises from phase shifts: for a minimizer \(\phi_g\) and any \(\alpha \in [-\pi, \pi)\), a global phase translation \(e^{i\alpha}\phi_g\) remains a minimizer. $2)$ The second invariance property comes from coordinate rotations: assuming the trapping potential \(V(\bm{x})\) is rotationally symmetric about the $z$-axis, i.e., for any \(\beta \in [-\pi, \pi)\), \(V(\bm{x})=V(A_{\beta}\bm{x})\), where
\begin{align*}
	A_{\beta}=\left(\begin{matrix}
		\cos\beta &-\sin\beta\\
		\sin\beta &\cos\beta
	\end{matrix}\right)\ \text{for}\; d=2,\quad A_{\beta}=\left(\begin{matrix}
		\cos\beta &-\sin\beta &0\\
		\sin\beta & \cos\beta &0\\
		0      & 0         &1
	\end{matrix}\right)\ \text{for}\; d=3.
\end{align*}
Then, for a minimizer \(\phi_g\) and any \(\beta \in [-\pi, \pi)\), a coordinate transformation \(\phi_g(A_{\beta}\bm{x})\) also produces a minimizer. 

\textbf{Contribution.} Previous studies \cite{2024Riem,2024On,2010A,2017Computation,2020Sobolev,2023The,2024Convergence,2010Tackling,2022Exp} have considered both non-rotational and rotational cases. {Our work primarily focuses on the rotating setting, where the situation differs significantly from the non-rotating case. To the best of our knowledge, only \cite{2024Convergence} has established a quantitative local convergence rate in this setting. However, this convergence rate describes convergence to an equivalence class of minimizers rather than to a specific limit point. Moreover, it is restricted to the specific preconditioner $\mathcal{P}_\phi = \mathcal{H}_\phi$.} The first major contribution of this work is the proposal of a unified framework for the design and analysis of preconditioned Riemannian gradient methods for minimizing the Gross-Pitaevskii energy functional. This framework considers both the phase shift invariance and the coordinate rotation invariance of the energy functional. Under the assumption that the energy functional is a Morse–Bott functional, we provide an exact characterization of the linear convergence rate for preconditioned Riemannian gradient methods. This framework encompasses all existing Sobolev gradient projection methods. Furthermore, by precisely characterizing the local convergence behavior, we derive the locally quasi-optimal preconditioner and identify the corresponding optimal local convergence rate. Finally, a central contribution of this work is the extension of the optimal convergence rate of Riemannian gradient descent from isolated minimizers satisfying the second-order sufficient condition to the Morse-Bott setting.

The rest of the paper is organized as follows: In Section \ref{Sec2}, we introduce preliminary notations and present the properties of minimization problem. In Section \ref{Sec3}, we present the necessary assumptions on the preconditioner and then introduce preconditioned Riemannian gradient methods and discuss its properties. In Section \ref{Sec4}, the convergence results of the proposed algorithms and the corresponding theoretical proofs are provided. In Section \ref{Sec5}, we verify the theoretical findings through a series of convincing numerical experiments. Finally, conclusions are presented in Section \ref{Sec6}.

\section{Preliminaries}\label{Sec2}
In this section, we introduce problem settings, basic notations, and some important properties of the problem. 

\subsection{Problem settings and notations}
In our analytical settings, the domain is truncated from the full space $\mathbb{R}^d$ to the bounded domain $\mathcal{D}$ and the homogeneous Dirichlet boundary condition is imposed on $\partial\mathcal{D}$ due to the trapping potential. On the bounded domain $\mathcal{D}$, we adopt standard notations for the Lebesgue spaces $L^p(\mathcal{D})=L^p(\mathcal{D},\mathbb{C})$ and the Sobolev space $H^1(\mathcal{D})=H^1(\mathcal{D},\mathbb{C})$ as well as the corresponding norms $\|\cdot\|_{L^p}$ and $\|\cdot\|_{H^1}$. Here, we drop the $\mathcal{D}$ dependence in the norms to simplify the notations. Thereby, we consider the Gross-Pitaevskii energy functional \eqref{GP-Energy} and the constrained optimization problem \eqref{O-P} on $\mathcal{D}$, i.e.,
\begin{align}\label{Riem-Opt-Problem}
	\nonumber E(\phi)&:=\frac{1}{2}\int_{\mathcal{D}}\left(\frac{1}{2}|\nabla\phi|^2+V(\bm{x})|\phi|^2
	-\Omega\overline{\phi} \mathcal{L}_z\phi+F(\rho_{\phi})\right)\text{d}\bm{x}\quad\text{and}\\
	\phi_g&:=\argmin_{\phi\in\mathcal{M}} E(\phi)\quad
	\mbox{with} \quad  \mathcal{M}:=\left\{\phi\in H_0^1(\mathcal{D})\big|\|\phi\|_{L^2}^2=1\right\}.
\end{align}
Furthermore, $\mathcal{M}$ is a Riemannian manifold, its tangent space is denoted by \(T_{\phi}\mathcal{M}\):
{\begin{align}\label{Tangent-space}
	T_{\phi}\mathcal{M}:=\left\{v\in H^1_0(\mathcal{D})\;\Big|\;\text{Re}\int_{\mathcal{D}}\phi\overline{v}\;\text{d}\bm{x}=0\right\}\quad\text{for}\quad \phi\in\mathcal{M}.
\end{align}}
For the simplicity of presentation, in what follows, we always assume that\\
\begin{itemize}
	\item[\bf (A1)] {The domain} $\mathcal{D}\subset\mathbb{R}^d$ is a bounded Lipschitz-domain that is rotationally symmetric about the $z$-axis for $d=2, 3$, such as a disk for $d=2$ and a ball for $d=3$.
	\item[\bf (A2)] {The potential function} $V\in L^{\infty}(\mathcal{D})$ is a rotationally symmetric about the $z$-axis, i.e., $V(\bm{x})=V(A_{\beta}\bm{x})$. 
	\item[\bf (A3)] {The nonlinear term} $f \colon \mathbb{R}_{+} \to \mathbb{R}_{+}$ is differentiable, satisfies $f(0) = 0$, and $\lim\limits_{s \to 0^+} f'(s^2)s^2 = 0$. Additionally, there exists $\theta \in [0,3)$ such that the function $s \mapsto f'(s^2)s^2$ is Lipschitz continuous with polynomial growth. Specifically, for every $s_1, s_2 \ge 0$, we have
	\begin{align*}
		\left| f'(s_1^2)s_1^2 - f'(s_2^2)s_2^2 \right| \le C (s_1 + s_2)^{\theta} |s_1 - s_2|.
	\end{align*}
	\item[\bf (A4)] There is a constant $K>0$ such that
	\begin{align*}
		V(\bm{x})-\frac{1+K}{2}\Omega^2(x^2+y^2)\ge0\quad \text{for almost all}\ \bm{x}\in\mathcal{D}.
	\end{align*}
	\item[\bf (A5)] If $\phi_g$ is a minimizer, then $\mathcal{L}_z\phi_g\in H_0^1(\mathcal{D})$.
\end{itemize}

\medskip

Let us begin with some explanations of the above assumptions. {\bf (A1)} and {\bf (A2)} ensure that the Gross-Pitaevskii energy functional possesses rotational invariance with respect to coordinate rotations. For  {\bf (A3)}, the condition $f\ge0$ can be relaxed to being lower-bounded, but for simplicity, we assume non-negativity. The assumption on  $f'$ is adapted from the classical reference \cite{2003Semi} to ensure that the Gross-Pitaevskii energy functional is $C^2(H^1_0(\mathcal{D}),\mathbb{R})$. Regarding {\bf (A4)}, we can relax the condition to allow values greater than a certain negative constant, but for simplicity in our analysis, we assume that {\bf (A4)} holds. Since any stationary states must be exponentially decaying, {\bf (A5)} is rarely violated in practical calculations. {{\bf (A5)} ensures that, under assumption {\bf (A2)}, $i\mathcal{L}_z\phi_g$ is well-defined in the tangent space $T_{\phi_g}\mathcal{M}$. If it were not satisfied, $i\mathcal{L}_z\phi_g$ would not lie in the tangent space, and thus could not be a zero eigenfunction of $E''(\phi_g) - \lambda_g\mathcal{I}$ (see {\bf Proposition~\ref{Prop1}}).} {We remark that {\bf (A5)} holds rigorously when the domain $\mathcal{D}$ is rotationally symmetric about the $z$-axis and has a $C^{1,1}$ boundary, since $\phi_g \in H^2(\mathcal{D}) \cap H_0^1(\mathcal{D})$ implies that the relevant trace is well defined and vanishes.} These assumptions we consider are widely accepted in both numerical simulations and physical experiments, making them meaningful in practice. Moreover, under the assumptions of {\bf (A1)}-{\bf (A4)}, the existence of minima \eqref{Riem-Opt-Problem} can be proven using standard techniques. For more details, see \cite{2013Mathematical}, which will not be discussed in this paper.

Since the Gross-Pitaevskii energy functional \(E\) is real-valued while the wave function \(\phi\) is complex-valued, \(E\) is not complex Fr\'echet differentiable in the usual complex Hilbert space. Therefore, we work within a real-linear space consisting of complex-valued functions, as done in \cite{2021The, 2003Semi}. In this setting, the function space is viewed as a real Hilbert space, meaning that all variations are taken with respect to real parameters. To this end, we equip the Lebesgue space \(L^2(\mathcal{D})\) and the Sobolev space \(H^1_0(\mathcal{D})\) with the following real inner products:
\begin{align*}
	(u,v)_{L^2}:=\text{Re}\int_{\mathcal{D}}u \overline{v}\;\text{d}\bm{x}\quad \text{and}\quad (u,v)_{H^1}:=\text{Re}\left(\int_{\mathcal{D}}u \overline{v}\;\text{d}\bm{x}+\int_{\mathcal{D}}\nabla u \overline{\nabla v}\;\text{d}\bm{x}\right).
\end{align*}
The corresponding real dual space is denoted
by $H^{-1}(\mathcal{D}):=\big(H^1_0(\mathcal{D})\big)^*$. And for any set $\mathcal{U}\subset\mathcal{M}$, we introduce the $\sigma$-neighborhood $\mathcal{B}_{\sigma}(\mathcal{U})$ of $\mathcal{U}$ by
\begin{align}
	\mathcal{B}_{\sigma}(\mathcal{U}):=\left\{\varphi\in\mathcal{M}\big|\exists\phi\in\mathcal{U},\|\varphi-\phi\|_{H^1}<\sigma\right\}.
\end{align}

Then, we define a real-symmetric and coercive bilinear form through the symmetric and coercive real linear operator $\mathcal{A}:H_0^1(\mathcal{D})\to H^{-1}(\mathcal{D})$ as follows:
\begin{equation}\label{Bilinear-operator}
	(u,v)_{\mathcal{A}} := \big\langle \mathcal{A}u, v \big\rangle \quad \text{for all}\; u,v \in H_0^1(\mathcal{D}),
\end{equation}
where $\langle\cdot,\cdot\rangle$ represents the canonical duality pairing between $H^{-1}(\mathcal{D})$ and $H_0^1(\mathcal{D})$. This bilinear form induces an inner product on $H_0^1(\mathcal{D})$, with the associated norm given by $\|v\|_{\mathcal{A}} := \sqrt{\langle\mathcal{A}v, v\rangle}$. Furthermore, for any closed subset $W \subset H_0^1(\mathcal{D})$, we denote its orthogonal complement relative to this inner product by
\begin{equation}\label{Orth-Comp}
	W^{\bot}_{\mathcal{A}} := \left\{ u \in H_0^1(\mathcal{D}) \big| (u,v)_{\mathcal{A}} = 0, \ {\text{for all}}\; v \in W \right\}.
\end{equation}

Finally, hereinafter, we introduce two types of constants:

\noindent $(i)$ {The constant} $ C $ denotes a generic constant depending only on $ \mathcal{D} $, $ d $, $ K $, and  $V_{\infty}:=\|V\|_{L^{\infty}}$. This includes constants arising from Sobolev inequalities.

\noindent $(ii)$ {The constant} $ C_{v_1,\dots, v_k} $ denotes a positive constant that depends monotonically increasing on the $ H^1$-norms of the functions $ v_1,\dots, v_k $. For any $j \in \{1,\dots,k\}$, if
\begin{equation}
	\|v_{j}\|_{H^1} \leq \|\widetilde{v}_{j}\|_{H^1},
\end{equation}
then it follows that
\begin{equation}
	C_{v_1,\dots, v_{j},\dots, v_k} \leq C_{v_1,\dots, \widetilde{v}_{j},\dots, v_k},
\end{equation}
and in particular, if $\|v_j\|_{H^1} \leq M$, we have
\begin{equation}
	C_{v_1,\dots, v_j,\dots, v_k} \leq C_{v_1,\dots, M,\dots, v_k}.
\end{equation}

\subsection{Properties of the problem}
Given $\phi\in H_0^1(\mathcal{D})$, we introduce a bounded real linear operator $\mathcal{H}_{\phi}:H_0^1(\mathcal{D})\to H^{-1}(\mathcal{D})$ by 
\begin{align}
	\left\langle\mathcal{H}_{\phi}u,v\right\rangle:=\frac{1}{2}\left(\nabla u,\nabla v\right)_{L^2}+\left(\left(V -\Omega\mathcal{L}_z\right)u,v\right)_{L^2}+\big\langle f(\rho_{\phi})u,v\big\rangle,\quad{\text{for all}}\;\ u,v \in H_0^1(\mathcal{D}),
\end{align}
where
\begin{align*}
	\big\langle f(\rho_{\phi})u,v\big\rangle:=\text{Re}\int_{\mathcal{D}}f(\rho_{\phi})u\overline{v}\;\text{d}\bm{x}.
\end{align*}

In particular, the linear part of $\mathcal{H}_{\phi}$, i.e., let $f(\rho_{\phi})=0$ in $\mathcal{H}_{\phi}$, is denoted by $\mathcal{H}_{0}$.
Under our assumptions, $\mathcal{H}_0$ is continuous and coercive. Especially,  $\|\cdot\|_{\mathcal{H}_0}$ is equivalent to the $H^1$-norm (cf. \cite{2010A}).

From an optimization perspective, the minimizer $\phi_g$ satisfies the first-order and second-order necessary conditions:{
\begin{align}
	\phi_g \in \mathcal{M}, \quad 
	E'(\phi_g) = \lambda_g \mathcal{I} \phi_g, \quad \text{and} \quad 
	\left\langle \big(E''(\phi_g) - \lambda_g \mathcal{I}\big) v, v \right\rangle \ge 0 
	\quad \text{for all } v \in T_{\phi_g} \mathcal{M},
\end{align}}
where $E'(\phi)=\mathcal{H}_{\phi}\phi:H^1_0(\mathcal{D})\to H^{-1}(\mathcal{D})$ denotes the {real Fr\'echet derivative} of $E(\phi)$, $\lambda_g=\left\langle\mathcal{H}_{\phi_g}\phi_g,\phi_g\right\rangle$ is an eigenvalue with eigenfunction $\phi_g$, $\mathcal{I}: L^2(\mathcal{D})\to L^2(\mathcal{D})\subset H^{-1}(\mathcal{D})$ denotes the canonical identification $\mathcal{I}v:=(v,\cdot)_{L^2}$, $E''$ denotes the second real Fr\'echet derivative. Given $\phi\in H_0^1(\mathcal{D})$, \(E''(\phi):H_0^1(\mathcal{D})\to H^{-1}(\mathcal{D})\) is computed as
\begin{align}
	\left\langle E''(\phi)u,v\right\rangle&=\left\langle \mathcal{H}_{\phi}u,v\right\rangle+\left\langle f'(\rho_{\phi})\big(|\phi|^2+\phi^2\overline{\cdot}\big)u,v\right\rangle,
\end{align}
where 
\begin{align*}
	\left\langle f'(\rho_{\phi})\big(|\phi|^2+\phi^2\overline{\cdot}\big)u,v\right\rangle:=\text{Re}\int_{\mathcal{D}}f'(\rho_{\phi})\big(|\phi|^2+\phi^2\overline{\cdot}\big)u\overline{v}\;\text{d}\bm{x}.
\end{align*}
Obviously, $E''(\phi)$ is symmetric. Notice that under the assumption of {\bf (A3)}, both $E'$ and $E''$ are well defined as bounded real linear operators on $H_0^1(\mathcal{D})$ (see {\bf Proposition~\ref{Property-of-E''}}).

In particular, for $\Omega=0$ and $f(s)=\eta s,\ \eta\ge0$, when the space of functions is restricted to real-valued functions, then the second-order sufficient condition is
satisfied at the minimizer:
\begin{align}\label{SOSC}
	\left\langle\big(E''(\phi_g)-\lambda_g\mathcal{I}\big)v,v\right\rangle\ge C\|v\|^2_{H^1}\quad \text{for all}\ v\in T_{\phi_g}\mathcal{M}.
\end{align}
In the \ref{Appendix-SOSC}, we explain why the second-order sufficient condition takes the above form in an infinite-dimensional Hilbert space. This condition implies the local uniqueness of the minimum. This is not true for $\Omega>0$, but we will see that it holds on a closed subspace of $T_{\phi_g}\mathcal{M}$.

Indeed, given a minimizer \(\phi_g\) and any angles \(\alpha, \beta \in [-\pi, \pi)\), $e^{i\alpha}\phi_g(A_{\beta}\bm{x})$ is also a minimizer with the same eigenvalue $\lambda_g$ by
\[
\|e^{i\alpha}\phi_g(A_{\beta}\bm{x})\|_{L^2} \equiv \|\phi_g\|_{L^2},\quad E(e^{i\alpha}\phi_g(A_{\beta}\bm{x})) \equiv E(\phi_g), 
\]
and
\[
\lambda_g=2E(\phi_g)+\int_{\mathcal{D}}\left(f(\rho_{\phi_g})|\phi_g|^2-F(\rho_{\phi_g})\right)\text{d}\bm{x}, 
\]
which may present additional challenges in the convergence analysis of common algorithms. 

In light of this, local uniqueness of minimizers can only be expected up to a constant phase and rotation factor. To account for the general lack of uniqueness by phase shifts and coordinate rotations, we define the phase shifts and coordinate rotations as linear group actions $I_{\alpha}^{\beta}$ for any function $\phi$ 
\begin{align}\label{I-alpha-beta}
	I_{\alpha}^{\beta}\phi:=e^{i\alpha}\phi(A_{\beta}\bm{x})\quad \text{for all}\ \alpha,\beta\in[-\pi, \pi).
\end{align}
We introduce the following set and energy level constructed from a minimizer $\phi_g$:
\begin{align}\label{mathcal_S}
	\mathcal{S}:=\Big\{\phi\in \mathcal{M}\big|\phi=I_{\alpha}^{\beta}\phi_{g},\ \alpha,\beta\in[-\pi, \pi)\Big\}\quad\text{and}\quad E_{\mathcal{S}}:=E(\phi),\quad {\text{for all}}\;\phi\in\mathcal{S}.
\end{align}
Noting that $\mathcal{S}$ is the orbit of the ground state under the group action $I_{\alpha}^{\beta}$, it is a finite-dimensional $C^1$ submanifold of $\mathcal{M}$. {The tangent space to this orbit at $\phi$ is spanned by the infinitesimal generators of the group action, namely phase and rotation:}
\[
T_{\phi}\mathcal{S} = \mathrm{span}\big\{ i\phi,\, i\mathcal{L}_z \phi \big\}.
\]
 In addition, $\dim \mathcal{S} = 1$ if $\phi$ is rotationally symmetric (i.e., $\phi = e^{ic\theta} \varphi(r,z)$), and $\dim \mathcal{S} = 2$ otherwise. In this work, we focus on the more challenging case $\dim \mathcal{S} = 2$, where the symmetry-induced degeneracy is maximal. To eliminate the influence of this degeneracy, we define the subspace
\begin{align}\label{Orth-Complement}
	N_{\phi}\mathcal{M} := \Big\{ v \in T_{\phi}\mathcal{M} \,\Big|\, (i\phi, v)_{L^2} = 0,\, (i\mathcal{L}_z\phi, v)_{L^2} = 0 \Big\},
\end{align}
which is orthogonal to the symmetry directions in $L^2$. This space will play a key role in the convergence analysis.

\begin{remark}
	If the quadratic and nonlinear parts of $E$ admit additional finite symmetries arising from linear group actions, the resulting critical submanifold $\mathcal{S}$ may have a higher dimension. However, the theoretical results established in this work still hold. Without loss of generality, we focus on the two-dimensional case, which is consistent with numerical experiments.
\end{remark}

The following proposition states that the second-order sufficient condition does not hold for  the case $\Omega>0$.

\begin{prop}\label{Prop1}
	Assume {\bf (A1)}-{\bf (A5)}. Then, for all $\phi\in\mathcal{S}$, it holds that {$T_{\phi}\mathcal{S}\subseteq\ker\left(E''(\phi)-\lambda_g\mathcal{I}\right)|_{T_{\phi}\mathcal{M}}$}, i.e., for all $v\in T_{\phi}\mathcal{M}$
	\begin{align*}
		\left\langle (E''(\phi)-\lambda_g\mathcal{I})i\phi,v\right\rangle=0\quad and\quad	\left\langle (E''(\phi)-\lambda_g\mathcal{I})i\mathcal{L}_z\phi,v\right\rangle=0.
	\end{align*}
\end{prop}
Additionally, it follows that {$T_{\phi}\mathcal{S} \subseteq \ker\left(E''(\phi) - \lambda_g \mathcal{I}\right)$}.
\begin{proof} 
	See details in \ref{Appendix_Pro}.
\end{proof}
Therefore, concerning the second-order sufficient condition, the best scenario we can expect is that $T_{\phi}\mathcal{S} = \ker\left(E''(\phi) - \lambda_g \mathcal{I}\right)|_{T_{\phi}\mathcal{M}}$ with $\phi\in\mathcal{S}$. When this condition is met, one calls $E$ a Morse-Bott functional on $\mathcal{S}$ (see \cite{1954Bott,2020L-S,2011An}), i.e., 
\begin{definition}\label{Morse-Bott-P}
	{The energy functional} $E$ is a Morse-Bott functional on {the ground state manifold} $\mathcal{S}$ if for all $\phi\in\mathcal{S}$,
	\begin{align*}
		\ker\left(E''(\phi) - \lambda_g \mathcal{I}\right)|_{T_{\phi}\mathcal{M}} =T_{\phi}\mathcal{S}=\text{\rm span}\big\{i\phi, i\mathcal{L}_z \phi\big\}.
	\end{align*}
\end{definition}
Generally, physical problems often exhibit symmetric structures, which result in degenerate local minimizers, making it challenging to determine the local convergence rate of algorithms. However, according to the following proposition, under the condition that the Morse-Bott property is satisfied, we can relax the requirement for non-degeneracy of local minimizers, thereby enabling us to derive the convergence rate of the algorithm similarly to the non-degenerate case.
\begin{prop}\label{Prop2}
	Assume {\bf (A1)}-{\bf (A5)} and let $E$ be a
	Morse-Bott functional on $\mathcal{S}$. Then, the operator $E''(\phi)-\lambda_g\mathcal{I}$ is coercive on $N_{\phi}\mathcal{M}$ when $\phi\in\mathcal{S}$, i.e., 
	\begin{align*}
		\left\langle(E''(\phi)-\lambda_g\mathcal{I})v,v\right\rangle\ge C\|v\|^2_{H^1}\quad\ for\ all\ v\in N_{\phi}\mathcal{M}.
	\end{align*}
\end{prop} 
\begin{proof} 
	See details in \ref{Appendix_Pro1}.
\end{proof}
In particular, for the numerical example to be provided later, we have verified that the Gross-Pitaevskii energy functional  indeed qualifies as a Morse-Bott functional.

Finally, for any $\phi\in H_0^1(\mathcal{D})$, the important properties of $E(\phi)$ and $E''(\phi)$ are summarized below. It will be frequently used in the subsequent analysis.
\begin{prop}\label{Property-of-E''}
	Given $\phi\in H_0^1(\mathcal{D})$ and for all $u, v\in H_0^1(\mathcal{D})$, the following conclusions hold:
	\begin{itemize}[label={}, labelsep=0pt, leftmargin=*]
		\item $(i)$ {The second variation} $E''(\phi)$ satisfies the invariance under the following linear group actions
		\begin{align*}
			\left\langle E''(I_{\alpha}^{\beta}\phi)I_{\alpha}^{\beta}v,I_{\alpha}^{\beta}v\right\rangle=\left\langle E''(\phi)v,v\right\rangle \quad for\ all\ \alpha,\beta\in[-\pi, \pi). 
		\end{align*}
		\item $(ii)$ {The second variation} $E''(\phi)$ is a continuous operator on $H_0^1(\mathcal{D})$, i.e.,
		\begin{align*}
			\left|\big\langle E''(\phi)u,v\big\rangle\right|\le C_{\phi}\|u\|_{H^1}\|v\|_{H^1}.
		\end{align*}
		\item $(iii)$ Given $\psi\in H_0^1(\mathcal{D})$, for {$p_0=6/(4-\theta)\in\left[\frac{3}{2},6\right)$}, the following inequality holds
		\begin{align*}
			\left|\left\langle \big(E''(\phi)-E''(\psi)\big)u,v\right\rangle\right|\le C_{\phi,\psi}\|u\|_{H^1}\|v\|_{H^1}\|\phi-\psi\|_{L^{p_0}}.
		\end{align*}
		\item $(iv)$ The following Lipschitz-type inequality holds
		\begin{align*}
			E(\phi+v)-E(\phi)\le \big\langle E'(\phi),v\big\rangle+\frac{1}{2}\big\langle E''(\phi)v,v\big\rangle+C_{\phi,v}\|v\|^3_{H^1}.
		\end{align*}
	\end{itemize}
\end{prop}
\begin{proof}
	The proofs of these conclusions are straightforward, and are provided in \ref{Appendix-A} for completeness.
\end{proof}

\section{Preconditioned Riemannian gradient methods}\label{Sec3}

In this section, we first review the Riemannian geometric structure of the problem, and then propose the generalized preconditioned Riemannian gradient methods.
\subsection{Riemannian Geometry structure of the problem}
{Firstly, we recall some basic concepts and formulas, namely, Riemannian metrics, orthogonal projections, Riemannian gradients, and retractions, as systematically developed for finite-dimensional Riemannian manifolds in \cite{N.B.appear}. 
In the infinite-dimensional Hilbert manifold setting of this work, these notions are defined by direct analogy; the underlying differential-geometric framework is standard and can be found, for instance, in \cite{1988Manifold, 1995Lang}.}

For the Riemannian manifold $\mathcal{M}$, the Riemannian metric $g_{\phi}(\cdot,\cdot):T_{\phi}\mathcal{M}\times T_{\phi}\mathcal{M}\to\mathbb{R}$ is the restriction of a complete inner product $(\cdot,\cdot)_X$ on $H_0^1(\mathcal{D})$ to $T_{\phi}\mathcal{M}$, i.e.,
\begin{align*}
	g_{\phi}(u,v):=(u,v)_{X}|_{{{T_{\phi}\mathcal{M}}}}\quad \text{for all}\ u,v\in T_{\phi}\mathcal{M}.
\end{align*}
The performance of gradient-based optimization methods in a Hilbert space depends on the metric, making the choice of $(\cdot,\cdot)_X$ critical (see \cite{2010Sobolev}). In this work, we propose utilizing a preconditioner $\mathcal{P}_{\phi}$, {defined for each $\phi \in H_0^1(\mathcal{D})$ as a symmetric and coercive real linear operator from $H_0^1(\mathcal{D})$ to $H^{-1}(\mathcal{D})$}, to define the inner product as described in \eqref{Bilinear-operator}. In the optimization theory, a well-known strategy to enhance the convergence rate of gradient-based methods is applying a suitable preconditioner. The preconditioner should approximate the Hessian operator of the objective functional as closely as possible. Consequently, $\mathcal{P}_{\phi}$ is assumed to meet the following condition:

\medskip

\noindent \quad {\bf (A6)} Given $\phi\in H_0^1(\mathcal{D})$ and for all $u,v\in H_0^1(\mathcal{D})$, $\mathcal{P}_{\phi}$ satisfies:

\smallskip

\begin{itemize}
	\item[$(i)$] {The preconditioner} $\mathcal{P}_{\phi}$ satisfies the invariance under the following linear group actions
	\begin{align*}
		\left\langle \mathcal{P}_{I_{\alpha}^{\beta}\phi}I_{\alpha}^{\beta}v,I_{\alpha}^{\beta}v\right\rangle=\left\langle \mathcal{P}_{\phi}v,v\right\rangle \quad for\ all\ \alpha,\beta\in[-\pi, \pi). 
	\end{align*}
	\item[$(ii)$] {The preconditioner} $\mathcal{P}_{\phi}$ is coercive and continuous on $H_0^1(\mathcal{D})$, i.e.,
	\begin{align*}
		\left\langle  \mathcal{P}_{\phi}v,v\right\rangle\ge C\|v\|^2_{H^1}\quad\text{and}\quad\left\langle \mathcal{P}_{\phi}u,v\right\rangle\le C_{\phi}\|u\|_{H^1}\|v\|_{H^1}.
	\end{align*}
	\item[$(iii)$] Given $\psi\in H_0^1(\mathcal{D})$, for a constant {$1\le p_1<6$}, the following inequality holds
	\begin{align*}
		\left|\left\langle \big(\mathcal{P}_{\phi}-\mathcal{P}_{\psi}\big)u,v\right\rangle\right|\le C_{\phi,\psi}\|u\|_{H^1}\|v\|_{H^1}\|\phi-\psi\|_{L^{p_1}}.
	\end{align*}
	\item[$(iv)$] {The preconditioner} $\mathcal{P}_{\phi}$ satisfies the following inequality:
	\begin{align*}
		\left\|\mathcal{P}_{\phi}^{-1}\left(E''(\phi)-\mathcal{P}_{\phi}\right)v\right\|_{H^1}\le C_{\phi}\|v\|_{L^{p_2}} \quad \text{for a constant}\ {1\le p_2<6}.
	\end{align*}
\end{itemize}

For the inner product {$(v, w)_{\mathcal{P}_{\phi}} := \langle \mathcal{P}_{\phi} v, w \rangle$ for all $v,w\in H_0^1(\mathcal{D})$}, the $\mathcal{P}_{\phi}$-orthogonal projection operator $\text{Proj}^{\mathcal{P}_{\phi}}_{\phi}: H_0^1(\Omega)\to T_{\phi}\mathcal{M}$ is defined as: for all $v\in T_{\phi}\mathcal{M}$
\begin{align}\label{Proj-P}
	\text{Proj}_{\phi}^{\mathcal{P}_{\phi}}(v)=v-\frac{(\phi,v)_{L^2}}{\big(\phi,\mathcal{P}_{\phi}^{-1}\mathcal{I}\phi\big)_{L^2}}\mathcal{P}_{\phi}^{-1}\mathcal{I}\phi.
\end{align}

Confined to the inner product $(\cdot,\cdot)_{\mathcal{P}_{\phi}}$ and the orthogonal projection $\text{Proj}_{\phi}^{\mathcal{P}_{\phi}}$, we give the formula of the Riemannian gradient $\nabla_{\mathcal{P}}^{\mathcal{R}} E(\phi)$ as follows:
\begin{align}\label{Riem-Grad}
	\nabla_{{\mathcal{P}}}^{\mathcal{R}} E(\phi)&=\text{Proj}_{\phi}^{\mathcal{P}_{\phi}}\nabla_{\mathcal{P}}E(\phi)=\mathcal{P}_{\phi}^{-1}\mathcal{H}_{\phi}\phi-\lambda_{\phi}\mathcal{P}_{\phi}^{-1}\mathcal{I}\phi,\quad
	\lambda_{\phi}=\frac{\big(\phi,\mathcal{P}_{\phi}^{-1}\mathcal{H}_{\phi}\phi\big)_{L^2}}{\big(\phi,\mathcal{P}_{\phi}^{-1}\mathcal{I}\phi\big)_{L^2}}.
\end{align}

Finally, according to the following normalized retraction $\mathfrak{R}_{\phi}(tv)$ \cite{N.B.appear}:
\begin{align}\label{Retraction}
	\mathfrak{R}_{\phi}(tv):=(\phi+tv)/\big\|\phi+tv\big\|_{L^2}\quad\text{for all}\ v\in T_{\phi}\mathcal{M},
\end{align}
the Riemannian gradient method forces all the iterates  to stay on $\mathcal{M}$.

\subsection{Algorithms}
With these preparations, we begin to give the algorithms. Provided with an inner product $(\cdot,\cdot)_{\mathcal{P}_{\phi}}$ (or preconditioner $\mathcal{P}_{\phi}$), an descent direction $d_n$, and the corresponding step size $\tau_n$, the preconditioned Riemannian gradient method can be formulated as an iterative sequence by \eqref{Riem-Grad} and \eqref{Retraction}:
\begin{align}\label{Riem-Opt-menthod}
	\phi^{n+1}=\mathfrak{R}_{\phi^n}(\tau_n d_n)=\frac{\phi^n+\tau_n d_n}{\quad\big\|\phi^n+\tau_n d_n\big\|_{L^2}}\quad\text{with}\quad d_n=-\nabla_{\mathcal{P}}^{\mathcal{R}}E(\phi^n).
\end{align}
Depending on the different choices of the preconditioner $\mathcal{P}_{\phi}$, descent direction $d_n$, and step size $\tau_n$, a variety of algorithms can be derived. In this paper, we do not specify the particular form of the preconditioner but provide a theoretical analysis for preconditioners that satisfy the general form outlined {\bf (A6)}. This theoretical analysis will be detailed in Section \ref{Sec4}.
Moreover, in practical computations, the step size $\tau_n$ is typically determined using either an exact line search or a backtracking line search method (see \cite{2017Efficient,2006Numerical}). Furthermore, if $E\bigl(\mathfrak{R}_{\phi^n}(\tau d_n)\bigr)$ is a rational function of $\tau$, then both backtracking and exact line search can be carried out efficiently (see \cite{2024Convergence}). In addition to these line search-based approaches, adaptive step size strategies, such as the Barzilai-Borwein step size method, which is a special nonmonotone gradient descent method \cite{1988BB}, are widely used to accelerate gradient-type optimization algorithms.

\begin{remark}\label{Proj-Sobo-Grad}
	Different preconditioners can lead to various types of algorithms, such as the $L^2$-projected gradient method \cite{2021Normalized} and a series of projected Sobolev gradient methods  \cite{2024On,2010A,2017Computation,2020Sobolev,2023The,2024Convergence,2010Tackling,2022Exp}. All these methods can be encompassed within the framework of \eqref{Riem-Opt-menthod}, with the respective preconditioners being \(\mathcal{P}_{\phi} = \mathcal{I},\ a\mathcal{I} - \frac{1}{2}\Delta\), \(a\mathcal{I} - \frac{1}{2}\Delta + V(\bm{x})\), \(a\mathcal{I} + \mathcal{H}_0\), and \(a\mathcal{I} + \mathcal{H}_{\phi}\) for all \(a \geq 0\). {In particular, the latter four are preconditioners that satisfy {\bf (A6)}. Beyond the preconditioned Riemannian gradient methods, there are other works that combine preconditioning techniques with the framework of Riemannian optimization \cite{2026JCAM,2024Riem,2017Efficient,2017Computation}}.
\end{remark}

Based on these assumptions, for the preconditioner $\mathcal{P}_{\phi}$, the Riemannian gradient $\nabla_{\mathcal{P}}^{\mathcal{R}}E(\phi)$, and the normalized retraction, we have the following properties.
\begin{prop}\label{Property-P}
	Assume {\rm \textbf{(A1)}}-{\rm \textbf{(A6)}}. Given $\phi\in H_0^1(\mathcal{D})$ and for all $u, v\in H_0^1(\mathcal{D})$ and $w\in H^{-1}(\mathcal{D})$, the following conclusions hold:
	\begin{itemize}[label={}, labelsep=0pt, leftmargin=*]
		\item $(i)$ If $E$ is a
		Morse-Bott functional on $\mathcal{S}$, then for all $\phi\in\mathcal{S}$, $\mathcal{P}_{\phi}$ and $E''(\phi)-\lambda_g\mathcal{I}$ satisfy the spectral equivalence, i.e.,
		\begin{align}\label{Mu-L}
			\inf_{v\in N_{\phi}\mathcal{M}\backslash\{0\}}\frac{\big\langle \big(E''(\phi)-\lambda_g\mathcal{I}\big)v,v\big\rangle}{\big\langle\mathcal{P}_{\phi}v,v\big\rangle}=\mu>0,\;\sup_{v\in T_{\phi}\mathcal{M}\backslash\{0\}}\frac{\big\langle \big(E''(\phi)-\lambda_g\mathcal{I}\big)v,v\big\rangle}{\big\langle\mathcal{P}_{\phi}v,v\big\rangle}=L<\infty.
		\end{align}
		\item $(ii)$ {The operator} $\mathcal{P}^{-1}_{\phi}\mathcal{H}_{\phi}:H_0^1(\mathcal{D})\to H_0^1(\mathcal{D})$ is a bounded linear operator, i.e., 
		\begin{align*}
			\big\|\mathcal{P}^{-1}_{\phi}\mathcal{H}_{\phi}v\big\|_{H^1}\le C_{\phi}\|v\|_{H^{1}}.
		\end{align*}
		Furthermore, $\mathcal{P}_{\phi}^{-1}(\mathcal{H}_{\phi}-\mathcal{P}_{\phi})$ satisfies the following estimate:
		\begin{align*}
			\big\|\mathcal{P}_{\phi}^{-1}(\mathcal{H}_{\phi}-\mathcal{P}_{\phi})v\big\|_{H^1}\le C_{\phi}\|v\|_{L^{p}}\quad\text{with}\quad {p=\max\{p_0,p_2\}\in[1,6)}.
		\end{align*}
		\item $(iii)$ Let $\phi\in\mathcal{M}$, there exists $\sigma$ such that for all $\psi\in\mathcal{B}_{\sigma}(\phi)$, the operator $\nabla^{\mathcal{R}}_{\mathcal{P}} E(\cdot):H_0^1(\mathcal{D})\to H_0^1(\mathcal{D})$ and the functional $\lambda_{(\cdot)}:H^1_0(\mathcal{D})\to \mathbb{R}$ are local Lipschitz continuous at $\phi$, i.e., 
		\begin{align*}
			\big\|\nabla^{\mathcal{R}}_{\mathcal{P}} E(\phi)-\nabla^{\mathcal{R}}_{\mathcal{P}} E(\psi)\big\|_{H^1}\le C_{\phi}\|\phi-\psi\|_{H^1}\quad and\quad \big|\lambda_{\phi}-\lambda_{\psi}\big|\le C_{\phi}\|\phi-\psi\|_{L^p},
		\end{align*}
		where $p=\max\{p_0,p_1,p_2,2\}\in[1,6)$.
		\item $(iv)$ Let $\phi\in\mathcal{M}$, for all $v\in T_{\phi}\mathcal{M}$, it holds that
		\begin{align*}
			\big|\mathfrak{R}_{\phi}(tv)-(\phi+tv)\big|\le\frac{1}{2}t^2\|v\|^2_{L^2}|\phi+tv|.
		\end{align*}
	\end{itemize}
\end{prop}
\begin{proof}
	See details in \ref{Appendix-B}.
\end{proof}

\section{Convergence analysis}\label{Sec4} 
In this section, all the analysis results are based on assumptions {\bf (A1)}-{\bf (A6)}, we first give the convergence results of the algorithm, and then prove these theoretical results. The results are as follows.
\subsection{Main results}
\begin{thm}\label{Global-Convergence}
	There exists a constant $\tau_{\max}>0$ that depends on the initial function $\phi^0$ such that for any $\tau_n\in (0,\tau_{\max})$, the iterations $\{\phi^n\}_{n\in\mathbb{N}}$
	generated by the P-RG have the following properties:
		\item $(i)$ It holds the sufficient descent condition, i.e., the energy is diminishing,
		{\rm\begin{align*} 
				E(\phi^{n+1})-E(\phi^n)\le
				-C_{\tau_n}\left\|d_n \right\|^2_{\mathcal{P}_{\phi^n}}\quad for\ all\ n\ge0
		\end{align*}}
		with a constant $C_{\tau_n}\ge\tau_n-\tau_n^2/\tau_{\max}$. So, the energy sequence $\{E(\phi^n)\}_{n\in\mathbb{N}}$ converges: 
		\[
		E_g:=\lim\limits_{n\to\infty}E(\phi^n).
		\]
		\item $(ii)$ There exists a subsequence $\{\phi^{n_j}\}_{j\in\mathbb{N}}$ and $\phi_g\in\mathcal{M}$ such that
		\begin{align*}
			\lim\limits_{j\to\infty}\|\phi^{n_j}-\phi_g\|_{H^1}=0.
		\end{align*}
		Furthermore, $\phi_g$ satisfies the first-order necessary condition, i.e.,
		{\rm\begin{align*}
				\lambda_{\phi_g}=\lim\limits_{j\to\infty}\lambda_{\phi^{n_j}}=\lambda_g=\big\langle\mathcal{H}_{\phi_g}\phi_g,\phi_g\big\rangle\quad and \quad  \mathcal{H}_{\phi_g}\phi_g=\lambda_g\mathcal{I}\phi_g.
		\end{align*}}
\end{thm}
The constant $\tau_{\max}$ is a global estimate, but as noted in {\bf Lemma \ref{C_t}}, larger steps maintaining sufficient descent are allowed around $\mathcal{S}$. 
In addition, if $E$ is a Morse-Bott functional on $\mathcal{S}$, we can weaken {\bf (A6)}-$(iii)$ to the standard Lipschitz continuity around $\phi_g$, i.e., for all $\phi, \psi \in\mathcal{B}_{\sigma}(\phi_g)$ and $u, v\in H_0^1(\mathcal{D})$,
\begin{align}\label{weaker-A6-iii}
	\left|\left\langle \big(\mathcal{P}_{\phi}-\mathcal{P}_{\psi}\big)u,v\right\rangle\right|\le C\|u\|_{H^1}\|v\|_{H^1}\|\phi-\psi\|_{H^1}.
\end{align}
This weaker condition still ensures the validity of {\bf Proposition \ref{Property-P}}, thereby guaranteeing the local convergence of the algorithm.

\begin{thm}\label{Local-Convergence}
	Let $E$ be a Morse-Bott functional on $\mathcal{S}$. Then, for every sufficiently small $\varepsilon>0$, there exist $\sigma>0$ and $\phi_g \in \mathcal{S}$ such that for all $\phi^0 \in \mathcal{B}_{\sigma}(\mathcal{S})$, the sequence $\{\phi^n\}_{n\in\mathbb{N}}$ generated by the P-RG has a locally linear convergence rate, i.e.,
	\begin{align*}
		\|\phi^n-\phi_g\|_{H^1}\le C_{\varepsilon}\|\phi^0-\phi_g\|_{H^1}\big(\sqrt{1-2C_{\tau}(\mu-\varepsilon)}\big)^n,\quad {\text{for all}}\;\ \tau\in(0,2/(L+\varepsilon)),
	\end{align*}
	where $C_{\varepsilon}$ is a constant depended on $\varepsilon$, $C_{\tau}=\tau-\frac{\tau^2}{2}(L+\varepsilon)$, $\mu$ and $L$ see \eqref{Mu-L}.
	Therefore, with the stepsize choice $\tau=1/(L+\varepsilon)$, this choice yields a precisely characterized convergence rate
	\begin{align}\label{Rate-1}
		\|\phi^n-\phi_g\|_{H^1}\le C_{\varepsilon}\|\phi^0-\phi_g\|_{H^1}\Bigg(\sqrt{1-\frac{\mu-\varepsilon}{L+\varepsilon}}\Bigg)^n.
	\end{align}
	
\end{thm}
Examining the local convergence rates, it becomes evident that the convergence rate improves as $\mu$ approaches $L$. Notably, a superlinear convergence rate (see \cite{2006Numerical}) is attainable when $\mu = L$. Furthermore, according to {\bf Remark~\ref{Proj-Sobo-Grad}}, this observation clarifies that the essence of acceleration in projected Sobolev gradient methods is fundamentally akin to preconditioning: both achieve faster convergence by improving the condition number of the problem. {We emphasize that the contraction factor of the form $\rho=\sqrt{1-(\mu-\varepsilon)/(L+\varepsilon)}$ is the standard rate associated with first-order methods under a Polyak-\L ojasiewicz-type descent
inequality, which we will use in the proof. While the Morse–Bott (second-order) structure implies the Polyak-\L ojasiewicz inequality and thus guarantees local linear convergence, the standard rate bound may be suboptimal. In particular, problems satisfying stronger second-order sufficient conditions may in principle admit faster (local) convergence rates, although projected Sobolev gradient methods (being first-order schemes) remain governed by the condition number through the ratio $\mu/L$ and therefore exhibit the Polyak-\L ojasiewicz-type rate given above.}

According to \eqref{Mu-L}, the operator 
\begin{align*}
	\mathcal{P}_{\phi} = E''(\phi) - \widetilde{\lambda}_{\phi}\mathcal{I},\quad\text{with}\quad \widetilde{\lambda}_{\phi}=\left\langle\mathcal{H}_{\phi}\phi,\phi\right\rangle,
\end{align*}
represents a theoretically optimal local preconditioner. However, it is not necessarily coercive even at $\phi_g$. Thus, a natural idea is to choose a {quasi}-optimal local preconditioner:
\begin{align}\label{Varep-P}
	\mathcal{P}_{\phi}=E''(\phi)-\big(\widetilde{\lambda}_{\phi}-\sigma_0\big)\mathcal{I}
\end{align}
around $\phi_g$, where {$\sigma_0>0$ is a small regularization parameter introduced solely to ensure the coercivity of $\mathcal{P}_{\phi}$. In theory, one chooses $\sigma_0>0$ as small as possible, yet large enough to ensure coercivity of the preconditioned operator on $H_0^1(\mathcal{D})$.} 
Since the preconditioner \eqref{Varep-P} does not satisfy {\bf (A6)}-$(iii)$, its global convergence cannot be guaranteed in general. {In fact, although $E''(\phi)$ itself satisfies {\bf (A6)}-$(iii)$, the term $\widetilde{\lambda}_{\phi} \mathcal{I}$ lacks the necessary regularity, as $\widetilde{\lambda}_{\phi}$ depends on $\phi$ only in a Lipschitz continuous manner; consequently, the full preconditioner $\mathcal{P}_{\phi}$ fails to satisfy this condition.} However, it can be shown that the preconditioner \eqref{Varep-P} is Lipschitz continuous with respect to $\phi$ based on the Lipschitz continuity of $E''(\phi)$ and $\widetilde{\lambda}_{\phi}$. Therefore, the {local} convergence of the P-RG can still be guaranteed for the preconditioner \eqref{Varep-P}. 

The following theorem demonstrates that the P-RG exhibit the best rate of local convergence when the preconditioner is chosen in the specified form.
\begin{thm}\label{Var-P-convergence}
	Let $E$ be a Morse-Bott functional on $\mathcal{S}$. Then, for every sufficiently small $\varepsilon > 0$, there exist $\sigma>0$ and $\phi_g \in \mathcal{S}$ such that for all $\phi^0 \in \mathcal{B}_{\sigma}(\mathcal{S})$, the sequence $\{\phi^n\}_{n\in\mathbb{N}}$ generated by the P-RG with the quasi-optimal local preconditioner \eqref{Varep-P} yields another locally linear convergence rate, i.e., $\text{for all}\ \tau\in(0,2/(L+\varepsilon))$
	\begin{align*}
		\|\phi^n-\phi_g\|_{H^1}\le C_{\varepsilon}\|\phi^0-\phi_g\|_{H^1}\left(\max\left\{|1-\tau\mu|,|1-\tau L |\right\}+\varepsilon\right)^n.
	\end{align*}
	Hence, when $\tau=2/(L+\mu)$, we have the well-known best local linear convergence rate for $\big\{\phi^n\big\}_{n\in\mathbb{N}}$
	{\rm\begin{align}
			\|\phi^n-\phi_g\|_{H^1}\le C_{\varepsilon}\|\phi^0-\phi_g\|_{H^1}\left(\frac{L-\mu}{L+\mu}+\varepsilon\right)^n.
	\end{align}}
\end{thm}
It is observed that the rate of convergence described in the {\bf Theorem \ref{Var-P-convergence}} matches the optimal convergence rate achieved by the gradient descent method for solving unconstrained, strongly convex optimization problems \cite{2006Numerical}. This observation suggests that, when non-uniqueness stems exclusively from specific symmetries, the problem retains properties analogous to those of a strongly convex optimization problem. Indeed, this is subtly implied by the definition of the Morse-Bott property, and our theoretical findings rigorously substantiate this assertion. Furthermore, in this context, we have $\mu = (\lambda_3 - \lambda_g)\big/(\lambda_3 - \lambda_g + \sigma_0)$ and $L = 1$. See \ref{Appendix-Opt-Pre} for the computation of $\mu$ and $L$, and \eqref{lambda1} for the definition of $\lambda_3$. {Therefore, in the idealized continuous setting, where the preconditioned linear system is solved exactly in the infinite-dimensional Hilbert space, without any concern for computational cost or numerical stability, the asymptotic convergence rate 
	\[
	\rho = \frac{L-\mu}{L+\mu} +\varepsilon= \frac{\sigma_0}{2(\lambda_3 - \lambda_g)} + \varepsilon
	\] 
	shows that a smaller $\sigma_0 > 0$ yields faster local convergence, but only when the iterate is sufficiently close to the ground state $\phi_g$, a requirement that becomes stricter as $\sigma_0$ decreases, since coercivity of $E''(\phi) - (\widetilde{\lambda}_{\phi} - \sigma_0)\mathcal{I}$ is guaranteed only in a $\sigma_0$-dependent neighborhood of $\phi_g$. In practice, however, this theoretical advantage comes at a price: smaller $\sigma_0$ typically leads to a more ill-conditioned preconditioned system, thereby increasing the computational cost of solving it, as noted in the following \textbf{Remark~\ref{Re4.1}}.}

\begin{remark}\label{Re4.1}{
In our current framework, the P-RG algorithm solves, in the full space $H_0^1(\mathcal{D})$, linear systems of the form: for given $w \in H^{-1}(\mathcal{D})$, find $u \in H_0^1(\mathcal{D})$ such that
	\[
	\big\langle (E''(\phi) - (\widetilde{\lambda}_{\phi} - \sigma_0)\mathcal{I}) u, v \big\rangle = \langle w, v \rangle \quad \text{for all } v \in H_0^1(\mathcal{D}).
	\]
	Since the operator $E''(\phi_g) - \lambda_g \mathcal{I}$ lacks coercivity due to the $U(1) \times SO(2)$ symmetry (corresponding to phase and spatial rotations), the system becomes increasingly ill-conditioned as $\sigma_0 \to 0^+$, and singular when $\sigma_0 = 0$. Consequently, in practical numerical simulations, the choice of $\sigma_0$ entails a delicate trade-off between the theoretical convergence rate $\rho$ and the computational cost of solving the preconditioned system. This balance cannot be characterized by a universal rule, as it depends on a multitude of intertwined factors, such as physical parameters $\Omega$ and $\eta$, the spatial discretization scheme, and the linear solver employed.}
	{In addition,
		a natural question arises: can one avoid the ill-conditioning associated with $\sigma_0 \to 0^+$ (or $\sigma_0=0$) altogether, thereby allowing $\sigma_0 \to 0^+$ (or $\sigma_0=0$) to be used safely and yielding local superlinear or even quadratic convergence?
	In principle, the system can be rendered well-posed by restricting it to a suitable closed subspace. Specifically, define
	\[
	\mathcal{T}_\phi := \left\{ u \in H_0^1(\mathcal{D}) \,\middle|\, (u, i\phi)_{L^2} = 0,\; (u, i\mathcal{L}_z \phi)_{L^2} = 0 \right\}.
	\]
	Then, for any $w \in \mathcal{T}_\phi^*$, we solve: find $u \in \mathcal{T}_\phi$ such that
	\[
	\big\langle (E''(\phi) - (\widetilde{\lambda}_{\phi} - \sigma_0)\mathcal{I}) u, v \big\rangle = \langle w, v \rangle \quad \text{for all } v \in \mathcal{T}_\phi.
	\]
	When $\sigma_0 = 0$ and $\phi$ is sufficiently close to the ground state $\phi_g$, the operator is coercive on $\mathcal{T}_\phi$, ensuring well-posedness. This construction is well-motivated by the underlying geometric structure: the degeneracy stems from the invariance of the energy functional under the action of the symmetry group $ U(1) \times SO(2)$. This group induces an equivalence relation on $H_0^1(\mathcal{D}) \setminus \{0\}$ via
	\[
	u \sim v \quad \Longleftrightarrow \quad v(\bm{x}) = e^{i\alpha} u(A_\beta \bm{x}) \quad \text{for } \alpha,\,\beta \in [-\pi,\pi).
	\]
 	The associated quotient space is then
	\[
	[H_0^1(\mathcal{D})] := \big( H_0^1(\mathcal{D}) \setminus \{0\} \big) / (U(1) \times SO(2)),
	\]
	and its tangent space at the equivalence class $[\phi]$ is precisely the subspace $\mathcal{T}_\phi$ defined above.}
\end{remark}
	{We note that the discussion in Remark~\ref{Re4.1} regarding the quotient manifold provides a feasibility-oriented justification rather than a rigorous convergence analysis. Since this work focuses on local linear convergence within a first-order Riemannian optimization framework, such considerations lie beyond the scope of the present paper. Moreover, although methods formulated on the quotient manifold are theoretically sound, their practical implementation may incur additional computational overhead at each iteration. In summary, there may be no generally applicable rule for choosing $\sigma_0$ that guarantees consistently high computational efficiency in practice. Nevertheless, the theoretical results presented in this work, together with the quasi-optimal preconditioner, offer valuable insight into the origin of computational difficulties and thereby guide the algorithmic adjustments when challenges arise.}

Finally, we give the following corollary.
\begin{corollary}\label{E-eq-Phi}
	Let $E$ be a Morse-Bott functional on $\mathcal{S}$. For the sequence $\{\phi^n\}_{n\in\mathbb{N}}$ generated by the P-RG and its corresponding limit point $\phi_g$, if $\phi^0\in\mathcal{B}_{\sigma}(\mathcal{S})$, then the energy difference and the wave function difference are equivalent, i.e.,
	\begin{align*}
		\sqrt{E^{n}-E^{n+1}}\le \sqrt{E^{n}-E(\phi_g)}\lesssim\|\phi^{n}-\phi_g\|_{H^1}\lesssim \sqrt{E^{n}-E(\phi_g)}\lesssim \sqrt{E^{n}-E^{n+1}},
	\end{align*}
	where $E^n:=E(\phi^n)$.
\end{corollary}
This corollary shows that to terminate the iteration, the frequently used conditions via wave function error $|\phi^{n+1}-\phi^n|$ (see \cite{2003Computing}) and via energy error $|E^{n+1}-E^n|$ (see \cite{2017Efficient}) are equivalent. {The proof is provided later in Section~4.3.}

\subsection{Technical lemmas}
Before presenting the proof, we introduce several key lemmas that will be instrumental in establishing various aspects of our results. Specifically: {\bf Lemma \ref{Transfer}-\ref{Nonlinear-Iteration}} will be employed to demonstrate the local convergence rates, i.e., {\bf Theorem \ref{Local-Convergence}} and {\bf Theorem \ref{Var-P-convergence}}.

In order to obtain accurate local convergence rates, we establish some local estimates. Firstly, we introduce the following lemma.
\begin{lemma}\label{Transfer}
	Let $E$ be a
	Morse-Bott functional on $\mathcal{S}$. For any $\phi \in \mathcal{M}$ and $\phi_g\in\mathcal{S}$, there exists $\phi^*_g \in \mathcal{S}$ such that the following orthogonality conditions hold:
	{\rm
		\begin{align*}
			(\phi-\phi_g^*,i\phi_g^*)_{L^2}=0\quad and\quad(\phi-\phi_g^*,i\mathcal{L}_z\phi_g^*)_{L^2}=0.
		\end{align*}
	}
	Furthermore, $\|\phi-\phi_g^*\|_{H^1}\le C_{\phi} \|\phi-\phi_g\|_{H^1}.$
\end{lemma}
\begin{proof}
	We construct a functional as follows
	\begin{align}\label{f-phi-u}
		\mathcal{F}_{\phi}(u)&:=\frac{1}{2}\|\phi-u\|^2_{\mathcal{H}_0}+\frac{U}{2}\|\phi-u\|^2_{L^2}\\
		\nonumber&\qquad\qquad+\underbrace{\frac{1}{2}\left\langle f(\rho_{u})(\phi-u),\phi-u\right\rangle+\frac{1}{2}\int_{\mathcal{D}}\int_{|u|^2}^{|\phi|^2}f'(s)\big(|\phi|^2-s\big)\;\text{d}s\text{d}\bm{x}}_{=:I},
	\end{align}
	where $U$ is an undetermined constant. According to {\bf (A3)}, we have
	\begin{align*}
		|I|&\le C\left\langle \left(1+|u|^{1+\theta}\right)(\phi-u),\phi-u\right\rangle+C\int_{\mathcal{D}}\int_{|u|^2}^{|\phi|^2}s^{(\theta-1)/2}\big(|\phi|^2-|u|^2\big)\;\text{d}s\text{d}\bm{x}\\
		&\le C\left\langle \left(1+|u|^{1+\theta}\right)(\phi-u),\phi-u\right\rangle+C\left\langle \left(|\phi|+|u|\right)^{1+\theta}(\phi-u),\phi-u\right\rangle\\
		&\le C\left\langle \left(1+\left(|\phi|+|u|\right)^{1+\theta}\right)(\phi-u),\phi-u\right\rangle.
	\end{align*}
	Similar to \eqref{Positive1}, we further obtain
	\begin{align*}
		|I|&\le C\|\phi-u\|^2_{L^2}+C\left(\|\phi\|^{1+\theta}_{L^6}+\|u\|^{1+\theta}_{L^6}\right)\|\phi-u\|^2_{L^{p}}\\
		&\le C_{\phi,u}\left(\varepsilon^{-\frac{(1-2/p)d}{2-(1-2/p)d}}\|\phi-u\|^{2}_{L^{2}}+\varepsilon\|\phi-u\|^{2}_{H^1}\right),
	\end{align*}
	where $p=12/(5-\theta)\in[\frac{12}{5},6)$.
	Let $u\in\mathcal{S}$, combined with the coerciveness and continuity of $\mathcal{H}_0$, we can choose a sufficiently small constans $\varepsilon$ and a sufficiently large constant $U=C_{\phi}\neq-\lambda_{g}$ positively correlated with $\|\phi\|_{H^1}$ such that
	\begin{align}\label{n-e}
		C\|\phi-u\|^2_{H^1}\le \mathcal{F}_{\phi}(u)\le C_{\phi}\|\phi-u\|^2_{H^1}.
	\end{align}
	Now we consider the global optimization of $\mathcal{F}_{\phi}(u)$ on the manifold $\mathcal{S}$:
	\begin{align*}
		\phi_g^*:=\argmin\limits_{u\in\mathcal{S}}\mathcal{F}_{\phi}(u).
	\end{align*}
	Noting that $\mathcal{S}$ is a finite dimensional $C^1$ submanifold and $\mathcal{F}_{\phi}$ is a {continuously} differentiable function with respect to $u$, then the solution $\phi_g^{*}$ to the above optimization problem exists and it satisfies the first order necessary condition, i.e., let $\gamma_1(t)=e^{it}\phi_g^{*},\ \gamma_2(t)=\phi_g^{*}(A_{t}\bm{x})$, for $i=1$ or $2$, 
	\begin{align*}
		\frac{\text{d}\mathcal{F}_{\phi}(\gamma_i(t))}{\text{d}t}\Bigg|_{t=0}=0.
	\end{align*}
	Calculating directly yields the following result
	\begin{align*}
		&\frac{\text{d}\mathcal{F}_{\phi}(\gamma_i(t))}{\text{d}t}\Bigg|_{t=0}=-\left\langle\big(\mathcal{H}_{\phi_g^{*}}+C_{\phi}\big)(\phi-\phi_g^{*}),\gamma^{\prime}_i(0)\right\rangle+\left\langle f'(\rho_{\phi_g^*})|\phi-\phi_g^{*}|^2\phi_g^{*},\gamma^{\prime}_i(0)\right\rangle\\
		&\qquad\qquad\qquad\qquad\qquad\qquad\qquad\qquad\qquad\qquad\quad-\left\langle f'(\rho_{\phi_g^*})\big(|\phi|^2-|\phi_g^{*}|^2\big)\phi_g^{*},\gamma^{\prime}_i(0)\right\rangle\\
		&=-\left\langle\big(\mathcal{H}_{\phi_g^{*}}+C_{\phi}\big)(\phi-\phi_g^{*}),\gamma^{\prime}_i(0)\right\rangle+\left\langle f'(\rho_{\phi_g^*})(2|\phi_g^{*}|^2-\phi\overline{\phi_g^{*}}-\phi_g^{*}\overline{\phi})\phi_g^{*},\gamma^{\prime}_i(0)\right\rangle\\
		&=-\left\langle\big(\mathcal{H}_{\phi_g^{*}}+C_{\phi}\big)(\phi-\phi_g^{*}),\gamma^{\prime}_i(0)\right\rangle-\left\langle f'(\rho_{\phi_g^*})\big(|\phi_g^{*}|^2+(\phi_g^{*})^2\overline{\cdot}\big)(\phi-\phi_g^{*}),\gamma^{\prime}_i(0)\right\rangle\\
		&=-\left\langle \big(E''(\phi_g^{*})+C_{\phi}\big)(\phi-\phi_g^{*}),\gamma^{\prime}_i(0)\right\rangle.
	\end{align*}
	Thus, we derive
	\begin{align*}
		\left\langle \big(E''(\phi_g^{*})+C_{\phi}\big)(\phi-\phi_g^*),i\phi_g^*\right\rangle&=\left(\lambda_g+C_{\phi}\right)( \phi-\phi_g^*,i\phi_g^*)_{L^2}=0,\\
		\left\langle \big(E''(\phi_g^{*})+C_{\phi}\big)(\phi-\phi_g^*),i\mathcal{L}_z\phi_g^*\right\rangle&=\left(\lambda_g+C_{\phi}\right)( \phi-\phi_g^*,i\mathcal{L}_z\phi_g^*)_{L^2}=0.
	\end{align*}
	In addition, since \(\phi_g^*\) corresponds to the global minimum of \(\mathcal{F}_{\phi}\) and according to \eqref{n-e}, we have
	\begin{align*}
		C\big\|\phi-\phi_g^*\big\|^2_{H^1}\le \mathcal{F}_{\phi}(\phi_g^*)\le \mathcal{F}_{\phi}(\phi_g)\le C_{\phi}\|\phi-\phi_g\|^2_{H^1}.
	\end{align*}
	This completes the proof.
\end{proof}

This lemma shows that $E$ satisfies the Polyak-\L ojasiewicz inequality around $\phi_g$.
\begin{lemma}\label{Polyak-Loj}
	Let $E$ be a Morse-Bott functional on $\mathcal{S}$. For any $\phi_g\in\mathcal{S}$, and for every sufficiently small $\varepsilon>0$, there exists $\sigma>0$ such that for any $\phi \in \mathcal{B}_{\sigma}(\phi_g)$, the following Polyak-\L ojasiewicz inequality holds:
	\[
	E(\phi) - E(\phi_g)\leq \frac{1}{2(\mu - \varepsilon)} \left\| \nabla^{\mathcal{R}}_{\mathcal{P}} E(\phi)\right\|^2_{\mathcal{P}_{\phi}}.
	\]
\end{lemma}
\begin{proof}
	According to $E(\phi_g^*)=E(\phi_g)$ and Taylor's formula at $\phi$, we have
	\begin{align}
		\nonumber E(\phi)-&E(\phi_g)=E(\phi)-E(\phi_g^*)\\
		\nonumber=&\left\langle E'(\phi),\phi-\phi_g^*\right\rangle-\frac{1}{2}\left\langle E''(\phi)(\phi-\phi_g^*),\phi-\phi_g^*\right\rangle+o\left(\|\phi-\phi_g^*\|^2_{H^1}\right)\\
		=&\left( \nabla^{\mathcal{R}}_{\mathcal{P}} E(\phi),\phi-\phi_g^*\right)_{\mathcal{P}_{\phi}}\hspace{-0.32cm}-\frac{1}{2}\left\langle \big(E''(\phi)-\lambda_{\phi}\mathcal{I}\big)(\phi-\phi_g^*),\phi-\phi_g^*\right\rangle+o\left(\|\phi-\phi_g^*\|^2_{H^1}\right)\label{P-L0}.
	\end{align}
	Note that
	\begin{align}
		\nonumber\phi-\phi_g^*&=-\phi_g^*+(\phi_g^*,\phi)_{L^2}\phi+(\phi-\phi_g^*,\phi)_{L^2}\phi\\
		\nonumber&=\phi-\phi_g^*+(\phi_g^*-\phi,\phi)_{L^2}\phi+\frac{1}{2}\left(\|\phi\|_{L^2}^2-\|\phi_g^*\|_{L^2}^2+\|\phi-\phi_g^*\|_{L^2}\right)\phi\\
		&=\text{Proj}^{L^2}_{\phi}(\phi-\phi_g^*)+\frac{1}{2}\|\phi-\phi_g^*\|^2_{L^2}\phi,\label{P-L1-1}\\
		\nonumber\phi-\phi_g^*&=\phi-(\phi,\phi_g^*)_{L^2}\phi_g^*-(\phi_g^*-\phi,\phi_g^*)_{L^2}\phi_g^*\\
		&=\text{Proj}^{L^2}_{\phi_g^*}(\phi-\phi_g^*)-\frac{1}{2}\|\phi-\phi_g^*\|^2_{L^2}\phi_g^*,\label{P-L1}\\
		\Longrightarrow\;&\text{Proj}^{L^2}_{\phi}(\phi-\phi_g^*)=\text{Proj}^{L^2}_{\phi_g^*}(\phi-\phi_g^*)-\frac{1}{2}\|\phi-\phi_g^*\|^2_{L^2}\phi-\frac{1}{2}\|\phi-\phi_g^*\|^2_{L^2}\phi_g^*,\label{P-L1-2}
	\end{align}
	where $\text{Proj}^{L^2}_{\phi_g^*}(\phi-\phi_g^*)\in N_{\phi_g^*}\mathcal{M}$.
	Substituting \eqref{P-L1-1} into \eqref{P-L0}, and using {\bf Proposition \ref{Property-of-E''}}-$(ii)$ and {\bf Proposition \ref{Property-P}}-$(iii)$, we derive
	\begin{align*}
		E(\phi)-E(\phi_g)&=\left( \nabla^{\mathcal{R}}_{\mathcal{P}} E(\phi),\text{Proj}^{L^2}_{\phi}(\phi-\phi_g^*)\right)_{\mathcal{P}_{\phi}}\\
		\nonumber-&\frac{1}{2}\left\langle \big(E''(\phi)-\lambda_{\phi}\mathcal{I}\big)\text{Proj}^{L^2}_{\phi}(\phi-\phi_g^*),\text{Proj}^{L^2}_{\phi}(\phi-\phi_g^*)\right\rangle+o\left(\|\phi-\phi_g^*\|^2_{H^1}\right).
	\end{align*}
	Plugging \eqref{P-L1-2} into the above identity, we get
	
	\begin{align}\label{P-L2}
		\nonumber E(\phi)-E(\phi_g)&=\left(\nabla^{\mathcal{R}}_{\mathcal{P}} E(\phi),\text{Proj}^{L^2}_{\phi_g^*}(\phi-\phi_g^*)\right)_{\mathcal{P}_{\phi}}\\
		-&\frac{1}{2}\left\langle \big(E''(\phi)-\lambda_{\phi}\mathcal{I}\big)\text{Proj}^{L^2}_{\phi_g^*}(\phi-\phi_g^*),\text{Proj}^{L^2}_{\phi_g^*}(\phi-\phi_g^*)\right\rangle+o\left(\|\phi-\phi_g^*\|^2_{H^1}\right).
	\end{align}
	Based on {\bf Proposition \ref{Property-of-E''}}-$(iii)$, {\bf Proposition \ref{Property-P}}-$(iii)$, and {\bf (A6)}-$(iii)$, the following estimations hold
	\begin{align*}
		\left\langle \big(E''(\phi)-E''(\phi_g^*)\big)\text{Proj}^{L^2}_{\phi_g^*}(\phi-\phi_g^*),\text{Proj}^{L^2}_{\phi_g^*}(\phi-\phi_g^*)\right\rangle&=o\left(\|\phi-\phi_g^*\|^2_{H^1}\right),\\
		\left\langle \big(\lambda_{\phi_g^*}\mathcal{I}-\lambda_{\phi}\mathcal{I}\big)\text{Proj}^{L^2}_{\phi_g^*}(\phi-\phi_g^*),\text{Proj}^{L^2}_{\phi_g^*}(\phi-\phi_g^*)\right\rangle&=o\left(\|\phi-\phi_g^*\|^2_{H^1}\right),\\
		\left\langle\big(\mathcal{P}_{\phi}-\mathcal{P}_{\phi_g^*}\big)\text{Proj}^{L^2}_{\phi_g^*}(\phi-\phi_g^*),\text{Proj}^{L^2}_{\phi_g^*}(\phi-\phi_g^*)\right\rangle&=o\left(\|\phi-\phi_g^*\|^2_{H^1}\right).
	\end{align*}
	According to {\bf Proposition \ref{Property-P}}-$(i)$ and $\text{Proj}^{L^2}_{\phi_g^*}(\phi-\phi_g^*)\in N_{\phi_g^*}\mathcal{M}$, the following lower bound estimate holds
	\begin{align*}
		\frac{\left\langle\big( E''(\phi_g^*)-\lambda_{g}\mathcal{I}\big)\text{Proj}^{L^2}_{\phi_g^*}(\phi-\phi_g^*),\text{Proj}^{L^2}_{\phi_g^*}(\phi-\phi_g^*)\right\rangle}{\left\langle\mathcal{P}_{\phi_g^*}\text{Proj}^{L^2}_{\phi_g^*}(\phi-\phi_g^*),\text{Proj}^{L^2}_{\phi_g^*}(\phi-\phi_g^*)\right\rangle}\ge\mu.
	\end{align*}
	In summary, the estimate we want is derived
	\begin{align*}
		-\frac{1}{2}&\left\langle \big(E''(\phi)-\lambda_{\phi}\mathcal{I}\big)\text{Proj}^{L^2}_{\phi_g^*}(\phi-\phi_g^*),\text{Proj}^{L^2}_{\phi_g^*}(\phi-\phi_g^*)\right\rangle\\
		&\qquad\qquad\qquad\le-\frac{\mu}{2}\left\langle\mathcal{P}_{\phi}\text{Proj}^{L^2}_{\phi_g^*}(\phi-\phi_g^*),\text{Proj}^{L^2}_{\phi_g^*}(\phi-\phi_g^*)\right\rangle
		+o\left(\|\phi-\phi_g^*\|^2_{H^1}\right).
	\end{align*}
	Combining the above inequality with \eqref{P-L2}, we get
	\begin{align}\label{P-L3}
		\nonumber E(\phi)-E(\phi_g)&\le\left(\nabla^{\mathcal{R}}_{\mathcal{P}} E(\phi),\text{Proj}^{L^2}_{\phi_g^*}(\phi-\phi_g^*)\right)_{\mathcal{P}_{\phi}}\\
		&-\frac{\mu}{2}\left(\text{Proj}^{L^2}_{\phi_g^*}(\phi-\phi_g^*),\text{Proj}^{L^2}_{\phi_g^*}(\phi-\phi_g^*)\right)_{\mathcal{P}_{\phi}}+o\left(\|\phi-\phi_g^*\|^2_{H^1}\right).
	\end{align}
	By {\bf Lemma \ref{Transfer}} and {\bf (A6)}-$(ii)$, we know that 
	\begin{align}\label{phi_g^*-phi_g}
		\|\phi-\phi_g^*\|_{{H^1}}\le C\|\phi-\phi_g^*\|_{\mathcal{P}_{\phi}}\le C_{\phi}\|\phi-\phi_g^*\|_{H^1}\le C_{\phi} \|\phi-\phi_g\|_{H^1}. 
	\end{align}
	Recalling \eqref{P-L1}, then for all sufficiently small $\varepsilon$, there exists $\sigma$ such that for any $\phi \in \mathcal{B}_{\sigma}(\phi_g)$, we have
	\begin{align}\label{P-L4}
		\left|o\left(\|\phi-\phi_g^*\|^2_{H^1}\right)\right|\le\frac{\varepsilon}{2}\left\|\text{Proj}^{L^2}_{\phi_g^*}(\phi-\phi_g^*)\right\|^2_{\mathcal{P}_{\phi}}.
	\end{align}
	Then, by \eqref{P-L3}, the Polyak-\L ojasiewicz inequality is deduced as follows
	\begin{align*}
		&E(\phi)-E(\phi_g)\\
		\le&\left(\nabla^{\mathcal{R}}_{\mathcal{P}} E(\phi),\text{Proj}^{L^2}_{\phi_g^*}(\phi-\phi_g^*)\right)_{\mathcal{P}_{\phi}}-\frac{\mu-\varepsilon}{2}\left(\text{Proj}^{L^2}_{\phi_g^*}(\phi-\phi_g^*),\text{Proj}^{L^2}_{\phi_g^*}(\phi-\phi_g^*)\right)_{\mathcal{P}_{\phi}}\\
		\le&\sup\limits_{v\in H_0^1(\mathcal{D})}\Bigg(\left( \nabla^{\mathcal{R}}_{\mathcal{P}} E(\phi),v\right)_{\mathcal{P}_{\phi}}-\frac{\mu-\varepsilon}{2}(v,v)_{\mathcal{P}_{\phi}}\Bigg)
		=\frac{1}{2(\mu-\varepsilon)}\left\|\nabla^{\mathcal{R}}_{\mathcal{P}} E(\phi)\right\|^2_{\mathcal{P}_{\phi}}.
	\end{align*}
\end{proof}

In order to obtain the exact rate of local convergence, we need to derive the exact local energy dissipation as follows. For brevity, we denote $\widetilde{\phi}^{n+1}$ by $\widetilde{\phi}^{n+1}=\phi^n+\tau_nd_n$.
\begin{lemma}\label{C_t}
	Let $E$ be a
	Morse-Bott functional on $\mathcal{S}$. For any $\phi_g\in\mathcal{S}$, and for every sufficiently small $\varepsilon>0$, there exists $\sigma>0$ such that for any {$\phi^n \in \mathcal{B}_{\sigma}(\phi_g)$}, the local energy dissipation is estimated by:
	\begin{align*}
		E(\phi^{n+1})-E(\phi^n)\le-C_\tau\|d_n\|^2_{\mathcal{P}_{\phi^n}}\quad for\ all\  \tau\in\big(0,2/(L+\varepsilon)\big),
	\end{align*} 
	where 
	$C_{\tau}=\tau-\frac{\tau^2}{2}(L+\varepsilon)$.
	In particular, the optimal upper bound is obtained when $\tau=1/(L+\varepsilon)$, i.e.,
	\begin{align*}
		E(\phi^{n+1})-E(\phi^n)\le-\frac{1}{2(L+\varepsilon)}\|d_n\|^2_{\mathcal{P}_{\phi^n}}.
	\end{align*} 
\end{lemma}
\begin{proof}
	Using {\bf Proposition \ref{Property-P}}-$(iii)$, the estimates of $\phi^{n+1}-\phi^n$ and $\|d_n\|_{H^1}$ are given by
	\begin{align}
		\nonumber\phi^{n+1}-\phi^n&=\widetilde{\phi}^{n+1}-\phi^n+\phi^{n+1}-\widetilde{\phi}^{n+1}\\
		\nonumber&=\widetilde{\phi}^{n+1}-\phi^n+o\left(\big\|\widetilde{\phi}^{n+1}-\phi^n\big\|_{H^1}\right)\widetilde{\phi}^{n+1}\\
		&=\tau d_n+o\left(\|d_n\|_{H^1}\right)\widetilde{\phi}^{n+1}\label{C_t1},\\
		\|d_n\|_{H^1}&=\mathcal{O}\left(\|\phi^n-\phi_g\|_{H^1}\right)\label{C_t2}.
	\end{align} 
	Under Taylor expansion at $\phi^n$, we have
	\begin{align*}
		E(\phi^{n+1})-E(\phi^n)=-\tau\|d_n\|^2_{\mathcal{P}_{\phi^n}}+\frac{\tau^2}{2}\left\langle \big(E''(\phi^{n})-\lambda_{\phi^n}\mathcal{I}\big)d_n,d_n\right\rangle+o\left(\|d_n\|^2_{H^1}\right).
	\end{align*}
	Similarly, we estimate the second term on the right of the above equation.
	According to {\bf Proposition \ref{Property-of-E''}}-$(iii)$, {\bf Proposition \ref{Property-P}}-$(iii)$, and the continuity of $\mathcal{P}_{\phi}$, we derive
	\begin{align*}
		&\left\langle \big(E''(\phi^n)-\lambda_{\phi^n}\mathcal{I}\big)d_n,d_n\right\rangle-\left\langle \big(E''(\phi_g)-\lambda_g\big)d_n,d_n\right\rangle=o\left(\|d_n\|^2_{H^1}\right),\\
		&\left\langle\big(\mathcal{P}_{\phi^n}-\mathcal{P}_{\phi_g}\big)d_n,d_n\right\rangle=o\left(\|d_n\|^2_{H^1}\right).
	\end{align*}
	By $d_n\in T_{\phi^n}\mathcal{M}$ and the continuity of $\text{Proj}_{\phi}^{\mathcal{P}_{\phi}}$, we get
	\begin{align*}
		d_n=\text{Proj}^{\mathcal{P}_{\phi^n}}_{\phi^n}d_n=\text{Proj}^{\mathcal{P}_{\phi_g}}_{\phi_g}d_n+o(d_n).
	\end{align*}
	This shows that 
	\begin{align*}
		\left\langle \big(E''(\phi_g)-\lambda_g\big)d_n,d_n\right\rangle=\left\langle \big(E''(\phi_g)-\lambda_g\big)\text{Proj}^{\mathcal{P}_{\phi_g}}_{\phi_g}d_n,\text{Proj}^{\mathcal{P}_{\phi_g}}_{\phi_g}d_n\right\rangle+o\left(\|d_n\|^2_{H^1}\right).
	\end{align*}
	Using {\bf Proposition \ref{Property-P}}-$(i)$, the following upper bound estimate holds
	\begin{align*}
		\left\langle \big(E''(\phi_g)-\lambda_g\big)d_n,d_n\right\rangle\le L\|d_n\|^2_{\mathcal{P}_{\phi_g}}.
	\end{align*}
	Combining the above estimates, we get 
	\begin{align*}
		\frac{\tau^2}{2}\left\langle \big(E''(\phi^n)-\lambda_{\phi^n}\big)d_n,d_n\right\rangle
		\le \frac{\tau^2}{2}L\|d_n\|^2_{\mathcal{P}_{\phi^n}}+o	\left(\tau^2\|d_n\|^2_{H^1}\right).
	\end{align*}
	The local estimate is obtained from the above result:
	\begin{align*}
		E(\phi^{n+1})-E(\phi^n)&\le-\tau\|d_n\|^2_{\mathcal{P}_{\phi^n}}+\frac{\tau^2}{2}L\|d_n\|^2_{\mathcal{P}_{\phi^n}}+o\left(\tau^2\|d_n\|^2_{H^1}\right).
	\end{align*}
	{By \textbf{Proposition~\ref{Property-P}}-$(iii)$, for any $\phi^n \in \mathcal{B}_{\sigma}(\phi_g)$, we obtain
	\[
	\|d_n\|_{H^1}\le C\|\phi^n-\phi_g\|_{H^1}\le C\sigma.
	\]
	Combining this with \textbf{(A6)}-$(ii)$, we deduce that for all sufficiently small $\varepsilon > 0$, there exists $\sigma > 0$ such that for any $\phi^n \in \mathcal{B}_{\sigma}(\phi_g)$,}
	\begin{align*}
		\left|o\left(\tau^2\|d_n\|^2_{H^1}\right)\right|\le\frac{\tau^2}{2}\varepsilon\|d_n\|^2_{\mathcal{P}_{\phi^n}}.
	\end{align*}
	Consequently, the conclusion is obtained
	\begin{align*}
		E(\phi^{n+1})-E(\phi^n)
		&\le\left(\frac{\tau^2L}{2}-\tau\right)\|d_n\|^2_{\mathcal{P}_{\phi^n}}+o\left(\tau^2\|d_n\|^2_{H^1}\right)
		\le \frac{\tau^2(L+\varepsilon)-2\tau}{2} \|d_n\|^2_{\mathcal{P}_{\phi^n}}\\
		&=-C_{\tau}\|d_n\|^2_{\mathcal{P}_{\phi^n}}
		\le-\sup\limits_{\tau\in\big(0,2/(L+\varepsilon)\big)}\left(\tau-\frac{\tau^2}{2}(L+\varepsilon)\right)\|d_n\|^2_{\mathcal{P}_{\phi^n}}\\
		&=-\frac{1}{2(L+\varepsilon)}\|d_n\|^2_{\mathcal{P}_{\phi^n}},\qquad {\rm when}\qquad \tau=1/(L+\varepsilon).
	\end{align*}
\end{proof}

To prove {\bf Theorem \ref{Var-P-convergence}}, we define the operator \(g(\phi) :=\nabla^{\mathcal{R}}_{\mathcal{P}} E(\phi)\), and let \(\mathcal{J}_{\phi_g}: H_0^1(\mathcal{D}) \to N_{\phi_g}\mathcal{M}\) denote the $\mathcal{P}_{\phi_g}$-orthogonal projection from $H^1_0(\mathcal{D})$ onto $N_{\phi_g}\mathcal{M}$.  

The lemma that follows shows the regularity of \(g\). 
\begin{lemma}\label{F-diff}
	For any \(\mathcal{P}_{\phi}\), \(g(\phi)\) is real Fr\'echet differentiable at \(\phi_g\), and the derivative \(g'(\phi_g)\) is given by
	{\rm\begin{align*}
			g'(\phi_g) = \text{Proj}^{ \mathcal{P}_{\phi_g}}_{\phi_g} \mathcal{P}^{-1}_{\phi_g} \left(E''(\phi_g) - \lambda_g \mathcal{I}\right).
	\end{align*}}
\end{lemma}
\begin{proof}
	Noting that 
	\begin{align*}
		g(\phi)= \text{Proj}^{\mathcal{P}_{\phi}}_{\phi}\mathcal{P}_{\phi}^{-1}\mathcal{H}_{\phi}\phi=\text{Proj}^{\mathcal{P}_{\phi}}_{\phi}\mathcal{P}_{\phi}^{-1}\left(\mathcal{H}_{\phi}\phi-\lambda_g\mathcal{I}\phi\right)\quad\text{and}\quad \mathcal{H}_{\phi_g}\phi_g-\lambda_g\mathcal{I}\phi_g=0,
	\end{align*}
	combined with the continuity of $\text{Proj}^{\mathcal{P}_{\phi}}_{\phi}$ (see \eqref{L-Proj}) and $\mathcal{P}_{\phi}$ at $\phi_g$, for all $h\in H_0^1(\mathcal{D})$, we obtain
	\begin{align*}
		g(\phi_g+h)-g(\phi_g)&=\text{Proj}^{\mathcal{P}_{\phi_g+h}}_{\phi_g+h}\mathcal{P}_{\phi_g+h}^{-1}\left(\mathcal{H}_{\phi_g+h}(\phi_g+h)-\lambda_g\mathcal{I}(\phi_g+h)\right)\\
		&=\text{Proj}^{\mathcal{P}_{\phi_g+h}}_{\phi_g+h}\mathcal{P}_{\phi_g+h}^{-1}\left(E''(\phi_g)h-\lambda_g\mathcal{I}h+o\left(h\right)\right)\\
		&=\text{Proj}^{\mathcal{P}_{\phi_g}}_{\phi_g}\mathcal{P}_{\phi_g}^{-1}\left(E''(\phi_g)-\lambda_g\mathcal{I}\right)h+o\left(h\right).
	\end{align*}
	This suggests that for any $\mathcal{P}_{\phi}$,
	\begin{align*}
		g'(\phi_g)h&= \text{Proj}^{\mathcal{P}_{\phi_g}}_{\phi_g}\mathcal{P}^{-1}_{\phi_g}\left(E''(\phi_g)-\lambda_g\mathcal{I}\right)h.
	\end{align*}
\end{proof}

We further define \(\mathcal{G}_{\tau}(\phi_g): N_{\phi_g}\mathcal{M} \to N_{\phi_g}\mathcal{M}\) by
\begin{align*}
	\mathcal{G}_{\tau}(\phi_g) &:=\mathcal{J}_{\phi_g} \left(I - \tau g'(\phi_g)\right)\big|_{N_{\phi_g}\mathcal{M}}= \mathcal{J}_{\phi_g} \left(I - \tau\mathcal{P}^{-1}_{\phi_g} \left(E''(\phi_g) - \lambda_g \mathcal{I}\right)\right)\Big|_{N_{\phi_g}\mathcal{M}}.
\end{align*}
The spectrum characterization of $\mathcal{G}_{\tau}(\phi_g)$ is given as follows.
\begin{lemma}\label{G-spectrum}
	Let $E$ be a
	Morse-Bott functional on $\mathcal{S}$. Then, the spectrum of $\mathcal{G}_{\tau}(\phi_g)$ fulfills
	\begin{align*}
		\sigma\left(\mathcal{G}_{\tau}(\phi_g)\right)\subset\big\{1-\tau,1-\tau\mu_1,1-\tau\mu_2,\cdots\big\},
	\end{align*}
	where $(\mu_i,v_i)\in\mathbb{R}\backslash\{0\}\times N_{\phi_g}\mathcal{M}\backslash\{0\}$ denotes the eigenpairs to the eigenvalue problem:
	\begin{align*}
		\mathcal{J}_{\phi_g}\mathcal{P}^{-1}_{\phi_g} \left(E''(\phi_g) - \lambda_g \mathcal{I}\right)v_i=\mu_iv_i.
	\end{align*}
	Furthermore, the spectral radius of $\mathcal{G}_{\tau}(\phi_g)$ is bounded by
	\begin{align*}
		\rho\left(\mathcal{G}_{\tau}(\phi_g)\right) \le\max\big\{|1-\tau\mu|,|1-\tau L|\big\}.
	\end{align*}
\end{lemma}
\begin{proof}
	Let $\widetilde{\mathcal{G}}_{\tau}:=\mathcal{G}_{\tau}(\phi_g)-(1-\tau)\mathcal{J}_{\phi_g}=\mathcal{G}_{\tau}(\phi_g)-(1-\tau)I|_{N_{\phi_g}\mathcal{M}}$.
	Since $\sigma\big(\widetilde{\mathcal{G}}_{\tau}\big)$ is only a shift $1-\tau$ with respect to $\sigma\left(\mathcal{G}_{\tau}(\phi_g)\right)$, the spectrum of $\mathcal{G}_{\tau}(\phi_g)$ is obtained by considering the spectrum of $\widetilde{\mathcal{G}}_{\tau}$. In fact, for any {uniformly} bounded sequence $\big\{v^n\big\}_{n\in\mathbb{N}}\subset N_{\phi_g}\mathcal{M}$, the sequence $\left\{\widetilde{\mathcal{G}}_{\tau}v^n\right\}_{n\in\mathbb{N}}$ contains a converging subsequence. By Rellich-Kondrachov embedding, we can extract a subsequence $\big\{v^{n_j}\big\}_{j\in\mathbb{N}}$ that converges to some $v^*\in N_{\phi_g}\mathcal{M}$ weakly in $H_0^1(\mathcal{D})$ and strongly in $L^p$ (with $1\le p<6$ for $d\le3$). Using {\bf (A6)}-$(iv)$ and {\bf Proposition \ref{Property-P}}-$(ii)$, we derive
	\begin{align*}
		\big\|\widetilde{\mathcal{G}}_{\tau}v\big\|_{H^1}&=\tau\left\|\mathcal{J}_{\phi_g}\mathcal{P}^{-1}_{\phi_g}\big(E''(\phi_g)-\mathcal{P}_{\phi_g}-\lambda_g\mathcal{I}\big)v\right\|_{H^1}\\
		&\le C\left(\big\|\mathcal{P}^{-1}_{\phi_g}\big(E''(\phi_g)-\mathcal{P}_{\phi_g}\big)v\big\|_{H^1}+\lambda_g\big\|\mathcal{P}^{-1}_{\phi_g}\mathcal{I}v\big\|_{H^1}\right)\le C\|v\|_{L^{p}}.
	\end{align*}
	Hence, replacing $v$ by $v^{n_j}-v^*$, $\widetilde{\mathcal{G}}_{\tau}v^{n_j}$ converges strongly to $\widetilde{\mathcal{G}}_{\tau}v^*$ in $H_0^1(\mathcal{D})$. This implies that $\widetilde{\mathcal{G}}_{\tau}$ is a compact operator from $N_{\phi_g}\mathcal{M}$ to $N_{\phi_g}\mathcal{M}$ {as it maps bounded sets into relatively compact ones.} The spectrum characterization of $\mathcal{G}_{\tau}(\phi_g)$ is obtained by the property of the compact operator $\widetilde{\mathcal{G}}_{\tau}$, {which is self-adjoint with respect to the $\mathcal{P}_{\phi_g}$-inner product on $N_{\phi_g}\mathcal{M}$. 
	Indeed, for all $u, v \in N_{\phi_g}\mathcal{M}$,
	\begin{align*}
		(\widetilde{\mathcal{G}}_{\tau}u,v)_{\mathcal{P}_{\phi_g}}&=\Big(\mathcal{J}_{\phi_g}\mathcal{P}^{-1}_{\phi_g} \left(E''(\phi_g) - \mathcal{P}_{\phi_g}-\lambda_g \mathcal{I}\right)u,v\Big)_{\mathcal{P}_{\phi_g}}\\
		&=\Big(\mathcal{J}_{\phi_g}\mathcal{P}^{-1}_{\phi_g} \left(E''(\phi_g) -\mathcal{P}_{\phi_g} - \lambda_g \mathcal{I}\right)v,u\Big)_{\mathcal{P}_{\phi_g}}=(\widetilde{\mathcal{G}}_{\tau}v,u)_{\mathcal{P}_{\phi_g}},
	\end{align*}
	which implies that $\widetilde{\mathcal{G}}_{\tau}$ has a real spectrum. Consequently,
	\begin{align*}
		\sigma\big(\widetilde{\mathcal{G}}_{\tau}\big) \subset \big\{0,\, \tau - \tau\mu_1,\, \tau - \tau\mu_2,\, \dots \big\}
		\quad \Longrightarrow \quad
		\sigma\big(\mathcal{G}_{\tau}(\phi_g)\big) \subset \big\{1 - \tau,\, 1 - \tau\mu_1,\, 1 - \tau\mu_2,\, \dots \big\},
	\end{align*}
	where all $\mu_i \in \mathbb{R}$.} For any eigenvalue $\mu_i$, we have
	\begin{align*}
		\mu_iv_i=\mathcal{J}_{\phi_g}\mathcal{P}^{-1}_{\phi_g} \left(E''(\phi_g) - \lambda_g \mathcal{I}\right)v_i\quad\Longrightarrow\quad\mu_i=\frac{\big\langle\big(E''(\phi_g)-\lambda_g\mathcal{I}\big)v_i,v_i\big\rangle}{\big\langle\mathcal{P}_{\phi_g}v_i,v_i\big\rangle}.
	\end{align*}
	This implies that, by {\bf Proposition \ref{Property-P}}-$(i)$, $\big\{\mu_1,\mu_2,\cdots\big\}\subset[\mu,L]$. The following content is to prove that $\mu\le1\le L$. Since $\widetilde{\mathcal{G}}_{\tau}$ is a compact operator, there exists a sequence $\{u^n\}_{n\in\mathbb{N}}\subset N_{\phi_g}\mathcal{M}$ such that $\big\|u^n\big\|_{H^1}=1$ and $\lim\limits_{n\to\infty}\widetilde{\mathcal{G}}_{\tau}u^n=0$ in $N_{\phi_g}\mathcal{M}$. Let $\widetilde{u}^n:=\widetilde{\mathcal{G}}_{\tau}u^n$, using {\bf(A6)}-$(iii)$ and -$(iv)$, we derive
	\begin{align*}
		\lim\limits_{n\to\infty}\Bigg|\frac{\big\langle\mathcal{P}_{\phi_g}\widetilde{u}^n,u^n\big\rangle}{\big\langle\mathcal{P}_{\phi_g}u^n,u^n\big\rangle}\Bigg|&\le C\lim\limits_{n\to\infty} 	\frac{\big\|\widetilde{u}^n\big\|_{H^1}\big\|u^n\big\|_{H^1}}{\big\|u^n\big\|^2_{H^1}}=0,
	\end{align*}
	and
	\begin{align*}
		\lim\limits_{n\to\infty}\frac{\big\langle\big(E''(\phi_g)-\lambda_g\mathcal{I}\big)u^n,u^n\big\rangle}{\big\langle\mathcal{P}_{\phi_g}u^n,u^n\big\rangle}&=\lim\limits_{n\to\infty}\frac{\big\langle\mathcal{P}_{\phi_g}\big(\widetilde{\mathcal{G}}_{\tau}/\tau+I\big)u^n,u^n\big\rangle}{\big\langle\mathcal{P}_{\phi_g}u^n,u^n\big\rangle}\\
		&=1+\frac{1}{\tau}\lim\limits_{n\to\infty}\frac{\big\langle\mathcal{P}_{\phi_g}\widetilde{u}^n,u^n\big\rangle}{\big\langle\mathcal{P}_{\phi_g}u^n,u^n\big\rangle}=1.
	\end{align*}
	This shows that $\big\{1,\mu_1,\mu_2,\cdots\big\}\subset[\mu,L]$. {Here, the values $\mu_i$ are eigenvalues of finite multiplicity corresponding to the compact operator $\widetilde{\mathcal{G}}_{\tau}$, while the point $1$ arises as an accumulation point of the spectrum.} Thus, $\rho\left(\mathcal{G}(\phi_g)\right) \le\max\big\{|1-\tau\mu|,|1-\tau L|\big\}$.
\end{proof}
Finally, an important lemma is proposed in the following.
\begin{lemma}\label{Nonlinear-Iteration}
	Suppose that the {bounded} linear operator $T$ on a Hilbert space $X$ satisfies the condition $\rho(T)=\rho<1$, and the sequence $\big\{v^n\big\}_{n\in\mathbb{N}}\subset X$ satisfies:
	\begin{align*}
		v^{n+1}=Tv^n+Y(v^n)\quad and \quad \lim\limits_{\|v\|_{X}\to0}\frac{\|Y(v)\|_{X}}{\|v\|_{X}}=0,
	\end{align*}
	{where $Y:X\to X$ is a mapping.}
	Then, for all sufficiently small $\varepsilon$, there exists $\sigma$ such that for all $\|v^0\|_X\le\sigma$, 
	\begin{align*}
		\|v^n\|_{X}\le C_{\varepsilon}\|v^0\|_X(\rho+\varepsilon)^n.
	\end{align*}
\end{lemma}
\begin{proof}
	Based on the discrete Gronwall inequality, the result is standard. Since $\lim\limits_{n\to\infty}\big\|T^n\big\|^{\frac{1}{n}}=\rho<1$, then for any sufficiently small $\varepsilon > 0$, there exists a constant $C_{\varepsilon}$ depending on $\varepsilon$ such that for all $n\in\mathbb{N}$, $\big\|T^n\big\|\le C_{\varepsilon}(\rho+\varepsilon/3)^n$. The condition $\lim\limits_{\|v\|_{X}\to0}\big\|Y(v)\big\|_{X}/\|v\|_{X}=0$ indicates that for any sufficiently small $\varepsilon$, there exists a small enough $\sigma_1$ such that for all $\|v\|_X\le\sigma_1$, $\big\|Y(v)\big\|_{X}\le\frac{\varepsilon}{3C_{\varepsilon}}\big\|v\big\|_{X}$. Let $\sigma\le\frac{\sigma_1}{(1+C_{\varepsilon})}$, we use mathematical induction to prove $\|v^n\|_{X}\le\sigma_1$ for all $n\ge 0$. Obviously, $n=0$ is true, now let us assume $\|v^{k}\|_{X}\le\sigma_1$ for all $k\le n-1\ (n\ge2)$. Hence, the following inequality holds for $k=n$
	\begin{align*}
		\big\|v^{n}\big\|_{X}&=\big\|Tv^{n-1}+Y(v^{n-1})\big\|_{X}\\
		&=\big\|T^2v^{n-2}+TY(v^{n-2})+Y(v^{n-1})\big\|_{X}
		=\Bigg\|T^{n}v^{0}+\sum\limits_{k=0}^{n-1}T^{n-1-k}Y(v^{k})\Bigg\|_{X}\\
		&\le\big\|T^{n}v^{0}\big\|_{X}+\sum\limits_{k=0}^{n-1}\big\|T^{n-1-k}\big\|\big\|Y(v^{k})\big\|_{X}\\
		&\le C_{\varepsilon}\big\|v^{0}\big\|_{X}(\rho+\varepsilon/3)^{n}+\sum\limits_{k=0}^{n-1}(\rho+\varepsilon/3)^{n-1-k}\frac{\varepsilon}{3}\big\|v^{k}\big\|_{X}\\
		\Longrightarrow&\qquad (\rho+\varepsilon/3)^{-n}\big\|v^{n}\big\|_{X}\le C_{\varepsilon}\big\|v^{0}\big\|_{X}+\sum\limits_{k=0}^{n-1}\frac{\varepsilon}{3\rho+\varepsilon}(\rho+\varepsilon/3)^{-k}\big\|v^{k}\big\|_{X}.
	\end{align*}
	Applying the classical discrete Gronwall inequality, we derive
	\begin{align*}
		(\rho+\varepsilon/3)^{-n}\big\|v^{n}\big\|_{X}\le C_{\varepsilon}\|v^{0}\|_{X}\Bigg(1+\frac{\varepsilon}{3\rho+\varepsilon}\Bigg)^{n}\Longrightarrow\big\|v^n\big\|_{X}\le C_{\varepsilon}\|v^{0}\|_{X}(\rho+\varepsilon)^{n}\le\sigma_1.
	\end{align*}
	This not only completes the induction but also proves the conclusion.
\end{proof}
The following remark clarifies the motivation and context behind our technical lemmas.

\begin{remark}
	If only orthogonality were required, \textbf{Lemma~\ref{Transfer}} would admit a simpler proof by considering $\argmin_{u \in \mathcal{S}} \|\phi - u\|_{L^2}^2$. However, the $L^2$-norm does not control the $H^1$-norm, creating an obstruction to establishing the Polyak-\L ojasiewicz inequality. This motivates the construction of the functional \eqref{f-phi-u}.  
	For \textbf{Lemma~\ref{F-diff}}, we emphasize that the Fr\'echet differentiability of $g(\cdot)$ at $\phi_g$ does not require $\mathcal{P}_{\phi}$ to be differentiable. {\textbf{Lemma~\ref{G-spectrum}} provides a generalized eigenvalue characterization that extends the weighted eigenvalue problems in \cite{2023The,2024Convergence} under the identification $\mu_i \leftrightarrow 1 - \mu_i$.}
	\textbf{Lemma~\ref{Nonlinear-Iteration}} is a standard tool in the local stability analysis of dynamical systems, analogous to Ostrowski's theorem for studying the stability of fixed points in nonlinear iterations (see, e.g., \cite{2023The}), and both yield the same convergence rates. Moreover, if the second-order sufficient condition holds at the minimizer (e.g., when $\Omega=0$), then the operator $\mathcal{G}_{\tau}(\phi_g)$ can be analyzed over the entire tangent space, and the best convergence rate estimate established in~\textbf{Theorem~\ref{Var-P-convergence}} for the preconditioner \eqref{Varep-P} also applies to all preconditioners satisfying~\textbf{(A6)}.
\end{remark}

With this, we are ready to prove the theorems.
\subsection{Proof of main results}

\begin{proof}[Proof of {\bf Theorem \ref{Global-Convergence}}]
	$(i)$ {\bf Sufficient descent property :}
	
	Let $e_n:=\big(\phi^{n+1}-\widetilde{\phi}^{n+1}\big)\big/\tau_n^2$, by {\bf Proposition \ref{Property-P}}-$(iv)$, we get 
	\begin{align}\label{e_n}
		\|e_n\|_{\mathcal{P}_{\phi^n}}\le\frac{1}{2}\|d_n\|^2_{L^2}\big\|\phi^n+\tau_nd_n\big\|_{\mathcal{P}_{\phi^n}}\le C_{\phi^n,d_n}\|d_n\|^2_{\mathcal{P}_{\phi^n}}.
	\end{align}
	Applying {\bf Proposition \ref{Property-of-E''}}-$(iv)$, the following inequality holds
	\begin{align*}
		&E(\phi^{n+1})-E(\phi^n)=E(\phi^n+\tau_nd_n+\tau_n^2e_n)-E(\phi^n)\\
		\le&\;\tau_n\left\langle E'(\phi^n), d_n+\tau_ne_n\right\rangle+\tau_n^2\left\langle E''(\phi^n)(d_n+\tau_ne_n),d_n+\tau_ne_n\right\rangle+\tau_n^3C_{\phi^n,d_n}\|d_n\|_{H^1}^3\\
		=&\tau_n\left(\nabla_{\mathcal{P}}E(\phi^n),d_n\right)_{\mathcal{P}_{\phi^n}}+\tau^2_n\left\langle E'(\phi^n),e_n\right\rangle+\tau_n^2\left\langle E''(\phi^n)(d_n+\tau_ne_n),d_n+\tau_ne_n\right\rangle\\
		&\qquad\qquad\qquad\qquad\qquad\qquad\qquad\qquad\qquad\qquad\qquad\qquad\qquad\quad+\tau_n^3C_{\phi^n,d_n}\|d_n\|^3_{H^1}\\
		=&\;-\tau_n\|d_n\|^2_{\mathcal{P}_{\phi^n}}+\tau^2_n\left\langle E'(\phi^n),e_n\right\rangle+\tau_n^2\left\langle E''(\phi^n)(d_n+\tau_ne_n),d_n+\tau_ne_n\right\rangle\\
		&\qquad\qquad\qquad\qquad\qquad\qquad\qquad\qquad\qquad\qquad\qquad\qquad\qquad\quad+\tau_n^3C_{\phi^n,d_n}\|d_n\|^3_{H^1}.
	\end{align*}
	Combined with {\bf Proposition \ref{Property-of-E''}}-$(ii)$, {\bf (A6)}-$(ii)$, $\|d_n\|_{\mathcal{P}_{\phi^n}}\le\big\|\mathcal{P}_{\phi}^{-1}\mathcal{H}_{\phi}\phi\big\|_{\mathcal{P}_{\phi^n}}$, and {\bf Proposition \ref{Property-P}}-$(ii)$, we further get
	\begin{align*}
		E(\phi^{n+1})-E(\phi^n)&\le-\tau_n\|d_n\|^2_{\mathcal{P}_{\phi^n}}+\tau_n^2C_{\phi^n}\|d_n\|^2_{\mathcal{P}_{\phi^n}}+\tau_n^3C_{\phi^n}\|d_n\|^2_{\mathcal{P}_{\phi^n}}\\
		\nonumber&=-\tau_n\|d_n\|^2_{\mathcal{P}_{\phi^n}}+\tau_n^2 C_{\phi^n}\|d_n\|^2_{\mathcal{P}_{\phi^n}}\\
		&=-C_{\tau_n}\|d_n\|^2_{\mathcal{P}_{\phi^n}}
	\end{align*}
	with $C_{\tau_n}:=\tau_n-\tau_n^2C_{\phi^n}.$
	Then, when $\tau_n\in(0,1/C_{\phi^n})$, $C_{\tau_n}>0$. With this, the remaining proof is done by induction. For $n=0$, by $\big\|\phi^0\big\|_{H^1}\le C\sqrt{E(\phi^0)}:=C_{E^0}$, we conclude $C_{C_{E^0}}\ge C_{\phi^0}$ and 
	\begin{align*}
		C_{\tau_0}\ge\tau_0-\tau_0^2C_{C_{E^0}}>0\quad for\ all\ \ \tau_0\in\left(0,1/C_{C_{E^0}}\right).
	\end{align*}
	Hence, there exists a constant $\tau_{\max}=1/C_{C_{E^0}}$  such that for all $\tau_0\in (0,\tau_{\max})$, we have
	\begin{align*}
		E(\phi^{1})-E(\phi^0)\le -C_{\tau_0}\|d_0\|^2_{\mathcal{P}_{\phi^0}}.
	\end{align*}
	Now, assuming that $(i)$ holds for $n=k$, we aim to show that $(i)$ holds for $n=k+1$. According to the assumption, we obtain
	\begin{align*}
		E(\phi^{k+1})\le E(\phi^0)\quad\text{and}\quad \|\phi^{k+1}\|_{H^1}\le C\sqrt{E(\phi^{k+1})}\le C_{E^0}.
	\end{align*}
	Similarly, we derive $C_{C_{E^0}}\ge C_{\phi^{k+1}}$ and 
	\begin{align*}
		C_{\tau_{k+1}}\ge\tau_{k+1}-\tau_{k+1}^2C_{C_{E^0}}>0\quad for\ all\ \ \tau_{k+1}\in(0,\tau_{\max}).
	\end{align*} 
	
	\noindent$(ii)$ {\bf Global convergence:}
	
	Since $\{E(\phi^n)\}_{n\in\mathbb{N}}$ is monotonic decreasing and bounded below (with $E(\phi^n)\le E(\phi^0)$), the sequence $\{\phi^n\}_{n\in\mathbb{N}}$ is uniformly bounded in $H_0^1(\mathcal{D})$. {By the Rellich-Kondrachov compact embedding theorem, there exists a subsequence $\{\phi^{n_j}\}_{j\in\mathbb{N}}$ that converges weakly in $H_0^1(\mathcal{D})$ to some limit $\phi_g \in \mathcal{M}$, and strongly in $L^p(\mathcal{D})$ for every $1 \leq p < 6$. Noting the following identities
	\begin{align*}
	\phi^{n_j}
	&= \nabla^{\mathcal{R}}_{\mathcal{P}} E(\phi^{n_j})
	- \big(\mathcal{P}^{-1}_{\phi^{n_j}} \mathcal{H}_{\phi^{n_j}} \phi^{n_j}
	- \phi^{n_j}\big)
	+ \lambda_{\phi^{n_j}} \, \mathcal{P}^{-1}_{\phi^{n_j}} \mathcal{I} \phi^{n_j},\\
	\lambda_{\phi^{n_j}}
		&= -\big\langle \mathcal{P}_{\phi^{n_j}} \nabla^{\mathcal{R}}_{\mathcal{P}} E(\phi^{n_j}), \phi^{n_j} \big\rangle
		+ \big\langle \mathcal{H}_{\phi^{n_j}} \phi^{n_j}, \phi^{n_j} \big\rangle \\
		&= -\big\langle \mathcal{P}_{\phi^{n_j}} \nabla^{\mathcal{R}}_{\mathcal{P}} E(\phi^{n_j}), \phi^{n_j} \big\rangle
		+ 2E(\phi^{n_j})
		+ \int_{\mathcal{D}} \Big( f(\rho_{\phi^{n_j}}) |\phi^{n_j}|^2 - F(\rho_{\phi^{n_j}}) \Big) \, \mathrm{d}\bm{x}.
	\end{align*}
	By {\bf Theorem~\ref{Global-Convergence}}-$(i)$, together with \eqref{L-H1}, \eqref{L-I}, {\bf (A3)}, and the strong convergence $\phi^{n_j} \to \phi_g$ in $L^p(\mathcal{D})$ ($1 \leq p < 6$), we obtain the following strong convergences in $H^1(\mathcal{D})$:
	\begin{align*}
		\nabla^{\mathcal{R}}_{\mathcal{P}} E(\phi^{n_j}) \xrightarrow{j \to \infty} 0, \quad
		&\mathcal{P}^{-1}_{\phi^{n_j}} \mathcal{H}_{\phi^{n_j}} \phi^{n_j} - \phi^{n_j} \xrightarrow{j \to \infty} \mathcal{P}^{-1}_{\phi_g} \mathcal{H}_{\phi_g} \phi_g - \phi_g,\\
		 \mathcal{P}^{-1}_{\phi^{n_j}} \mathcal{I} \phi^{n_j} \xrightarrow{j \to \infty} \mathcal{P}^{-1}_{\phi_g} \mathcal{I} \phi_g, \quad
		&\lambda_{\phi^{n_j}} \xrightarrow{j \to \infty} 2E_g + \int_{\mathcal{D}} \Big( f(\rho_{\phi_g}) |\phi_g|^2 - F(\rho_{\phi_g}) \Big) \, \mathrm{d}\bm{x}.
	\end{align*}
	Thus, we conclude that 
	$\phi^{n_j} \xrightarrow{j \to \infty} \phi_g$ strongly in $H^1(\mathcal{D})$. Consequently, the associated multipliers converge as well
	$
	\lambda_{\phi^{n_j}} \xrightarrow{j \to \infty} \lambda_{\phi_g} = \lambda_g,
	$ and the limit satisfies the first-order necessary condition
	\[
	\mathcal{H}_{\phi_g} \phi_g = \lambda_g \mathcal{I} \phi_g.
	\]}
\end{proof}
\begin{proof}[Proof of {\bf Theorem \ref{Local-Convergence}}] Since $E$ is a Morse-Bott functional on $\mathcal{S}$, there exists $\sigma_2$ such that both the Polyak-\L ojasiewicz\ inequality and {\bf Lemma \ref{C_t}} hold. 
	 By the Polyak-\L ojasiewicz\ inequality, there exists $\sigma<\sigma_2$ such that for any $\phi^0\in\mathcal{B}_{\sigma}(\mathcal{S})$ and some $\widetilde{\phi}_g\in\mathcal{S}$, we have
	\begin{align*}
		\|\phi^0-\widetilde{\phi}_g\|_{H^1}<\sigma<\sigma_2\quad\text{and}\quad E(\phi^0)-E_{\mathcal{S}}\le C\|\phi^0-\widetilde{\phi}_g\|^2_{H^1}\le C\sigma^2<\sigma^2_2,
	\end{align*}
{where $E_{\mathcal{S}}$ denotes the energy level of the functional $E$ on the ground state manifold $\mathcal{S}$, as defined in \eqref{mathcal_S}.}
	Thus, for all sufficiently small $\varepsilon$ and $\tau\in(0,2/(L+\varepsilon))$, the Polyak-\L ojasiewicz\ inequality and {\bf Lemma \ref{C_t}} hold when $n=0$. 
	For $\tau\in(0,2/(L+\varepsilon))$, we know that
	\begin{align*}
		C_{\tau}=\tau-\frac{\tau^2}{2}(L+\varepsilon)\in(0,1/(2(L+\varepsilon))\,],\quad1-2C_{\tau}(\mu-\varepsilon)\in\big[\,1-(\mu-\varepsilon)/(L+\varepsilon),1\big).
	\end{align*} 
	Next, we use mathematical induction to prove that for all $ n \ge 0 $, $\|\phi^n-\widetilde{\phi}_g\|_{H^1}<\sigma_2$. For $ n = 0 $, it is given that $\|\phi^n-\widetilde{\phi}_g\|_{H^1}<\sigma_2$. Assume that for some $ k \ge 1 $, $\|\phi^n-\widetilde{\phi}_g\|_{H^1}<\sigma_2$ for all $0\le n\le k$. As well, for all sufficiently small $\varepsilon$ and $\tau\in(0,2/(L+\varepsilon))$, the Polyak-\L ojasiewicz\ inequality and {\bf Lemma \ref{C_t}} hold when $0\le n\le k$. Therefore, for all $0\le n\le k$, we get 
	\begin{align*}
		E(\phi^{n+1})-E(\phi^{n})&\le-C_{\tau}\left\|d_n\right\|^2_{\mathcal{P}_{\phi^n}}\le-2C_{\tau}(\mu-\varepsilon)\left(E(\phi^{n})-E_{\mathcal{S}}\right),\\
		\Longrightarrow\; E(\phi^{n+1})-E_{\mathcal{S}}&\le\left(1-2C_{\tau}(\mu-\varepsilon)\right)\left(E(\phi^{n})-E_{\mathcal{S}}\right)\\
		&\le\left(1-2C_{\tau}(\mu-\varepsilon)\right)^{n+1}\left(E(\phi^{0})-E_{\mathcal{S}}\right),\\
		\Longrightarrow\; \left\|d_n\right\|^2_{\mathcal{P}_{\phi^n}}&\le C_{\tau}(E(\phi^{n})-E(\phi^{n+1}))\le C_{\tau}(E(\phi^{n})-E_{\mathcal{S}})\\
		&\le C_{\tau}\left(1-2C_{\tau}(\mu-\varepsilon)\right)^{n}(E(\phi^{0})-E_{\mathcal{S}}).
	\end{align*}
	According to \eqref{C_t1} and {\bf (A6)}-$(ii)$, we further get
\begin{align*}
		\|\phi^{k+1}-\widetilde{\phi}_g\|_{H^1}&\le\|\phi^{0}-\widetilde{\phi}_g\|_{H^1}+\sum\limits_{j=0}^{k}\|\phi^{j+1}-\phi^j\|_{H^1}\le\|\phi^{0}-\widetilde{\phi}_g\|_{H^1}+ C\sum\limits_{j=0}^{k}\left\|d_j\right\|_{\mathcal{P}_{\phi^j}}\\
		&\le \sigma+C\sigma\sum\limits_{j=0}^{k}\sqrt{1-2C_{\tau}(\mu-\varepsilon)}^{j}\le C\sigma.
	\end{align*}
	Hence, we choose $\sigma$ to satisfy $C\sigma<\sigma_2$. This suggests that $\|\phi^n-\widetilde{\phi}_g\|_{H^1}<\sigma$ for all $0\le n\le k+1,\ k\ge1$. That completes the induction. 
	
	The convergence rates of energy $E(\phi^n)$ and $d_n$ are immediately obtained:
	\begin{align*}
		E(\phi^{n+1})-E_{\mathcal{S}}&\le\left(1-2C_{\tau}(\mu-\varepsilon)\right)^{n+1}\left(E(\phi^{0})-E_{\mathcal{S}}\right)\\
		\left\|d_n\right\|^2_{\mathcal{P}_{\phi^n}}&\le C_{\tau}\left(E(\phi^{n})-E_{\mathcal{S}}\right)\le C_{\varepsilon}\left(E(\phi^{0})-E_{\mathcal{S}}\right)\left(1-2C_{\tau}(\mu-\varepsilon)\right)^{n}.
	\end{align*}
	For $\big\{\phi^n\big\}_{n\in\mathbb{N}}$, by \eqref{C_t1}, we have
	\begin{align}\label{Phi-eq-Grad}
		\nonumber\|\phi^{m}-\phi^n\|_{H^1}&\le\sum\limits_{j=n}^{m-1}\|\phi^{j+1}-\phi^j\|_{H^1}\le C\sum\limits_{j=n}^{m-1}\left\|d_j\right\|_{H^1}\le C\sum\limits_{j=n}^{m-1}\sqrt{E(\phi^{j})-E_{\mathcal{S}}}\\
		\nonumber&\le  C_{\varepsilon}\sqrt{E(\phi^{0})-E_{\mathcal{S}}}\sum\limits_{j=n}^{m-1}\left(\sqrt{1-2C_{\tau}(\mu-\varepsilon)}\right)^{j}\\
		&\le C_{\varepsilon}\sqrt{E(\phi^{0})-E_{\mathcal{S}}}\left(\sqrt{1-2C_{\tau}(\mu-\varepsilon)}\right)^{n}.
	\end{align}
	This means that $\left\{\phi^n\right\}_{n\in\mathbb{N}}$ is a Cauchy sequence, and is convergent. Let $m\to\infty$, by the Polyak-\L ojasiewicz\ inequality, and the continuity of $\nabla^{\mathcal{R}}_{\mathcal{P}} E(\phi)$, there is linear convergence as follows for $\big\{\phi^n\big\}_{n\in\mathbb{N}}$
	\begin{align*}
		\|\phi^n-\phi_g\|_{H^1}&\le C_{\varepsilon}\sqrt{E(\phi^{0})-E(\phi_g)}\left(\sqrt{1-2C_{\tau}(\mu-\varepsilon)}\right)^{n}\\
		&\le C_{\varepsilon}\|\phi^0-\phi_g\|_{H^1}\left(\sqrt{1-2C_{\tau}(\mu-\varepsilon)}\right)^{n}.
	\end{align*}
	In particular, with the optimal stepsize choice $\tau=1/(L+\varepsilon)$, this choice yields a precisely characterized convergence rate
	\begin{align*}
		\|\phi^n-\phi_g\|_{H^1}&\le C_{\varepsilon}\|\phi^0-\phi_g\|_{H^1}\left(\sqrt{1-\frac{\mu-\varepsilon}{L+\varepsilon}}\right)^{n}.
	\end{align*}
\end{proof}
\begin{proof}[Proof of {\bf Theorem \ref{Var-P-convergence}}] According to {\bf Theorem \ref{Local-Convergence}}, we already know that this sequence $\left\{\phi^n\right\}_{n\in\mathbb{N}}$ is linearly convergent for all $\tau\in(0,2/(L+\varepsilon))$ and for any $\phi^0\in\mathcal{B}_{\sigma}(\mathcal{S})$. Now we derive the optimal local convergence rate. Using {\bf Proposition \ref{Property-P}}-$(iii)$, the Polyak-\L ojasiewicz\ inequality, and \eqref{Phi-eq-Grad}, we obtain
	\begin{align}\label{Phi-eq-Energy}
		\nonumber\left\|\nabla_{\mathcal{P}}^{\mathcal{R}}E^n\right\|_{H^1}&\le C\|\phi^n-\phi_g\|_{H^1}\le C\sum\limits_{k=n}^{\infty}\sqrt{E(\phi^{k})-E_{\mathcal{S}}}\\
		&\nonumber\le C\sum\limits_{k=n}^{\infty}\left(\sqrt{1-2C_{\tau}(\mu-\varepsilon)}\right)^{k-n}\sqrt{E(\phi^{n})-E_{\mathcal{S}}}\\
		&\le C\sqrt{E(\phi^{n})-E_{\mathcal{S}}}\le C	\left\|\nabla_{\mathcal{P}}^{\mathcal{R}}E^n\right\|_{H^1}.
	\end{align}	
	And then we have  $\sum\limits_{k=n}^{\infty}o\left(\phi^k-\phi_g\right)=o\left(\phi^n-\phi_g\right)$ by 
	\begin{align*}
		\left\|\sum_{k=n}^{\infty}o\big(\phi^k-\phi_g\big)\right\|_{H^1}\le\delta_{\sigma}\sum_{k=n}^{\infty}\|\phi^k-\phi_g\|_{H^1}\le C\delta_{\sigma}\|\phi^n-\phi_g\|_{H^1},
	\end{align*}
	where $\delta_{\sigma}\to0^+$ as $\sigma\to0^+$. 
	
	{We perform a local linearization of the P-RG algorithm around the converged state $\phi_g$, yielding
		\begin{align*}
			\phi^{n+1} 
			&= \phi^n - \tau\, \mathrm{Proj}^{\mathcal{P}_{\phi_g}}_{\phi_g} \mathcal{P}_{\phi_g}^{-1}\big(E''(\phi_g)-\lambda_{\phi_g}\mathcal{I}\big)(\phi^n - \phi_g) + o(\phi^n - \phi_g).
		\end{align*}
		Noting that for any $v \in \mathrm{span}\{i\phi_g,\, i\mathcal{L}_z\phi_g\}$, we have
		\begin{align*}
			\Big(\mathrm{Proj}^{\mathcal{P}_{\phi_g}}_{\phi_g} \mathcal{P}_{\phi_g}^{-1}&\big(E''(\phi_g)-\lambda_{\phi_g}\mathcal{I}\big)(\phi^n - \phi_g),\, v\Big)_{L^2}\\
			&= \Big(\mathcal{P}_{\phi_g}^{-1}\big(E''(\phi_g)-\lambda_{\phi_g}\mathcal{I}\big)(\phi^n - \phi_g),\, v\Big)_{L^2} \\
			&= \Big\langle \big(E''(\phi_g)-\lambda_{\phi_g}\mathcal{I}\big)(\phi^n - \phi_g),\, \mathcal{P}_{\phi_g}^{-1}\mathcal{I} v \Big\rangle \\
			&= \Big\langle \big(E''(\phi_g)-\lambda_{\phi_g}\mathcal{I}\big)(\phi^n - \phi_g),\, \big(E''(\phi_g)-(\lambda_g-\sigma_0)\mathcal{I}\big)^{-1}\mathcal{I} v \Big\rangle\\
			&= \Big\langle \big(E''(\phi_g)-\lambda_{\phi_g}\mathcal{I}\big)(\phi^n - \phi_g),\, \mathcal{I} v/\sigma_0 \Big\rangle\\
			&= 0.
		\end{align*}
		Consequently,
		$
		\mathrm{Proj}^{\mathcal{P}_{\phi_g}}_{\phi_g} \mathcal{P}_{\phi_g}^{-1}\big(E''(\phi_g)-\lambda_{\phi_g}\mathcal{I}\big)(\phi^n - \phi_g) \in N_{\phi_g}\mathcal{M}.
		$
		Applying $I - \mathcal{J}_{\phi_g}$ to both sides of the linearized iteration yields
		\begin{align*}
			(I - \mathcal{J}_{\phi_g})(\phi^{n+1} - \phi^n) = o(\phi^n - \phi_g).
		\end{align*}
		This implies that
		\begin{align*}
			\phi^{n+1} - \phi^n 
			= \big(\mathcal{J}_{\phi_g} + I - \mathcal{J}_{\phi_g}\big)(\phi^{n+1} - \phi^n) = \mathcal{J}_{\phi_g}(\phi^{n+1} - \phi^n) + o(\phi^n - \phi_g),
		\end{align*}
		and by summing from $k = n$ to $\infty$, we obtain
		\begin{align*}
			\phi^n - \phi_g 
			= \mathcal{J}_{\phi_g}(\phi^n - \phi_g) + \sum_{k=n}^{\infty} o(\phi^k - \phi_g) 
			= \mathcal{J}_{\phi_g}(\phi^n - \phi_g) + o(\phi^n - \phi_g).
		\end{align*}
		Thus, the dynamics are effectively governed by the projected component $\mathcal{J}_{\phi_g}(\phi^n - \phi_g)$.}
		
	We can now identify the optimal local convergence rate of $\mathcal{J}_{\phi_g}(\phi^{n}-\phi_g)$. Specifically,
	\begin{align*}
		\mathcal{J}_{\phi_g}(\phi^{n+1}-\phi^n)&=\phi^{n+1}-\phi^n+o\left(\phi^{n}-\phi_g\right)=-\tau\nabla_{\mathcal{P}}^{\mathcal{R}}E^n+o\left(\phi^{n}-\phi_g\right)\\
		&=-\tau g'(\phi_g)(\phi^{n}-\phi_g)+o\left(\phi^{n}-\phi_g\right)\\
		\Longrightarrow\;\mathcal{J}_{\phi_g}(\phi^{n+1}-\phi_g)&=\mathcal{J}_{\phi_g}(\phi^{n}-\phi_g)-\tau g'(\phi_g)\mathcal{J}_{\phi_g}(\phi^{n}-\phi_g)+o\left(\phi^{n}-\phi_g\right)\\
		&=\mathcal{G}_{\tau}(\phi_g)\mathcal{J}_{\phi_g}(\phi^{n}-\phi_g)+o\left(\mathcal{J}_{\phi_g}(\phi^{n}-\phi_g)\right).
	\end{align*}
	Using {\bf Lemma \ref{G-spectrum}} and {\bf Lemma \ref{Nonlinear-Iteration}}, the faster local convergence rate of $\mathcal{J}_{\phi_g}(\phi^{n}-\phi_g)$ is obtained, for all $\phi^0\in\mathcal{B}_{\sigma}(\mathcal{S})$ and $\tau\in(0,2/(L+\varepsilon))$,
	\begin{align*}
		\left\|\mathcal{J}_{\phi_g}(\phi^{n}-\phi_g)\right\|_{H^1}&\le C_{\varepsilon}\|\phi^0-\phi_g\|_{H^1}\left(\max\big\{|1-\tau\mu|,|1-\tau L|\big\}+\varepsilon\right)^{n}.
	\end{align*}
	Based on $\phi^n-\phi_g=\mathcal{J}_{\phi_g}(\phi^n-\phi_g)+o(\phi^n-\phi_g)$, we have proved that
	\begin{align*}
		\|\phi^n-\phi_g\|_{H^1}\le C_{\varepsilon}\|\phi^0-\phi_g\|_{H^1}\left(\max\big\{|1-\tau\mu|,|1-\tau L|\big\}+\varepsilon\right)^{n}.
	\end{align*}
	In additon, when $\tau=2/(L+\mu)$, the optimal local convergence rate is obtained
	\begin{align*}
		\|\phi^n-\phi_g\|_{H^1}\le C_{\varepsilon}\|\phi^0-\phi_g\|_{H^1}\Bigg(\frac{L-\mu}{L+\mu}+\varepsilon\Bigg)^{n}.
	\end{align*}
\end{proof}
\begin{proof}[Proof of {\bf Corollary} \ref{E-eq-Phi}]
	According to \eqref{Phi-eq-Energy} and {\bf Lemma \ref{C_t}}, we get {
	\begin{align*}
		\|\phi^n-\phi_g\|_{H^1}\lesssim \sqrt{E^{n}-E(\phi_g)}\lesssim \|\nabla_{\mathcal{P}}^{\mathcal{R}}E^n\|_{H^1}\lesssim   \sqrt{E^{n}-E^{n+1}}.
	\end{align*}}
	Moreover, combining \eqref{Phi-eq-Energy} and the Polyak-\L ojasiewicz\ inequality, we further get {
	\begin{align*}
		\sqrt{E^{n}-E^{n+1}}\le\sqrt{E^{n}-E(\phi_g)}\lesssim \|\nabla_{\mathcal{P}}^{\mathcal{R}}E^n\|_{H^1}\lesssim \|\phi^n-\phi_g\|_{H^1}.
	\end{align*}}
	We complete the proof.
\end{proof}

\section{Numerical experiment}
\label{Sec5} 
In this section, we  verify numerically the  assumption of Morse-Bott property (i.e. \textbf{Definitiaon \ref{Morse-Bott-P}}) on the Gross-Pitaevskii energy functional    
and the  local convergence rate (i.e. \textbf{Theorems \ref{Local-Convergence} and \ref{Var-P-convergence}}) of the P-RG with different preconditioners around 
the ground state $\phi_g$.  To this end, we consider the minimization problem \eqref{Riem-Opt-Problem} on a disk 
$\mathcal{D} =: \big\{(x,y)=(r\cos(\Theta), r\sin(\Theta)) \mid r\in[0, 12], \Theta\in[0, 2\pi] \big\}$.
The trapping potential, nonlinear interaction and  angular velocity are respectively set as 
$V(\bm{x}) = |\bm{x}|^2 / 2$,   $f(s) = 500 s$ and  $\Omega = 0.9$. 

\vspace{0.3cm}

To numerically solve  problem \eqref{Riem-Opt-Problem},  we utilize respectively the standard   eighth-order and second-order central finite  difference method to discretize 
all related derivatives in the P-RG w.r.t. $\Theta$ and $r$
on   an equally-spacing grids  $\widetilde{\mathcal{D}}=:\big\{(r_{i+1/2}, \Theta_j) \mid i=0,\cdots, N_r-1, j=0,\cdots, N_\Theta-1 \big\}$. Here, 
$r_{i+1/2}=(i+1/2) h_r$, $\Theta_j=j h_\Theta$ with  $h_r=12/2^8$ and $ h_\Theta=2\pi/2^{10}$ the mesh sizes in $r$- and $\Theta$-direction. 
The P-RG is stopped when meet the  criterion $r^n := \left\|\mathcal{H}_{\phi^n}\phi^n-\widetilde{\lambda}_{\phi^n}\phi^n\right\|_{\infty} \le 10^{-10}$, and the resulted iterate $\phi^n$ is regarded as
the ground state $\phi_g$.

\begin{exmp}
	\label{eg:MB_prop}
	Here, we check if the Gross-Pitaevskii energy functional $E(\phi)$ is a Morse-Bott functional 
	at the ground state $\phi_g$. We first compute $\phi_g$ via the P-RG in two stages using different preconditioners. In the first stage, {we employ the globally convergent preconditioner $\mathcal{P}_\phi = \mathcal{H}_{\phi}$ for $10^{4}$ iterations to obtain an approximate solution sufficiently close to the ground state. 
	This initial phase is essential, as the quasi-optimal local preconditioner introduced below is not designed for global convergence and may fail if initialized far from the solution.} In the second stage, we switch to a quasi-optimal local preconditioner given by $\mathcal{P}_\phi = E''(\phi) - (\widetilde{\lambda}_{\phi} - \sigma_0)\mathcal{I}$ with $\sigma_0 = 10^{-1}$. After an additional $7,224$ iterations, the termination conditions are satisfied. Then, we compute the chemical potential of $\phi_g$, i.e., $\lambda_g=\left\langle\mathcal{H}_{\phi_g}\phi_g,\phi_g\right\rangle$, and 
	the first five smallest eigenvalues $\lambda_\ell\, (\ell=1, \cdots, 5)$ of $E''(\phi_g)|_{T_{\phi_g}\mathcal{M}}$. 
\end{exmp}

Fig. \ref{fig1} shows the contour plots of the  density  $|\phi_g|^2$. 
Table \ref{tab1} lists the value of  $\lambda_g$ and $\lambda_\ell$ ($\ell=1, \cdots, 5$).
From the table and additional results not shown here for brevity, we can obtain that: the smallest eigenvalue of  $E''(\phi_g)|_{T_{\phi_g}\mathcal{M}}$ equals to $\lambda_g$ and its multiplicity
is two (i.e. $\lambda_1=\lambda_2<\lambda_3$). This implies  $E''(\phi_g)|_{T_{\phi_g}\mathcal{M}}$ has only two {eigenfunctions} $i \phi_g$ and $i \mathcal{L}_z\phi_g$ according to Proposition \ref{Prop1}, 
hence $\ker\left(E''(\phi_g) - \lambda_g \mathcal{I}\right)|_{T_{\phi_g}\mathcal{M}} =T_{\phi_g}\mathcal{S}$. Therefore, the Gross-Pitaevskii energy functional $E(\phi)$ 
is a Morse-Bott functional which confirms that the assumption in theorem \ref{Local-Convergence}-\ref{Var-P-convergence} is reasonable.

\begin{figure}[htpb]	
	\centering{	
		\includegraphics[height=5.4cm,width=7cm]{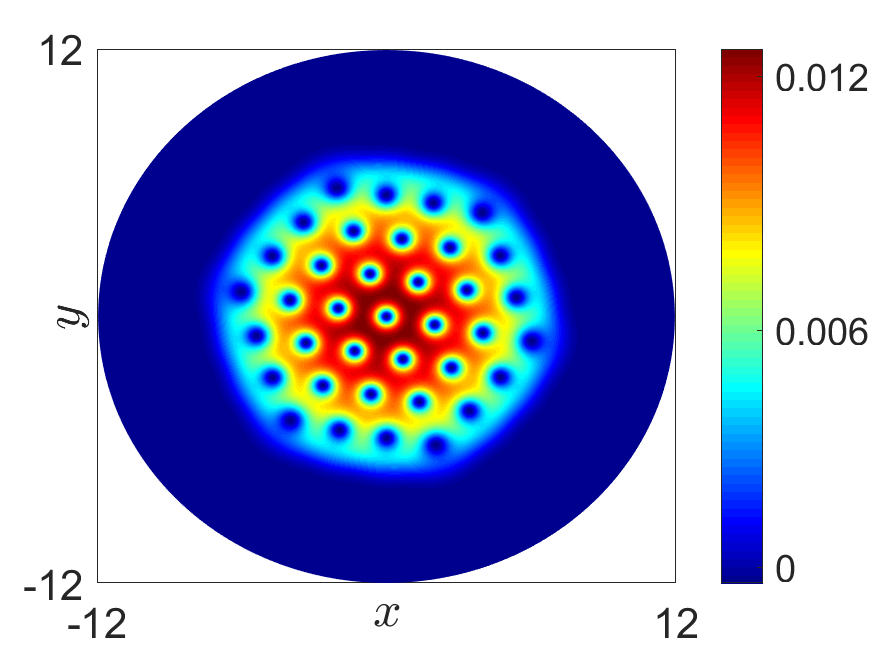}
	}
	\caption{Contour plots of the  density of the ground state  $|\phi_g(\bm{x})|^2$.}
	\label{fig1}
\end{figure}

\begin{table}[htpb]
	\tabcolsep 11pt  
	\caption{The value of  $\lambda_g$ and the first five smallest  eigenvalues of $E''(\phi_g)|_{T_{\phi_g}\mathcal{M}}$ in example \ref{eg:MB_prop}.}\label{tab1}
	\begin{center}
		\vspace{-1.5em}	\def\temptablewidth{1\textwidth}
		{\rule{\temptablewidth}{1pt}}
		\setlength{\tabcolsep}{1.4em}
		\begin{tabularx}{\temptablewidth}{c|c|c|c|c|c}
			$\lambda_g$ & $\lambda_1$ &$\lambda_2$ &$\lambda_3$ &$\lambda_4$ &$\lambda_5$  \\
			\hline
			$6.68323527$ & $6.68323527$ &  $6.68323527$  & $6.68344588$ & $6.68344588$ & $6.68559326$  \\
		\end{tabularx}
		
		{\rule{\temptablewidth}{1pt}}
	\end{center}
	
\end{table}

\begin{exmp}
	\label{eg:conv_rate_test}
	Here, we test the theoretical convergence rates of P-RG with different preconditioners around 
	the ground state $\phi_g$ shown in \textbf{ Theorems \ref{Local-Convergence} and  \ref{Var-P-convergence}}. 
	To this end, we take the same $\phi_g$ as  studied in last example. We  compare the performance of  
	P-RG  with following four  preconditioners:
	\item $(i)$ $\mathcal{P}_\phi=\mathcal{P}_1 := -\frac{1}{2}\Delta + V(\bm{x})$, $(ii)$ $\mathcal{P}_\phi=\mathcal{P}_2 := \mathcal{H}_0$, $(iii)$ $\mathcal{P}_\phi=\mathcal{P}_3 := \mathcal{H}_{\phi}$,
	\item  $(iv)$ $\mathcal{P}_\phi=\mathcal{P}_4 := E''(\phi) - (\widetilde{\lambda}_{\phi} - \sigma_0)\mathcal{I}$ with $\sigma_0 = 10^{-3}$.
	
	\vspace{0.2cm}
	
	Noticed that the P-RG with preconditioners $\mathcal{P}_1$ and $\mathcal{P}_2$  lead  to the projected Sobolev gradient methods proposed by Danaila et. al. in \cite{2010A, 2017Computation},
	P-RG with   $\mathcal{P}_3$ lead  to the one  proposed by Henning et. at.  in \cite{2020Sobolev},  while the P-RG with $\mathcal{P}_4$ is our proposed scheme. 
	Firstly, we computethe lower bound of the generalized eigenvalues of $\big(E''(\phi_g) - \lambda_g \mathcal{I},\, \mathcal{P}_{\phi_g}\big)$ on $N_{\phi_g}\mathcal{M}$ and the upper bound on $T_{\phi_g}\mathcal{M}$, i.e. $\mu$ and $L$ in \eqref{Mu-L}.
	Then, we  compute the optimal descent step size $\tau$ and theoretical convergence rate $\rho$ for the P-RG, i.e., 
	$\tau=1/L$ and $\rho=\sqrt{1-\mu/L}$  for P-RG with $\mathcal{P}_1$-$\mathcal{P}_3$, while
	$\tau=2/(L+\mu)$ and $\rho=(L-\mu)/(L+\mu)$ for  P-RG with $\mathcal{P}_4$.
	Secondly, we test the actual local convergence rate of these P-RG. We start the P-RG  with an initial guess $\phi^{0}$ close to  $\phi_g$, specifically selected from an intermediate iteration in the second stage of solving for $\phi_g$ as described in \textbf{Example~\ref{eg:MB_prop}}, 
	for which $\|\phi^{0} - \phi_g\|_{H^1}\approx2\times10^{-2}$. {Since this initial guess is closer to $\phi_g$ than the starting point used in the second stage of Example~5.1, we experimented with a more aggressive choice of $\sigma_0 = 10^{-3}$, thereby verifying that a smaller $\sigma_0$ indeed yields a faster local convergence rate, provided the coercivity condition holds.} The iteration is terminated when $E(\phi^n) - E(\phi_g) \le 10^{-14}$.
	According to \textbf{ Corollary} \ref{E-eq-Phi}, we used $\sqrt{E(\phi^n)-E(\phi_g)}$ to examine the actual convergence rate of the P-RG.

\end{exmp} 
Table \ref{tab2} lists the values of $\mu$, $L$,  $\tau$ and the theoretical convergence rate $\rho$ as predicted in \textbf{Theorems 
\ref{Local-Convergence}-\ref{Var-P-convergence}} of the P-RG with different preconditioners.  Fig. \ref{fig2} shows the evolution of the errors 
$\sqrt{E(\phi^n)-E(\phi_g)}\sim \mathcal{O}(\rho^n)$ actually computed by these P-RG. {The iterations are initialized sufficiently close to the ground state so that the observed convergence behavior reflects the local asymptotic regime.} From the table and additional results not shown here for brevity, we can obtain that: 
$(i)$ The actual convergence rates of those P-RG agree  well with those theoretical predictions (c.f.  Fig. \ref{fig2} red-colored solid lines and black-colored dashed lines), 
which numerically confirm that the estimates of the local convergence rate for P-RG with different preconditioners  in \textbf{Theorems  \ref{Local-Convergence}-\ref{Var-P-convergence}} 
are correct and sharp (c.f.   Fig. \ref{fig2} red-colored solid lines and blue-colored dashdot lines). 
$(ii)$ The P-RG with preconditioner $\mathcal{P}_4$ significantly outperforms  P-RG with other preconditioners in term of computational efficiency.  For example, in our tested case, 
P-RG with  preconditioner $\mathcal{P}_4$ 
converges within $10^2$ steps  (c.f.  Fig. \ref{fig2} $(iv)$) shown here, while 
P-RG with  preconditioner  $\mathcal{P}_1$,  $\mathcal{P}_2$ and  $\mathcal{P}_3$ requires more than $10^5$  steps to converge (c.f.  Fig. \ref{fig2} $(i)$-$(iii)$). 
Indeed, as indicated in \textbf{Theorem \ref{Var-P-convergence}} and shown in  Fig. \ref{fig2} $(iv)$, the P-RG with  preconditioner $\mathcal{P}_4$ is the best P-RG scheme in term  of local convergence.


\begin{table}[ht] 
	\tabcolsep 11pt  
	\caption{{The values of $\mu$, $L$, the descent step size $\tau$, and the theoretical convergence rate $\rho$ for the P-RG method with different preconditioners in Example~\ref{eg:conv_rate_test}. 
		For $\mathcal{P}_1$--$\mathcal{P}_3$, the parameters are given by $\tau = 1/L$ and $\rho = \sqrt{1 - \mu/L}$, as predicted by \textbf{Theorem~\ref{Local-Convergence}}. 
		For $\mathcal{P}_4$, we use the sharper estimates $\tau = 2/(L + \mu)$ and $\rho = (L - \mu)/(L + \mu)$ from \textbf{Theorem~\ref{Var-P-convergence}}.}}\label{tab2}
	\begin{center}
		\vspace{-1.5em}	\def\temptablewidth{1\textwidth}
		{\rule{\temptablewidth}{1pt}}
		\setlength{\tabcolsep}{1.8em}
		\begin{tabularx}{\temptablewidth}{c|c|c|c|c}
			&$\mathcal{P}_1$ & $\mathcal{P}_2$ &  $\mathcal{P}_3$ & $\mathcal{P}_4$\\
			\hline
			$\mu$\quad&\quad$8.249\times10^{-6}$\quad &\quad$5.811\times10^{-5}$\quad & \quad$3.168\times10^{-5}$ \quad  &\quad $0.17397014$\quad\\
			\hline
			$L$\quad &\quad$6.33028729$\quad &\quad $8.53455937$\quad &\quad $1.65411833$\quad&	\quad$1$\quad \\
			\hline
			$\tau$\quad &\quad$0.15797071$\quad&\quad$0.11717066$\quad&\quad$0.60455167$\quad&\quad$1.70362084$\quad\\
			\hline
			$\rho$\quad&\quad$0.99999934$\quad &\quad$0.99999659$\quad &\quad$0.99999042$\quad &\quad$0.70362084$\quad\\
		\end{tabularx}
		{\rule{\temptablewidth}{1pt}}
	\end{center}
\end{table}

\begin{figure}[ht]	
	\centering{	
		$(i)$\;\includegraphics[height=4cm,width=5cm]{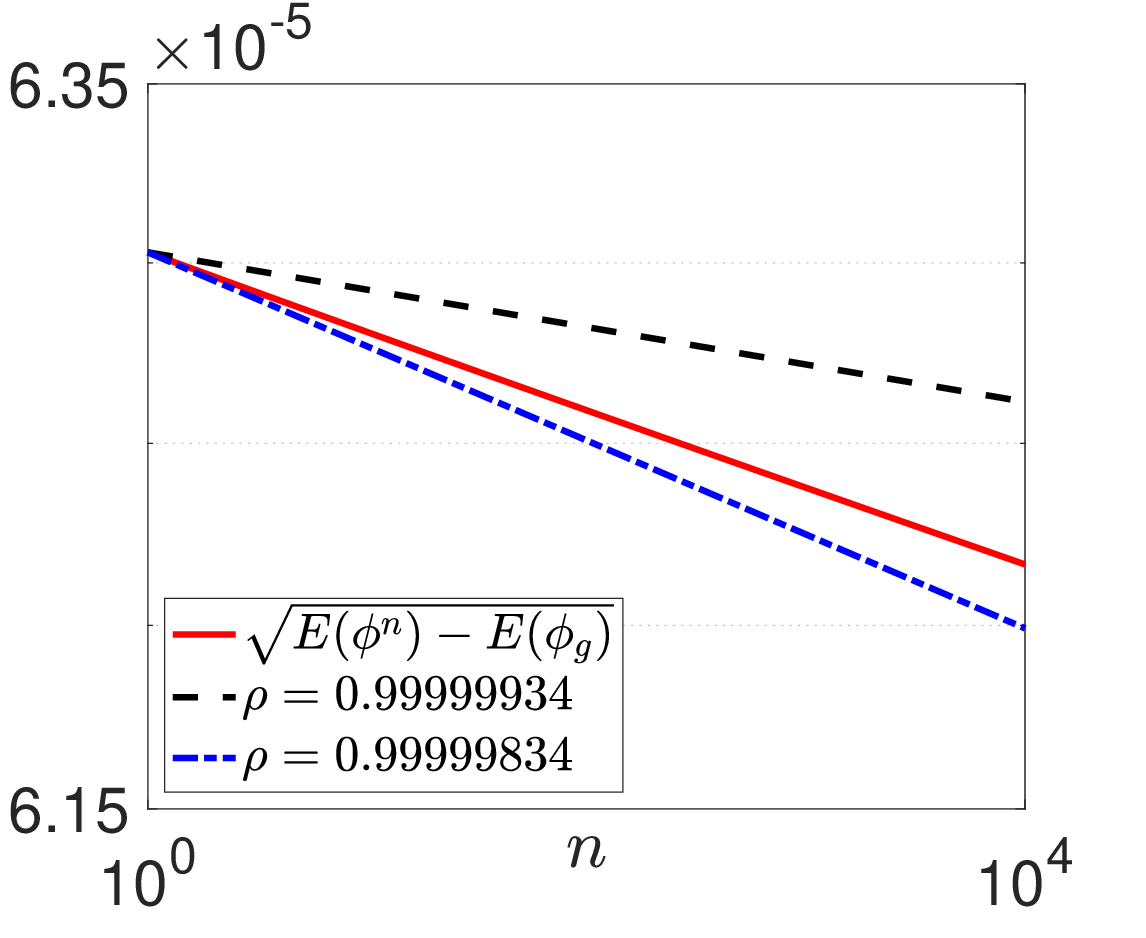} \qquad
		$(ii)$\;\includegraphics[height=4cm,width=5cm]{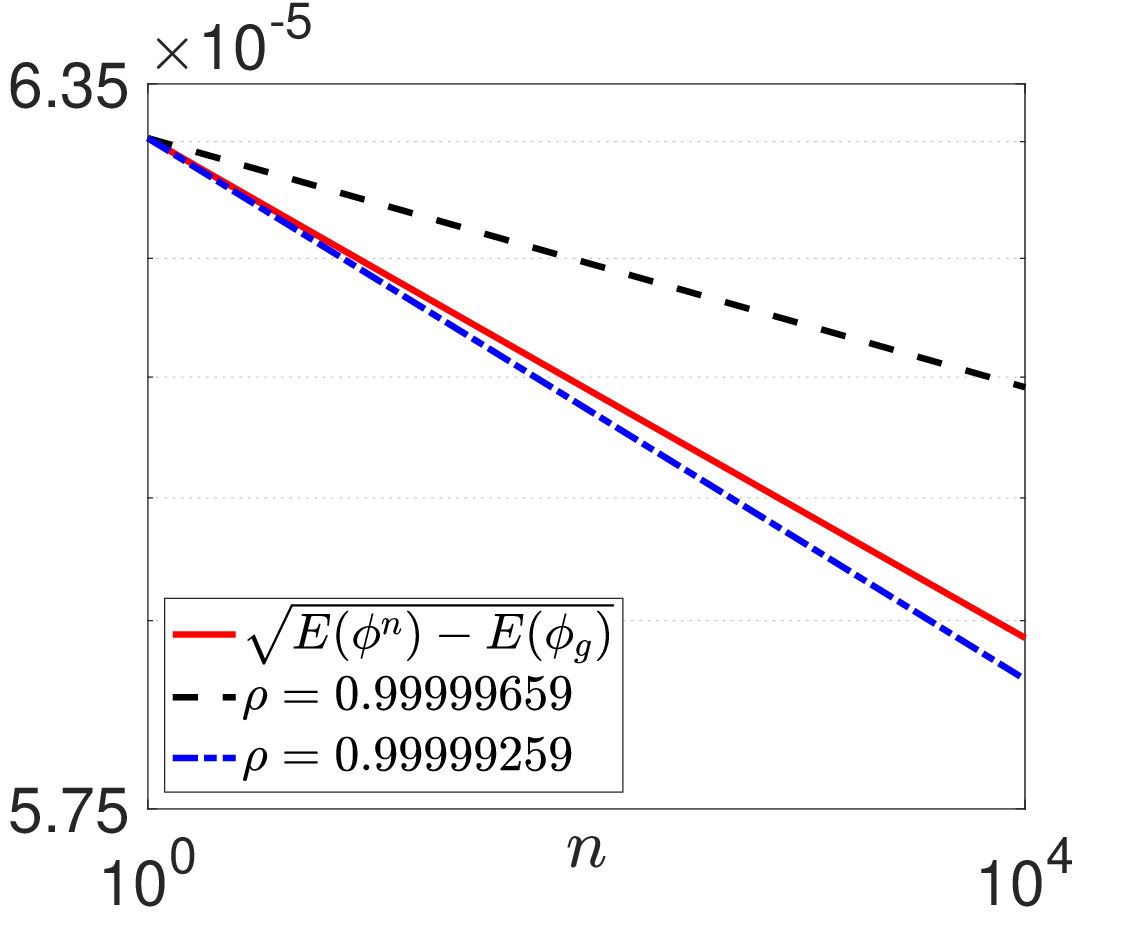}\\[1em]
		$(iii)$\;\includegraphics[height=4cm,width=5cm]{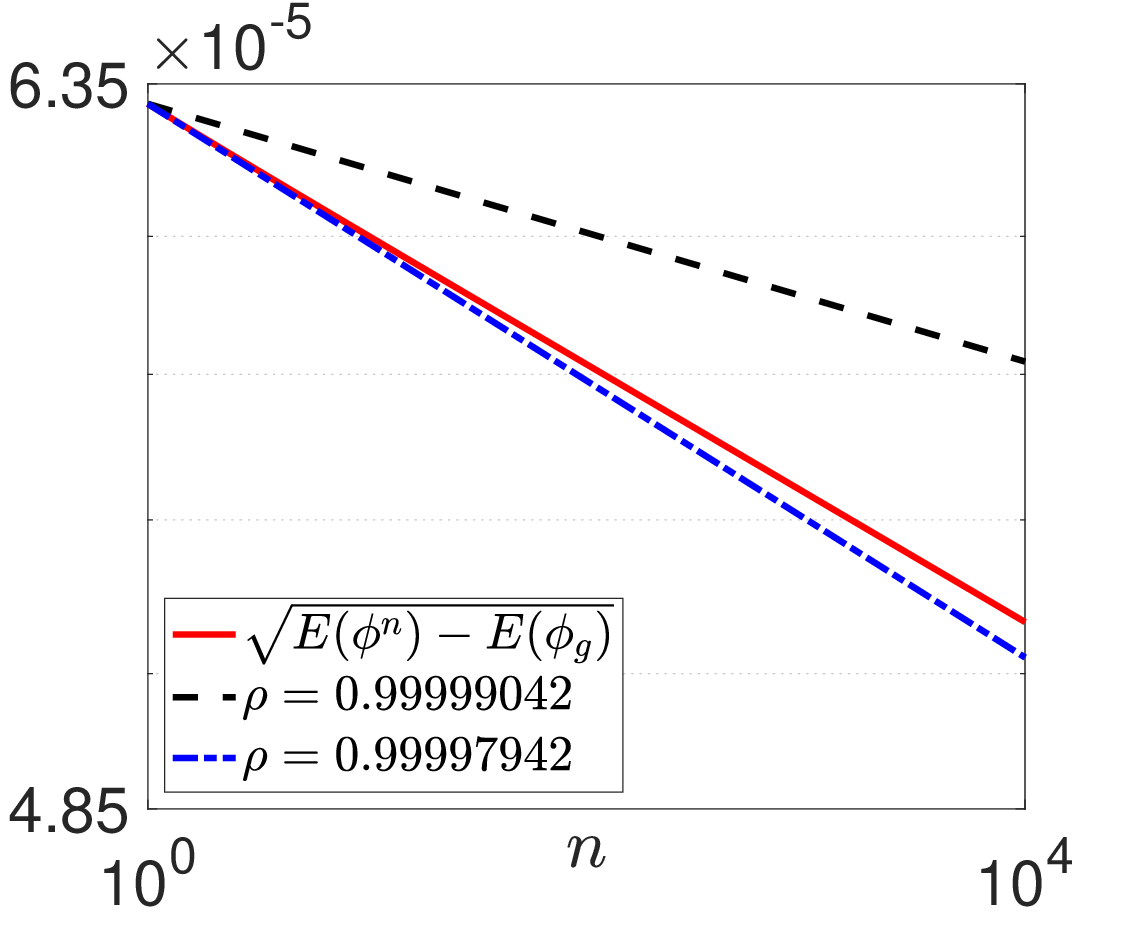}\qquad
		$(iv)$\;\includegraphics[height=4cm,width=5cm]{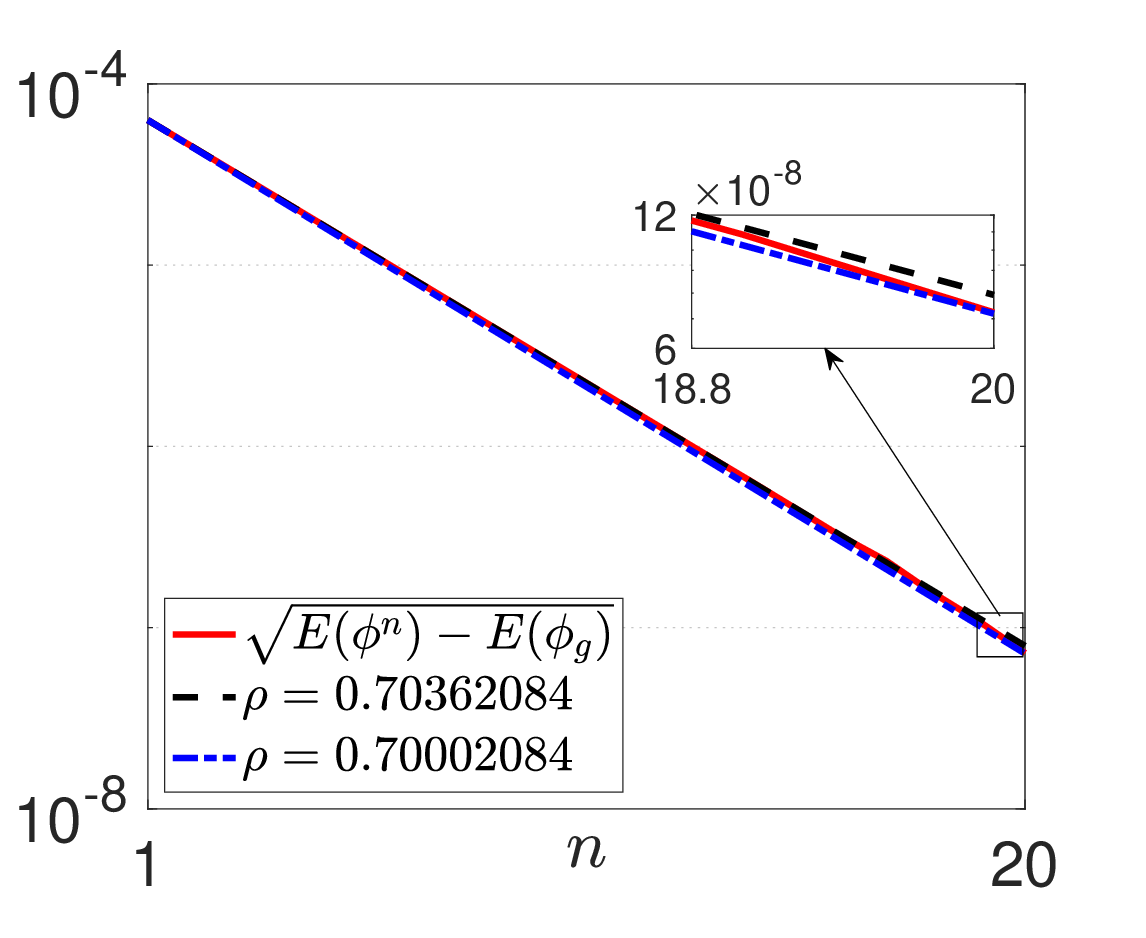}		
	}
	\caption{Plots of the error $\sqrt{E(\phi^n)-E(\phi_g)}\sim \mathcal{O}(\rho^n)$ w.r.t step $n$ for  the P-RG (the red-colored solid lines) with  preconditioners  $\mathcal{P}_1$- $\mathcal{P}_4$  
		(from (i) to (iv)) in example \ref{eg:conv_rate_test}:
		the black-colored dashed lines represent errors $\mathcal{O}(\rho^n)$ with 
		theoretical convergence rate  $\rho$ as predicted in theorems 
		\ref{Local-Convergence}-\ref{Var-P-convergence} and computed in table \ref{tab2}, while the 
		blue-colored dashdot lines  represent errors  $\mathcal{O}(\rho^n)$ with  $\rho$ sightly small than the actual convergence rate.}
	\label{fig2}
\end{figure}





\section{Conclusion}\label{Sec6}

In this paper, according to the properties of Gross-Pitaevskii energy functional, the preconditioned Riemannian gradient methods (P-RG) 
are proposed to compute the minimizers of rotating Gross-Pitaevskii energy functional. We rigorously prove the global and quantifiable local 
convergence of these methods. Our analysis reveals that the local convergence rate  critically depends on the condition number 
of \( \mathcal{P}^{-1}_{\phi_g}(E''(\phi_g)-\lambda_g\mathcal{I}) \) on \( N_{\phi_g}\mathcal{M} \). 
This insight suggests that a practical preconditioner should follow \eqref{Varep-P}, i.e.,
$\mathcal{P}_{\phi}=E''(\phi)-\big(\left\langle\mathcal{H}_{\phi}\phi,\phi\right\rangle-\sigma_0\big)\mathcal{I}$, and for this specific choice, we establish an optimal local linear convergence rate. 
In the end, numerical experiments 
show the assumption, i.e. the Gross-Pitaevskii energy functional is a Morse-Bott functional, 
is justifiable, and also confirm the theoretical results.
This work provides a framework to develop and  analyze  preconditioned Riemannian gradient methods with quantifiable local convergence rate
to compute minimizer of the Gross-Pitaevskii energy functional. 
In addition, it can be applied to analyze all existing projected Sobolev gradient methods for minimizing the Gross-Pitaevskii energy functional,
and extended to similar problems such as computing minimizers of multi-component  Gross-Pitaevskii energy functional \cite{2024Riem}. {Finally, our theoretical analysis and numerical experiments show that appropriately reducing $\sigma_0$ improves the linear convergence rate. However, it remains an open question whether superlinear or even quadratic convergence can be achieved by taking $\sigma_0 \to 0^+$ (or setting $\sigma_0 = 0$) when the initial iterate is sufficiently close to the ground state $\phi_g$. This direction would naturally lead to second-order Riemannian optimization methods, such as adaptive regularized or standard Riemannian Newton methods, for the degenerate case, which we leave for future investigation.}
 \\
 \\
 \\

\appendix
\setcounter{equation}{0}
\section{Proof of {\bf Proposition } \ref{Prop1}}\label{Appendix_Pro}
\begin{proof} 
	For any $\phi\in\mathcal{S}$, we show that $i\phi$ and $i\mathcal{L}_z\phi$ are eigenfunctions of $E''(\phi)|_{T_{\phi}\mathcal{M}}$ with corresponding eigenvalue $\lambda_g$. The second order necessary condition shows that
	\begin{align*}
		\big\langle E''(\phi)v,v\big\rangle-\lambda_g(v,v)_{L^2}\ge0\quad \text{for all}\ v\in T_{\phi}\mathcal{M}.
	\end{align*}
	Taking curves $\gamma_1(t)=e^{it}\phi$ and $\gamma_2(t)=\phi(A_t\bm{x})$, we have identities $\big\|\gamma_i(t)\big\|_{L^2}^2\equiv\big\|\gamma_i(0)\big\|_{L^2}^2$ and $E(\gamma_i(t))\equiv E(\gamma_i(0))$ for $i=1,2$. The calculation of their second derivative reveals that
	\begin{align*}
		\frac{\text{d}^2}{\text{d}t^2}\big\|\gamma_i(t)\big\|_{L^2}^2&=2\big(\gamma'_i(t),\gamma'_i(t)\big)_{L^2}+2\big(\gamma''_i(t),\gamma_i(t)\big)_{L^2}=0,\\
		\frac{\text{d}^2}{\text{d}t^2}E(\gamma_i(t))&=\big\langle E''(\gamma_i(t))\gamma'_i(t),\gamma'_i(t)\big\rangle+\lambda_g\big( \gamma_i(t),\gamma''_i(t)\big)_{L^2}=0.
	\end{align*}
	Summing up, we obtain
	\begin{align*}
		\big\langle E''(\phi)\gamma'_i(0),\gamma'_i(0)\big\rangle-\lambda_g\big( \gamma'_i(0),\gamma'_i(0)\big)_{L^2}=0.
	\end{align*}
	For the Rayleigh quotient functional 
	\begin{align*}
		Q_{\phi}(v)=\big\langle E''(\phi)v,v\big\rangle\big/(v,v)_{L^2}\quad \text{for all}\ v\in T_{\phi_g}\mathcal{M}\backslash\{0\}, 
	\end{align*}
	we see that $\gamma'_i(0)$ corresponds to its minimum. Applying the first order necessary condition, we find that
	\begin{align*}
		E''(\phi)\gamma'_i(0)=\lambda_g\mathcal{I}\gamma'_i(0)\quad \text{on}\quad  T_{\phi}\mathcal{M}.
	\end{align*}
	Since $H_0^1(\mathcal{D})=\left(\left(\text{span}\left\{\phi\right\}\right)^{\bot}_{L^2}\cap H_0^1(\mathcal{D})\right)\oplus\text{span}\left\{\phi\right\}=T_{\phi}\mathcal{M}\oplus\text{span}\left\{\phi\right\}$, {and the spectral property on $T_\phi\mathcal{M}$ has already been established, we just need to verify that $v=\phi$ satisfies the eigenequation to extend the result to the full space $H_0^1(\mathcal{D})$.} It can be obtained by the following calculation
	\begin{align*}
		\left\langle E''(\phi)\gamma'_i(0),\phi\right\rangle
		&=\frac{\text{d}}{\text{d}t}\left( E(\gamma_i(t))+\int_{\mathcal{D}}\left(f(\rho_{\gamma_i})|\gamma_i(t)|^2-F(\rho_{\gamma_i})\right)\text{d}\bm{x}\right)\Bigg|_{t=0}\\
		&=\frac{\text{d}}{\text{d}t}\left( E(\phi)+\int_{\mathcal{D}}\left(f(\rho_{\phi})|\phi|^2-F(\rho_{\phi})\right)\text{d}\bm{x}\right)\Bigg|_{t=0}=0.
	\end{align*}
\end{proof}
\section{Proof of {\bf Proposition }  \ref{Prop2}}\label{Appendix_Pro1}
\begin{proof}
	First, for any $\phi\in\mathcal{S}$, we prove that the Rayleigh quotient functional $Q_\phi(\cdot)$ is bounded below and attains its minimum on $N_\phi\mathcal{M}$. Define:
	\begin{align*}
		\lambda_3 := \inf_{v \in N_{\phi}\mathcal{M}\backslash\{0\}} Q_{\phi}(v) = \inf_{\substack{v \in N_{\phi}\mathcal{M} \\ \|v\|_{L^2} = 1}} a(v,v).
	\end{align*}
	Let {$\{v_n\}_{n\in\mathbb{N}}\subseteq N_{\phi}\mathcal{M}$} be a sequence such that:
	\begin{align*}
		\|v_n\|_{L^2} = 1\quad\text{and}\quad\lim\limits_{n\to\infty}a(v_n, v_n)=\lambda_3.
	\end{align*}
	By the coercivity of $\mathcal{H}_0$ and $f\ge0$, we obtain the following lower bound estimate for the bilinear form $a(\cdot,\cdot)$
	\begin{align*}
		a(v,v) = \langle E''(\phi)v, v \rangle 
		&= \langle \mathcal{H}_0 v, v \rangle + (f(\rho_\phi)v, v)_{L^2} + \big\langle f'(\rho_\phi)(|\phi|^2 + \phi^2 \overline{\,\cdot\,})v, v\big\rangle \\
		&\ge C\|v\|_{H^1}^2 + \big\langle f'(\rho_\phi)(|\phi|^2 + \phi^2 \overline{\,\cdot\,})v, v\big\rangle.
	\end{align*}
	Using {\bf (A3)}, H\"older's inequality, the Gagliardo-Nirenberg inequality, and the weighted Young inequality, we derive
	\begin{align}\label{Positive1}
		\nonumber \big(f'(\rho_\phi)(|\phi|^2 + \phi^2 \overline{\,\cdot\,})v, v\big)_{L^2} 
		&\le C\|\phi\|_{L^6}^{1+\theta} \|v\|_{L^p}^2 
		\le C_\phi \|v\|_{L^2}^{2-(1-2/p)d} \|v\|_{H^1}^{(1-2/p)d} \\
		&\le C_\phi \left( \varepsilon^{-\frac{(1-2/p)d}{2-(1-2/p)d}} \|v\|_{L^2}^2 + \varepsilon \|v\|_{H^1}^2 \right),
	\end{align}
	where $p = 12/(5 - \theta) \in [12/5, 6)$. Taking $\varepsilon = C/(2C_\phi)$, we finally obtain:
	$$
	a(v, v) = \langle E''(\phi)v, v \rangle \ge \frac{C}{2} \|v\|_{H^1}^2 - C_\phi \|v\|_{L^2}^2.
	$$
	With this lower bound estimate, we have
	$$
	C \|v_n\|_{H^1}^2 \leq a(v_n, v_n) + C_\phi \leq \lambda_3 + \varepsilon_n + C_\phi \to \lambda_3 + C_\phi,
	$$
	which implies $\|v_n\|_{H^1} \leq C + C_\phi < \infty$, i.e., the sequence {$\{v_n\}_{n\in\mathbb{N}}\subseteq N_{\phi}\mathcal{M}$} is bounded in $H_0^1(\mathcal{D})$. Since $H_0^1(\mathcal{D})$ is a reflexive Banach space, there exists a subsequence (still denoted by $v_n$) and some $v^* \in H_0^1(\mathcal{D})$ such that
	$$
	v_n \rightharpoonup v^* \quad \text{weakly in } H_0^1(\mathcal{D}).
	$$
	Moreover, by the compact embedding $H_0^1(\mathcal{D}) \subset\subset L^2(\mathcal{D})$, we have
	$$
	v_n \to v^* \quad \text{strongly in } L^2(\mathcal{D}).
	$$
	It then follows that
	\begin{align*}
		\|v^*\|_{L^2} &= \lim_{n \to \infty} \|v_n\|_{L^2} = 1, \\
		(i\phi, v^*)_{L^2} &= \lim_{n\to \infty} (i\phi, v_n)_{L^2} = 0, \\
		(i\mathcal{L}_z\phi, v^*)_{L^2} &= \lim_{n\to \infty} (i\mathcal{L}_z\phi, v_n)_{L^2} = 0.
	\end{align*}
	This shows that $v^* \in N_{\phi}\mathcal{M} \setminus \{0\}$. Consider the functional $F(v) = a(v,v)$. Since the bilinear form $a(\cdot,\cdot)$ is symmetric and coercive, $F$ is convex and coercive, and is defined on $H_0^1(\mathcal{D})$. By a classical result in functional analysis:  
	a coercive, proper (not identically $+\infty$), and convex functional on a reflexive Banach space is weakly lower semicontinuous.  
	Therefore, we have
	$$
	a(v^*, v^*) \leq \liminf_{n \to \infty} a(v_n, v_n) = \lambda_3.
	$$
	On the other hand, since $\|v^*\|_{L^2} = 1$, by the definition of $\lambda_3$, we also have
	$$
	a(v^*, v^*) \geq \lambda_3.
	$$
	Combining both inequalities, we conclude
	$$
	a(v^*, v^*) = \lambda_3, \quad \|v^*\|_{L^2} = 1 \quad \Rightarrow \quad Q_{\phi}(v^*) = \lambda_3.
	$$
	This shows that the infimum $\lambda_3$ is attained by $v^* \in N_{\phi}\mathcal{M}$, which completes the proof. According to {\bf Definition~\ref{Morse-Bott-P}}, for any $\phi \in \mathcal{S}$, we have
	\begin{align}\label{lambda1}
		Q_{\phi}(v) \ge {\min_{v \in N_{\phi}\mathcal{M}\setminus \{0\}}} Q_{\phi}(v) := \lambda_3 > \lambda_g, \quad {\text{for all}}\;\, v \in N_{\phi}\mathcal{M} \setminus \{0\}.
	\end{align}
	The proof of coercivity on $N_{\phi}\mathcal{M}$ follows similarly to \cite{2024Convergence}, where a case-by-case analysis can be used to establish the coercivity (see \cite[{\bf Lemma 2.3}]{2024Convergence}). Specifically, we proceed as follows: for all $v\in N_{\phi}\mathcal{M}$,
	\begin{itemize}
		\item If $\|v\|^2_{H^1}>\frac{2C_{\phi}+2\lambda_g}{C}\|v\|^2_{L^2}$, then $-\left(C_{\phi}+\lambda_g\right)\|v\|^2_{L^2}>-\frac{C}{2}\|v\|^2_{H^1}$ and therefore
		\begin{align*}
			\left\langle \big(E''(\phi) - \lambda_g \mathcal{I}\big)v, v \right\rangle\ge C\|v\|^2_{H^1}-\left(C_{\phi}+\lambda_g\right)\|v\|^2_{L^2}\ge\frac{C}{2}\|v\|^2_{H^1}.
		\end{align*}
		\item If $\|v\|^2_{H^1}\le\frac{C_{\phi}+2\lambda_g}{C}\|v\|^2_{L^2}$, then $\|v\|^2_{L^2}\ge\frac{C}{C_{\phi}+2\lambda_g}\|v\|^2_{H^1}$, which yields
		\begin{align*}
			\left\langle \big(E''(\phi) - \lambda_g \mathcal{I}\big)v, v \right\rangle\ge\left(\lambda_3-\lambda_g\right)\|v\|^2_{L^2}\ge\frac{C(\lambda_3-\lambda_g)}{2C_{\phi}+2\lambda_g}\|v\|^2_{H^1}.
		\end{align*}
	\end{itemize}
	\noindent	This proof is completed.
\end{proof}

\section{Proof of {\bf Proposition }  \ref{Property-of-E''}}\label{Appendix-A}
\begin{proof}
	\begin{itemize}[label={},left=0pt]
		\item $(i)$ Due to the phase shift and coordinate rotation invariance of the Gross-Pitaevskii energy functional $E$, for any $\phi, v \in H_0^1(\mathcal{D})$, we have
		\begin{align*}
			E(I_{\alpha}^{\beta}(\phi + t v)) \equiv E(\phi + t v), \quad {\text{for all}}\; \alpha, \beta \in [-\pi, \pi) \; \text{and} \;  t \in \mathbb{R}.
		\end{align*}
		This implies
		\begin{align*}
			&\frac{\mathrm{d}^2}{\mathrm{d}t^2} E(I_{\alpha}^{\beta}(\phi + t v)) \Bigg|_{t=0} = \frac{\mathrm{d}^2}{\mathrm{d}t^2} E(\phi + t v) \Bigg|_{t=0} \\
			\Longrightarrow\ &\left\langle E''(I_{\alpha}^{\beta}\phi) I_{\alpha}^{\beta}v, I_{\alpha}^{\beta}v \right\rangle = \left\langle E''(\phi) v, v \right\rangle.
		\end{align*}
		\item $(ii)$ Using the continuity of $\mathcal{H}_{\phi}$, H\"older's inequality, and the Sobolev embedding $H_0^1(\mathcal{D}) \subset L^{p}(\mathcal{D})$ for $d \le 3$ and $1 \le p \le 6$, we obtain
	\begin{align*}
			\left|\left\langle E''(\phi)u, v\right\rangle\right|
			&= \left|\left\langle\mathcal{H}_{0}u, v\right\rangle + \big\langle f(\rho_{\phi})u, v\big\rangle + \left\langle f'(\rho_{\phi})\big(|\phi|^2 + \phi^2 \overline{\,\cdot\,}\big)u, v\right\rangle\right| \\
			&\le C_{\phi}\|u\|_{H^1}\|v\|_{H^1} + C\|\phi\|^{1+\theta}_{L^{6}}\|u\|_{H^1}\|v\|_{H^1} \le C_{\phi}\|u\|_{H^1}\|v\|_{H^1}.
		\end{align*}
		\item $(iii)$ Using the inequality $|a^{1+\theta}-b^{1+\theta}|\le C(a^{\theta}+b^{\theta})|a-b|$ for all $a,b\ge0$, we have
		\begin{align}\label{f-continu}
			\left|f(\rho_{\phi})-f(\rho_{\psi})\right|=\left|\int_{\rho_{\psi}}^{\rho_{\phi}}f'(s)\;\text{d}s\right|\le C\left(|\phi|^{\theta}+|\psi|^{\theta}\right)|\phi-\psi|.
		\end{align}
		Using {\bf (A3)} again, we get
		\begin{align}\label{f'-continu}
			\nonumber&\left|f'(\rho_{\phi})|\phi|^2\frac{\phi^2}{|\phi|^2}-f'(\rho_{\psi})|\psi|^2\frac{\psi^2}{|\psi|^2}\right|\\
			\nonumber&\qquad\qquad\le \left|f'(\rho_{\phi})|\phi|^2\frac{\phi^2}{|\phi|^2}-f'(\rho_{\psi})|\psi|^2\frac{\phi^2}{|\phi|^2}\right|+\left|f'(\rho_{\psi})|\psi|^2\left(\frac{\phi^2}{|\phi|^2}-\frac{\psi^2}{|\psi|^2}\right)\right|\\
			\nonumber&\qquad\qquad\le\left|f'(\rho_{\phi})|\phi|^2-f'(\rho_{\psi})|\psi|^2\right|+C|\psi|^{1+\theta}\left|\frac{\phi^2|\psi|^2-|\phi|^2\psi^2}{|\phi|^2|\psi|^2}\right|\\
			\nonumber&\qquad\qquad\le C\left(|\phi|^{\theta}+|\psi|^{\theta}\right)\left|\phi-\psi\right|+C|\psi|^{1+\theta}\left|\frac{{\phi\big(\overline{\psi}-\overline{\phi}\big)}+\overline{\phi}\left(\phi-\psi\right)}{|\phi||\psi|}\right|\\
			&\qquad\qquad\le C\left(|\phi|^{\theta}+|\psi|^{\theta}\right)\left|\phi-\psi\right|.
		\end{align}
		Using the above results, the H\"older inequality, and $H_0^1(\mathcal{D})\subset L^{p_0}(\mathcal{D})$ ($p_0=6/(4-\theta)\in\left[\frac{3}{2},6\right)$), our conclusion is as follows
		\begin{align}\label{E''-continu}
			\nonumber&\left|\left\langle \big(E''(\phi)-E''(\psi)\big)u,v\right\rangle\right|\\
		\nonumber&=\left|\left\langle\big(f(\rho_{\phi})-f(\rho_{\psi})+f'(\rho_{\phi})\big(|\phi|^2+(\phi)^2\overline{\cdot}\big)-f'(\rho_{\psi})\big(|\psi|^2+(\psi)^2\overline{\cdot}\big)\big)u,v\right\rangle\right|\\
			\nonumber&\le C\left\langle\big(|\phi|^{\theta}+|\psi|^{\theta}\big)\left|\phi-\psi\right||u|,|v|\right\rangle\\
			\nonumber&\le C\left(\|\phi\|^{\theta}_{L^6}+\|\psi\|^{\theta}_{L^6}\right)\|u\|_{L^6}\|v\|_{L^6}\|\phi-\psi\|_{L^{p_0}}\\
			&\le C_{\phi,\psi}\|u\|_{H^1}\|v\|_{H^1}\|\phi-\psi\|_{L^{p_0}}.
		\end{align}
		\item $(iv)$ Using the Taylor's formula and $(iii)$, the final conclusion is obtained as follow
		{\rm\begin{align}
				\nonumber E(\phi+v)&-E(\phi)-\left\langle E'(\phi),v\right\rangle\\
				\nonumber&=\int_0^1\int_0^t\left\langle \big(E''(\phi+sv)-E''(\phi)\big)v,v\right\rangle\text{d}s\text{d}t+\frac{1}{2}\big\langle E''(\phi)v,v\big\rangle\\
				&\le C_{\phi,v}\|v\|^3_{H^1}\int_0^1\int_0^ts\;\text{d}s\text{d}t+\frac{1}{2}\big\langle E''(\phi)v,v\big\rangle\le C_{\phi,v}\|v\|^3_{H^1}+\frac{1}{2}\big\langle E''(\phi)v,v\big\rangle.\label{Order3}
		\end{align}}
	\end{itemize}
\end{proof}

\section{Proof of {\bf Proposition }  \ref{Property-P}}\label{Appendix-B}
\begin{proof}
	\begin{itemize}[label={},left=0pt]
		\item $(i)$ Let us first prove that $0 < \mu \le L < \infty$ for $\phi = \phi_g$.  
		By \textbf{Proposition \ref{Prop2}}, \textbf{Proposition \ref{Property-of-E''}}-$(ii)$, and assumption \textbf{(A6)}-$(ii)$, we obtain the following estimates:
		
		\begin{itemize}
			\item for all nonzero $v \in N_{\phi_g}\mathcal{M}$,
			\begin{align*}
				\frac{\big\langle \big(E''(\phi_g)-\lambda_g\mathcal{I}\big)v,v\big\rangle}{\big\langle\mathcal{P}_{\phi_g}v,v\big\rangle}
				\ge \frac{C\|v\|_{H^1}^2}{\langle\mathcal{P}_{\phi_g}v,v\rangle}
				\ge \frac{C\|v\|_{H^1}^2}{C_{\phi_g}\|v\|_{H^1}^2}
				= \frac{C}{C_{\phi_g}} > 0;
			\end{align*}
			
			\item for all nonzero $v \in T_{\phi_g}\mathcal{M}$,
			\begin{align*}
				\frac{\big\langle \big(E''(\phi_g)-\lambda_g\mathcal{I}\big)v,v\big\rangle}{\big\langle\mathcal{P}_{\phi_g}v,v\big\rangle}
				\le \frac{C_{\phi_g}\|v\|_{H^1}^2}{\langle\mathcal{P}_{\phi_g}v,v\rangle}
				\le \frac{C_{\phi_g}\|v\|_{H^1}^2}{C\|v\|_{H^1}^2}
				= \frac{C_{\phi_g}}{C} < \infty.
			\end{align*}
		\end{itemize}
		 Consequently,
			\begin{align*}
				0 < \inf_{v \in N_{\phi_g}\mathcal{M} \setminus \{0\}} 
				\frac{\big\langle (E''(\phi_g) - \lambda_g \mathcal{I}) v, v \big\rangle}{\langle \mathcal{P}_{\phi_g} v, v \rangle}
				= \mu
				\le
				L = \sup_{v \in T_{\phi_g}\mathcal{M} \setminus \{0\}} 
				\frac{\big\langle (E''(\phi_g) - \lambda_g \mathcal{I}) v, v \big\rangle}{\langle \mathcal{P}_{\phi_g} v, v \rangle}
				< \infty.
		\end{align*}
		
		Next, by \textbf{Proposition \ref{Property-of-E''}}-$(i)$ and \textbf{(A6)}-$(i)$, for any $\phi \in \mathcal{S}$, i.e., $\phi = I_\alpha^\beta \phi_g$, we have the identity
		\begin{align*}
			\frac{\big\langle (E''(\phi) - \lambda_g \mathcal{I}) v, v \big\rangle}{\langle \mathcal{P}_\phi v, v \rangle}
			=
			\frac{\big\langle (E''(\phi_g) - \lambda_g \mathcal{I}) I_{-\alpha}^{-\beta} v,\, I_{-\alpha}^{-\beta} v \big\rangle}{\langle \mathcal{P}_{\phi_g} I_{-\alpha}^{-\beta} v,\, I_{-\alpha}^{-\beta} v \rangle}.
		\end{align*}
		
		Moreover, if $v \in N_\phi \mathcal{M}$ (resp. $v \in T_\phi \mathcal{M}$), then $I_{-\alpha}^{-\beta} v \in N_{\phi_g} \mathcal{M}$ (resp. $I_{-\alpha}^{-\beta} v \in T_{\phi_g} \mathcal{M}$). Therefore, the same lower and upper bounds carry over to all $\phi \in \mathcal{S}$, yielding
			\begin{align*}
				0 < \inf_{v \in N_{\phi}\mathcal{M} \setminus \{0\}} 
				\frac{\big\langle (E''(\phi) - \lambda_g \mathcal{I}) v, v \big\rangle}{\langle \mathcal{P}_{\phi} v, v \rangle}
				= \mu
				\le
				L = \sup_{v \in T_{\phi}\mathcal{M} \setminus \{0\}} 
				\frac{\big\langle (E''(\phi) - \lambda_g \mathcal{I}) v, v \big\rangle}{\langle \mathcal{P}_{\phi} v, v \rangle}
				< \infty.
		\end{align*}
		\item $(ii)$ Noting that
		\begin{align*}
			\big\|\mathcal{H}_{\phi}v\big\|_{H^{-1}}=\sup\limits_{u\in H_0^1(\mathcal{D})}\frac{\big\langle \mathcal{H}_{\phi}v,u\big\rangle}{\quad\|u\|_{H^1}}\le C_{\phi}\|v\|_{H^1},
		\end{align*}
		we have $\big\|\mathcal{P}_{\phi}^{-1}\mathcal{H}_{\phi}v\big\|_{H^1}\le C	\big\|\mathcal{H}_{\phi}v\big\|_ {H^{-1}}\le C_{\phi}\|v\|_{H^1}$. Using {\bf (A6)}-$(iv)$, \eqref{f'-continu}, and $L^q(\mathcal{D})\subset L^p(\mathcal{D})$ for $ 1\le p\le q$, the estimation is derived
		\begin{align*}
			\big\|\mathcal{P}_{\phi}^{-1}(\mathcal{H}_{\phi}-\mathcal{P}_{\phi})v\big\|_{H^1}&= \frac{1}{2}\left\|\mathcal{P}_{\phi}^{-1}\left(E''(\phi)- \mathcal{P}_{\phi}-f'(\rho_{\phi})\big(|\phi|^2+\phi^2\overline{\cdot}\big)\right)v\right\|_{H^1}\\
			&\le C_{\phi}\left(\left\|\mathcal{P}_{\phi}^{-1}\left(E''(\phi)- \mathcal{P}_{\phi}\right)v\right\|_{H^1}+\left\|\big(f'(\rho_{\phi})\big(|\phi|^2+\phi^2\overline{\cdot}\big)\big)v\right\|_{L^{6/5}}\right) \\
			&\le C_{\phi}\left(\|v\|_{L^{p_2}}+\|v\|_{L^{p_0}}\right)\le C_{\phi}\|v\|_{L^{p}},
		\end{align*}
		where $p=\max\{p_0,p_2\}\in[1,6)$, and the penultimate inequality is a consequence of the following result:
		\begin{align*}
			\left\|\mathcal{P}_{\phi}^{-1}\big(f'(\rho_{\phi})\big(|\phi|^2+\phi^2\overline{\cdot}\big)\big)v\right\|^2_{H^1}&\le C_{\phi}\left\|\mathcal{P}_{\phi}^{-1}\big(f'(\rho_{\phi})\big(|\phi|^2+\phi^2\overline{\cdot}\big)\big)v\right\|^2_{\mathcal{P}_{\phi}}\\
			&=C_{\phi}\left\langle\big(f'(\rho_{\phi})\big(|\phi|^2+\phi^2\overline{\cdot}\big)\big)v ,\mathcal{P}_{\phi}^{-1}\big(f'(\rho_{\phi})\big(|\phi|^2+\phi^2\overline{\cdot}\big)\big)v\right\rangle\\
			&\le C_{\phi}\left\|\mathcal{P}_{\phi}^{-1}\big(f'(\rho_{\phi})\big(|\phi|^2+\phi^2\overline{\cdot}\big)\big)v\right\|_{L^6}\left\|\big(f'(\rho_{\phi})\big(|\phi|^2+\phi^2\overline{\cdot}\big)\big)v\right\|_{L^{6/5}}\\
			&\le C_{\phi}\left\|\mathcal{P}_{\phi}^{-1}\big(f'(\rho_{\phi})\big(|\phi|^2+\phi^2\overline{\cdot}\big)\big)v\right\|_{H^1}\left\||\phi|^{\theta+1}|v|\right\|_{L^{6/5}}\\
			&\le C_{\phi}\left\|\mathcal{P}_{\phi}^{-1}\big(f'(\rho_{\phi})\big(|\phi|^2+\phi^2\overline{\cdot}\big)\big)v\right\|_{H^1}\|\phi\|^{\theta+1}_{L^{6}}\|v\|_{L^{p_0}}.
		\end{align*}
		\item $(iii)$  This is analogous to $\mathcal{P}_{\phi}=-\frac{1}{2}\Delta$ (see \cite[{\bf Lemma 5.2}]{2024On}). According to the identity
		\begin{align*}
			\nabla_{\mathcal{P}}^{\mathcal{R}}E(\phi)-\nabla_{\mathcal{P}}^{\mathcal{R}}E(\psi)&=\text{Proj}^{\mathcal{P}_{\phi}}_{\phi}\left(\mathcal{P}^{-1}_{\phi}\mathcal{H}_{\phi}\phi-\mathcal{P}^{-1}_{\psi}\mathcal{H}_{\psi}\psi\right)\\
			&+\left(\text{Proj}_{\phi}^{\mathcal{P}_{\phi}}-\text{Proj}_{\psi}^{\mathcal{P}_{\psi}}\right)\mathcal{P}^{-1}_{\psi}\mathcal{H}_{\psi}\psi,
		\end{align*}
		we can get the continuity of $\nabla^{\mathcal{R}}_{\mathcal{P}} E(\phi)$ by proving that $\text{Proj}^{\mathcal{P}_{\phi}}_{\phi}$ and $\mathcal{P}_{\phi}^{-1}\mathcal{H}_{\phi}\phi$ are continuous. The continuity of $\mathcal{P}_{\phi}^{-1}\mathcal{H}_{\phi}\phi$ is considered first.
		By the direct calculation, we have
		\begin{align}
			\nonumber\mathcal{P}_{\phi}^{-1}\mathcal{H}_{\phi}\phi-\mathcal{P}_{\psi}^{-1}\mathcal{H}_{\psi}\psi&=\big(\mathcal{P}_{\phi}^{-1}-\mathcal{P}_{\psi}^{-1}\big)\mathcal{H}_{\phi}\phi\\
			&+\mathcal{P}_{\psi}^{-1}\big(\mathcal{H}_{\phi}-\mathcal{H}_{\psi}\big)\phi+\mathcal{P}_{\psi}^{-1}\left(\mathcal{H}_{\psi}-\mathcal{P}_{\psi}\right)(\phi-\psi)+(\phi-\psi).\label{four-part}
		\end{align}
		Based on {\bf (A6)}-$(ii)$ and -$(iii)$, and {\bf Proposition \ref{Property-P}}-$(ii)$, the following inequality holds
		\begin{align}\label{1-Part}
			\nonumber\left\|\big(\mathcal{P}_{\phi}^{-1}-\mathcal{P}_{\psi}^{-1}\big)\mathcal{H}_{\phi}\phi\right\|^2_{H^1}&=\big\|\mathcal{P}_{\psi}^{-1}\big(\mathcal{P}_{\psi}-\mathcal{P}_{\phi}\big)\mathcal{P}_{\phi}^{-1}\mathcal{H}_{\phi}\phi\big\|^2_{H^1}\\
			\nonumber&\le C_{\phi}\big\|\mathcal{P}_{\psi}^{-1}\big(\mathcal{P}_{\psi}-\mathcal{P}_{\phi}\big)\mathcal{P}_{\phi}^{-1}\mathcal{H}_{\phi}\phi\big\|^2_{\mathcal{P}_{\psi}}\\
			\nonumber&=C_{\phi}\left\langle\big(\mathcal{P}_{\psi}-\mathcal{P}_{\phi}\big)\mathcal{P}_{\phi}^{-1}\mathcal{H}_{\phi}\phi,\mathcal{P}_{\psi}^{-1}\big(\mathcal{P}_{\psi}-\mathcal{P}_{\phi}\big)\mathcal{P}_{\phi}^{-1}\mathcal{H}_{\phi}\phi\right\rangle\\
			\nonumber&\le C_{\phi} \big\|\mathcal{P}_{\psi}^{-1}\big(\mathcal{P}_{\psi}-\mathcal{P}_{\phi}\big)\mathcal{P}_{\phi}^{-1}\mathcal{H}_{\phi}\phi\big\|_{H^1}\big\|\mathcal{P}_{\phi}^{-1}\mathcal{H}_{\phi}\phi\big\|_{H^1}\|\phi-\psi\|_{L^{p_1}}\\
			&\le C_{\phi}\big\|\big(\mathcal{P}_{\phi}^{-1}-\mathcal{P}_{\psi}^{-1}\big)\mathcal{H}_{\phi}\phi\big\|_{H^1}\|\phi-\psi\|_{L^{p_1}}.
		\end{align}
		This suggests that $\big\|\big(\mathcal{P}_{\phi}^{-1}-\mathcal{P}_{\psi}^{-1}\big)\mathcal{H}_{\phi}\phi\big\|_{H^1}\le C_{\phi}\|\phi-\psi\|_{L^{p_1}}$.
		 $\mathcal{P}_{\psi}^{-1}\big(\mathcal{H}_{\phi}-\mathcal{H}_{\psi}\big)\phi$, recalling  \eqref{f-continu}, we derive
		\begin{align}\label{2-Part}
			\nonumber\big\|\mathcal{P}_{\psi}^{-1}\big(\mathcal{H}_{\phi}-\mathcal{H}_{\psi}\big)\phi\big\|_{H^1}&=	\left\|\mathcal{P}_{\psi}^{-1}\big(f(\rho_{\phi})-f(\rho_{\psi})\big)\phi\right\|_{H^1}\le C_{\phi}\left\|\big(f(\rho_{\phi})-f(\rho_{\psi})\big)\phi\right\|_{L^{6/5}}\\
			&\le C_{\phi}\left\||\phi|\big(|\phi|^{\theta}+|\psi|^{\theta}\big)|\phi-\psi|\right\|_{L^{6/5}}\le C_{\phi}\|\phi-\psi\|_{L^{p_0}}.
		\end{align}
		{\bf Proposition \ref{Property-P}}-$(ii)$ shows directly that
		\begin{align}\label{3-Part}
			\big\|\mathcal{P}_{\psi}^{-1}\big(\mathcal{H}_{\psi}-\mathcal{P}_{\psi}\big)(\phi-\psi)\big\|_{H^1}\le C_{\phi}\left(\|\phi-\psi\|_{L^{p_0}}+\|\phi-\psi\|_{L^{p_2}}\right).
		\end{align}
		In conjunction with \eqref{four-part}-\eqref{3-Part}, $L^q(\mathcal{D})\subset L^p(\mathcal{D})\; (1\le p\le q)$, and $H^1(\mathcal{D})\subset L^p(\mathcal{D})\; (1\le p\le6)$, we get 
		\begin{align}
			\label{L-H}
			\big\|\mathcal{P}_{\phi}^{-1}\mathcal{H}_{\phi}\phi-\mathcal{P}_{\psi}^{-1}\mathcal{H}_{\psi}\psi\big\|_{H^1}&\le C_{\phi}\|\phi-\psi\|_{H^1},\\
			\big\|\mathcal{P}_{\phi}^{-1}\mathcal{H}_{\phi}\phi-\phi-\mathcal{P}_{\psi}^{-1}\mathcal{H}_{\psi}\psi+\psi\big\|_{H^1}&\le C_{\phi}\|\phi-\psi\|_{L^p},\label{L-H1}
		\end{align}
		where $p=\max\{p_0,p_1,p_2\}\in[1,6)$. Then, we consider the continuity of $\text{Proj}^{\mathcal{P}_{\phi}}_{\phi}$.
		For all $v\in H_0^1(\mathcal{D})$, we have
		\begin{align}\label{Proj-diff}
			\nonumber\Big(\text{Proj}^{\mathcal{P}_{\phi}}_{\phi}-\text{Proj}^{\mathcal{P}_{\psi}}_{\psi}\Big)v&=\frac{(\phi,v)_{L^2}}{\big(\phi,\mathcal{P}_{\phi}^{-1}\mathcal{I}\phi\big)_{L^2}}\mathcal{P}^{-1}_{\phi}\mathcal{I}\phi-\frac{(\psi,v)_{L^2}}{\big(\psi,\mathcal{P}_{\psi}^{-1}\mathcal{I}\psi\big)_{L^2}}\mathcal{P}^{-1}_{\psi}\mathcal{I}\psi\\
			\nonumber&=\frac{(\phi,v)_{L^2}}{\big(\phi,\mathcal{P}_{\phi}^{-1}\mathcal{I}\phi\big)_{L^2}}\big(\mathcal{P}^{-1}_{\phi}\mathcal{I}\phi-\mathcal{P}^{-1}_{\psi}\mathcal{I}\psi\big)\\&+\Bigg(\frac{(\phi,v)_{L^2}}{\big(\phi,\mathcal{P}_{\phi}^{-1}\mathcal{I}\phi\big)_{L^2}}-\frac{(\psi,v)_{L^2}}{\big(\psi,\mathcal{P}_{\psi}^{-1}\mathcal{I}\psi\big)_{L^2}}\Bigg)\mathcal{P}^{-1}_{\psi}\mathcal{I}\psi.
		\end{align}
		Similarly, by replacing $\mathcal{H}_{\phi}$ and $\mathcal{H}_{\psi}$ with $\mathcal{I}$ in \eqref{four-part}-\eqref{3-Part}, and combining these with {\bf Proposition \ref{Property-P}}-$(ii)$, we derive the continuity of $\mathcal{P}_{\phi}^{-1}\mathcal{I}\phi$ as follows
		\begin{align}\label{L-I}
			\nonumber\big\|\mathcal{P}_{\phi}^{-1}\mathcal{I}\phi-\mathcal{P}_{\psi}^{-1}\mathcal{I}\psi\big\|_{H^1}&\le \nonumber\big\|\big(\mathcal{P}_{\phi}^{-1}-\mathcal{P}_{\psi}^{-1}\big)\mathcal{I}\phi\big\|_{H^1}+\big\|\mathcal{P}_{\psi}^{-1}\mathcal{I}(\phi-\psi)\big\|_{H^1}\\ 
			&\le C_{\phi}\left(\|\phi-\psi\|_{L^{p_0}}+\|\phi-\psi\|_{L^{p_1}}+\|\phi-\psi\|_{L^{p_2}}+\|\phi-\psi\|_{L^2}\right)\\
			\nonumber&\le C_{\phi}\|\phi-\psi\|_{H^1}.
		\end{align}
		Calculating directly yields the following results
		\begin{align}\label{Orthogonalization-Coefficient-diff}
			\nonumber\frac{(\phi,v)_{L^2}}{\big(\phi,\mathcal{P}_{\phi}^{-1}\mathcal{I}\phi\big)_{L^2}}-\frac{(\psi,v)_{L^2}}{\big(\psi,\mathcal{P}_{\psi}^{-1}\mathcal{I}\psi\big)_{L^2}}&=\frac{(\phi,v)_{L^2}-(\psi,v)_{L^2}}{\big(\phi,\mathcal{P}_{\phi}^{-1}\mathcal{I}\phi\big)_{L^2}}\\
			&-\frac{(\psi,v)_{L^2}\big(\big(\phi,\mathcal{P}_{\phi}^{-1}\mathcal{I}\phi\big)_{L^2}-\big(\psi,\mathcal{P}_{\psi}^{-1}\mathcal{I}\psi\big)_{L^2}\big)}{\big(\phi,\mathcal{P}_{\phi}^{-1}\mathcal{I}\phi\big)_{L^2}\big(\psi,\mathcal{P}_{\psi}^{-1}\mathcal{I}\psi\big)_{L^2}}.
		\end{align}
		Combining Cauchy’s inequality and \eqref{L-I} results in
		\begin{align}
			\left|(\phi,v)_{L^2}-(\psi,v)_{L^2}\right|&\le \|v\|_{L^2}\|\phi-\psi\|_{L^2}\\
			\nonumber\left|\big(\phi,\mathcal{P}_{\phi}^{-1}\mathcal{I}\phi\big)_{L^2}-\big(\psi,\mathcal{P}_{\psi}^{-1}\mathcal{I}\psi\big)_{L^2}\right|&=\left|\big(\phi,\mathcal{P}_{\phi}^{-1}\mathcal{I}\phi-\mathcal{P}_{\psi}^{-1}\mathcal{I}\psi\big)_{L^2}+\big(\phi-\psi,\mathcal{P}_{\psi}^{-1}\mathcal{I}\psi\big)_{L^2}\right|\\
			&\le C_{\phi}\left(\|\phi-\psi\|_{L^{p_1}}+\|\phi-\psi\|_{L^2}\right)\label{lamb-denominator-dif}
		\end{align}
		Using the above inequality, we derive
		\begin{align}
			\nonumber\big(\psi,\mathcal{P}_{\psi}^{-1}\mathcal{I}\psi\big)_{L^2}&\ge\big(\phi,\mathcal{P}_{\phi}^{-1}\mathcal{I}\phi\big)_{L^2}-\big|\big(\phi,\mathcal{P}_{\phi}^{-1}\mathcal{I}\phi\big)_{L^2}-\big(\psi,\mathcal{P}_{\psi}^{-1}\mathcal{I}\psi\big)_{L^2}\big|\\
			&\ge\big(\phi,\mathcal{P}_{\phi}^{-1}\mathcal{I}\phi\big)_{L^2}-C_{\phi}\left(\|\phi-\psi\|_{L^{p_1}}+\|\phi-\psi\|_{L^2}\right)\\
			\nonumber&\ge\big(\phi,\mathcal{P}_{\phi}^{-1}\mathcal{I}\phi\big)_{L^2}-C_{\phi}\|\phi-\psi\|_{H^1}.
		\end{align}
		Since $\mathcal{P}_{\phi}^{-1}\mathcal{I}\phi=0$ if and only if $\phi=0$, then there exists a sufficiently small $\sigma$ such that for all $\psi \in \mathcal{B}_{\sigma}(\phi)$,
		\begin{align}\label{Positive-denominator}
			\big(\psi,\mathcal{P}_{\psi}^{-1}\mathcal{I}\psi\big)_{L^2}\ge C>0.
		\end{align} 
		By \eqref{Orthogonalization-Coefficient-diff}-\eqref{Positive-denominator}, for all $\psi \in \mathcal{B}_{\sigma}(\phi)$, we get
		\begin{align}\label{L-Orthogonalization-Coefficient}
			\Bigg|\frac{(\phi,v)_{L^2}}{\big(\phi,\mathcal{P}_{\phi}^{-1}\mathcal{I}\phi\big)_{L^2}}-\frac{(\psi,v)_{L^2}}{\big(\psi,\mathcal{P}_{\psi}^{-1}\mathcal{I}\psi\big)_{L^2}}\Bigg|\le C_{\phi}\|v\|_{L^2}\left(\|\phi-\psi\|_{L^{p_1}}+\|\phi-\psi\|_{L^2}\right).
		\end{align}
		Hence, the continuity of $\text{Proj}^{\mathcal{P}_{\phi}}_{\phi}$ is derived through \eqref{Proj-diff}, \eqref{L-I} and \eqref{L-Orthogonalization-Coefficient}, i.e., for all $v\in H_0^1(\mathcal{D})$
		\begin{align}\label{L-Proj}
			\left\|\left(\text{Proj}^{\mathcal{P}_{\phi}}_{\phi}-\text{Proj}^{\mathcal{P}_{\psi}}_{\psi}\right)v\right\|_{H^1}&\le C_{\phi}\|v\|_{L^2}\left(\|\phi-\psi\|_{L^{p_1}}+\|\phi-\psi\|_{L^2}\right)\\
			\nonumber&\le C_{\phi}\|v\|_{L^2}\|\phi-\psi\|_{H^1}.
		\end{align}
		The local Lipschitz continuity of Riemannian gradient is also obtained by
		\begin{align*}
			\Big\|\nabla^{\mathcal{R}}_{\mathcal{P}} E&(\phi)-\nabla^{\mathcal{R}}_{\mathcal{P}}E(\psi)\Big\|_{H^1}=\left\|\text{Proj}^{\mathcal{P}_{\phi}}_{\phi}\mathcal{P}^{-1}_{\phi}\mathcal{H}_{\phi}\phi-\text{Proj}^{\mathcal{P}_{\psi}}_{\psi}\mathcal{P}^{-1}_{\psi}\mathcal{H}_{\psi}\psi\right\|_{H^1}\\
			&\le\left\|\left(\text{Proj}^{\mathcal{P}_{\phi}}_{\phi}-\text{Proj}^{\mathcal{P}_{\psi}}_{\psi}\right)\mathcal{P}^{-1}_{\phi}\mathcal{H}_{\phi}\phi\right\|_{H^1}+\left\|\text{Proj}^{\mathcal{P}_{\psi}}_{\psi}\left(\mathcal{P}^{-1}_{\phi}\mathcal{H}_{\phi}\phi-\mathcal{P}^{-1}_{\psi}\mathcal{H}_{\psi}\psi\right)\right\|_{H^1}\\
			&\le C_{\phi}\|\phi-\psi\|_{H^1}.
		\end{align*}
		Then, based on the identity
		\begin{align*}
			\lambda_{\phi}-\lambda_{\psi}&=\frac{(\phi,\phi)_{L^2}}{\big(\phi,\mathcal{P}_{\phi}^{-1}\mathcal{I}\phi\big)_{L^2}}-\frac{(\psi,\psi)_{L^2}}{\big(\psi,\mathcal{P}_{\psi}^{-1}\mathcal{I}\psi\big)_{L^2}}+\frac{(\phi,\mathcal{P}^{-1}_{\phi}\mathcal{H}_{\phi}\phi-\phi)_{L^2}}{\big(\phi,\mathcal{P}_{\phi}^{-1}\mathcal{I}\phi\big)_{L^2}}\\&-\frac{(\psi,\mathcal{P}^{-1}_{\phi}\mathcal{H}_{\phi}\phi-\phi)_{L^2}}{\big(\psi,\mathcal{P}_{\psi}^{-1}\mathcal{I}\psi\big)_{L^2}}+\frac{(\psi,\mathcal{P}^{-1}_{\phi}\mathcal{H}_{\phi}\phi-\phi-\mathcal{P}^{-1}_{\psi}\mathcal{H}_{\psi}\psi-\psi)_{L^2}}{\big(\psi,\mathcal{P}_{\psi}^{-1}\mathcal{I}\psi\big)_{L^2}},
		\end{align*}
		\eqref{L-H1}, \eqref{L-I}, and \eqref{L-Orthogonalization-Coefficient}, the local Lipschitz continuity of $\lambda_{\phi}$ is proved
		\begin{align*}
			\big|\lambda_{\phi}-\lambda_{\psi}\big|\le C_{\phi}\|\phi-\psi\|_{L^p},
		\end{align*}
		where $p=\max\{p_0,p_1,p_2,2\}\in[1,6)$. 
		\item $(iv)$ The proof can be found in \cite[{\bf Lemma 4.3}]{2024On}.
		Using the orthogonality $(\phi,v)_{L^2}=0$, we directly get
		\begin{align}\label{c_n}
			\nonumber\mathfrak{R}_{\phi}(tv)-(\phi+tv)&=\left(\frac{1}{\;\;\|\phi+tv\|_{L^2}}-1\right)(\phi+tv)=\left(\frac{1}{\sqrt{1+t^2\|v\|^2_{L^2}}}-1\right)(\phi+tv)\\
			&=-\frac{t^2\|v\|^2_{L^2}}{\sqrt{1+t^2\|v\|^2_{L^2}}\Big(1+\sqrt{1+t^2\|v\|^2_{L^2}}\Big)}\big(\phi+tv\big),\\
			\nonumber	\Longrightarrow\;\big|\mathfrak{R}_{\phi}(tv)&-(\phi+tv)\big|\le\frac{1}{2}t^2\|v\|^2_{L^2}|\phi+tv|.
		\end{align}
	\end{itemize}
\end{proof}
\section{On the Form of the Second-Order Sufficient Condition}\label{Appendix-SOSC}
In this appendix, we explain why the second-order sufficient condition for the GP energy functional takes the form given in \eqref{SOSC}. The second-order sufficient condition that is commonly known is of the following form:
\begin{align*}
	\langle (E''(\phi_g) - \lambda_g \mathcal{I}) v, v \rangle > 0, \quad {\text{for all}}\; v \in T_{\phi_g}\mathcal{M} \setminus \{0\}.
\end{align*}
In finite dimensions, this condition is equivalent to \eqref{SOSC} precisely because the unit sphere is compact, and this compactness ensures that the above condition guarantees a local minimum. However, in infinite-dimensional spaces, this is no longer the case. We construct a counterexample below to show that the second-order sufficient condition should be taken in the form of \eqref{SOSC}. {For simplicity, we still consider a minimizer $\phi_g$ of a sufficiently smooth functional on the $L^2$-normalization constraint manifold $\mathcal{M}$, and we continue to denote this functional by $E$.}

To see why, for all $v\in T_{\phi_g}\mathcal{M}$, let $\psi = (\phi_g + t v)/(\|\phi_g + t v\|_{L^2})$ and consider the Taylor expansion:
\begin{align*}
	E(\psi) &= E(\phi_g) + \frac{1}{2}\langle (E''(\phi_g) - \lambda_g \mathcal{I}) (\psi-\phi_g), (\psi-\phi_g) \rangle + o(\|\psi-\phi_g\|_{H^1}^2)\\
	&=E(\phi_g) + \frac{1}{2}\langle (E''(\phi_g) - \lambda_g \mathcal{I}) tv, tv \rangle+ o(\|tv\|_{H^1}^2),
\end{align*}
where the second equality is based on \eqref{C_t1}.
For $E(\phi) \geq E(\phi_g)$ to hold for all sufficiently small $\sigma$ and $\phi\in\mathcal{B}_{\sigma}(\phi_g)$, we must control the quadratic term uniformly. If the second variation is only pointwise positive but not coercive, i.e., if
\[
\inf_{\substack{v \in T_{\phi_g}\mathcal{M} \\ \|v\|_{H^1} = 1}} \langle (E''(\phi_g) - \lambda_g \mathcal{I}) v, v \rangle = 0,
\]
then there exists a sequence $\{v_n\}_{n\in\mathbb{N}} \subset T_{\phi_g}\mathcal{M}$ with $\|v_n\|_{H^1} = 1$ such that the quadratic form tends to zero, and the higher-order remainder may dominate, preventing $E(\phi_g)$ from being a local minimum. Specifically, suppose that the remainder satisfies $o(\|tv\|^2_{H^1}) =-\|tv\|^3_{H^1}$. Let $t_n = \sqrt{\langle (E''(\phi_g) - \lambda_g \mathcal{I}) v_n, v_n \rangle}$ (if $o(\|tv\|^2_{H^1}) =\|tv\|^3_{H^1}$, let $t_n = -\sqrt{\langle (E''(\phi_g) - \lambda_g \mathcal{I}) v_n, v_n \rangle}$). Then we have
\[
\langle (E''(\phi_g) - \lambda_g \mathcal{I}) t_n v_n, t_n v_n \rangle = \langle (E''(\phi_g) - \lambda_g \mathcal{I}) v_n, v_n \rangle^2,
\]
and
\[
\|t_n v_n\|^3_{H^1} = \langle (E''(\phi_g) - \lambda_g \mathcal{I}) v_n, v_n \rangle^{3/2}.
\]
Since the exponent $3/2 < 2$, the cubic remainder term dominates the quadratic term as $n \to \infty$. Now define the normalized sequence
\[
\psi^n = \frac{\phi_g + t_n v_n}{\|\phi_g + t_n v_n\|_{L^2}}.
\]
This sequence lies on the constraint manifold $\mathcal{M}$, and the second-order sufficiency condition is satisfied at $\phi_g$. However, for sufficiently large $n$, we have $E(\psi^n) < E(\phi_g)$, as shown by the following expansion:
\begin{align*}
	E(\psi^n) - E(\phi_g)&=\frac{1}{2} \langle (E''(\phi_g) - \lambda_g \mathcal{I}) t_n v_n, t_n v_n \rangle + o(\|t_n v_n\|^2_{H^1})\\
	&=\frac{1}{2} \langle (E''(\phi_g) - \lambda_g \mathcal{I})v_n,v_n \rangle^2-\langle (E''(\phi_g) - \lambda_g \mathcal{I})v_n,v_n \rangle^{3/2}\\
	&=\left(\frac{1}{2}\sqrt{\langle (E''(\phi_g) - \lambda_g \mathcal{I})v_n,v_n \rangle}-1\right)\langle (E''(\phi_g) - \lambda_g \mathcal{I})v_n,v_n \rangle^{3/2}<0.
\end{align*}
This suggests that $\phi_g$ is not a local minimizer. Therefore, to prove that the second-order condition is sufficient to ensure the critical point is a minimizer, one must demonstrate that the scenario described earlier cannot occur for general functionals.

This difficulty underscores the need for stronger conditions in the infinite-dimensional setting. Thus, we contend that the standard second-order sufficient condition requires uniform positivity (coercivity) on the tangent space:
\[
\langle (E''(\phi_g) - \lambda_g \mathcal{I}) v, v \rangle \geq C \|v\|_{H^1}^2, \quad {\text{for all}}\; v \in T_{\phi_g}\mathcal{M},
\]
for some $C > 0$. {Of course, it should be noted that the above condition is stated for general functionals. For the Gross-Pitaevskii energy functional in the setting of this paper, the two formulations of the second-order condition remain equivalent due to the compact resolvent of the operator $E''(\phi_g)$. Moreover, under the Morse-Bott assumption, this equivalence also holds on the closed subspace $N_{\phi_g}\mathcal{M}$, as established in \textbf{Proposition~\ref{Prop2}}. Nevertheless, for the sake of generality, we adopt the stronger (coercive) form presented above.}

\section{Computation of $\mu$ and $L$ for the Quasi-optimal Preconditioner \eqref{Varep-P}}\label{Appendix-Opt-Pre}

The upper bound $L \leq 1$ is immediate from the inequality
\[
\frac{\langle (E''(\phi_g) - \lambda_g \mathcal{I}) v, v \rangle}{\langle (E''(\phi_g) - \lambda_g \mathcal{I}) v, v \rangle + \sigma_0 \|v\|_{L^2}^2} \leq 1,
\]
since $\sigma_0 > 0$ and the quadratic form in the numerator is non-negative for $v \in T_{\phi_g}\mathcal{M}$. To show that $L = 1$, it suffices to construct a sequence $\{v_n\}_{n\in\mathbb{N}}$ such that the ratio tends to 1 as $n \to \infty$. Recall that $E''(\phi_g)$ is an unbounded, self-adjoint, coercive operator with compact resolvent. Therefore, it admits a discrete spectrum with eigenpairs $(v_n, \mu_n)$ satisfying
\[
E''(\phi_g) v_n = \mu_n v_n,
\]
where $0 \leq \lambda_g < \mu_3 \leq \cdots \leq \mu_n \to \infty$ as $n \to \infty$. The first two eigenfunctions are given by $v_1 = i\phi_g$ and $v_2 = i\mathcal{L}_z \phi_g / \|i\mathcal{L}_z \phi_g\|_{L^2}$ (assuming $i\mathcal{L}_z \phi_g \not\in \mathrm{span}\{i\phi_g\}$, otherwise, $v_2 = v_1$), both associated with the eigenvalue $\mu_1 = \mu_2 = \lambda_g$. All eigenfunctions are normalized in $L^2$ and mutually orthogonal in $L^2$.
Since the eigenfunctions $\{v_n\}_{n\in\mathbb{N}}$ are $L^2$-orthogonal to $i\phi_g$ and $i\mathcal{L}_z\phi_g$, 
{$\text{Proj}_{\phi_g}^{L^2}v_n = v_n - (v_n, \phi_g )_{L^2} \phi_g \in N_{\phi_g}\mathcal{M}$} for $n\ge3$. We claim that the sequence $\left\{\text{Proj}_{\phi_g}^{L^2}v_n\in N_{\phi_g}\mathcal{M}\right\}_{n\ge3\in\mathbb{N}}$ is suitable for our purpose. It remains to show that 
\[
\langle E''(\phi_g) \text{Proj}_{\phi_g}^{L^2} v_n, \text{Proj}_{\phi_g}^{L^2} v_n \rangle \to \infty \quad \text{as } n \to \infty.
\]

To this end, consider the following two inequalities
\begin{align*}
	\langle E''(\phi_g) (\text{Proj}_{\phi_g}^{L^2} + I - \text{Proj}_{\phi_g}^{L^2}) v_n, (\text{Proj}_{\phi_g}^{L^2} + I - \text{Proj}_{\phi_g}^{L^2}) v_n \rangle &> 0, \\
	\langle E''(\phi_g) (\text{Proj}_{\phi_g}^{L^2} - I + \text{Proj}_{\phi_g}^{L^2}) v_n, (\text{Proj}_{\phi_g}^{L^2} - I + \text{Proj}_{\phi_g}^{L^2}) v_n \rangle &> 0.
\end{align*}
Note that $(\text{Proj}_{\phi_g}^{L^2} + I - \text{Proj}_{\phi_g}^{L^2})v_n = v_n$ and $(\text{Proj}_{\phi_g}^{L^2} - I + \text{Proj}_{\phi_g}^{L^2})v_n = (2\text{Proj}_{\phi_g}^{L^2} - I)v_n$, but more importantly, adding these inequalities yields
\[
\langle E''(\phi_g) v_n, v_n \rangle \leq 2 \langle E''(\phi_g) \text{Proj}_{\phi_g}^{L^2} v_n, \text{Proj}_{\phi_g}^{L^2} v_n \rangle + 2 \langle E''(\phi_g) (I - \text{Proj}_{\phi_g}^{L^2}) v_n, (I - \text{Proj}_{\phi_g}^{L^2}) v_n \rangle.
\]

Now observe that
\[
\langle E''(\phi_g) (I - \text{Proj}_{\phi_g}^{L^2}) v_n, (I - \text{Proj}_{\phi_g}^{L^2}) v_n \rangle = (\phi_g, v_n)^2 \langle E''(\phi_g) \phi_g, \phi_g \rangle\le C\quad \text{for}\quad n\ge3.
\]
Therefore, we obtain
\[
\mu_n = \langle E''(\phi_g) v_n, v_n \rangle \leq 2 \langle E''(\phi_g) \text{Proj}_{\phi_g}^{L^2} v_n, \text{Proj}_{\phi_g}^{L^2} v_n \rangle + C,
\]
which implies
\[
\langle E''(\phi_g) \text{Proj}_{\phi_g}^{L^2} v_n, \text{Proj}_{\phi_g}^{L^2} v_n \rangle \geq \frac{1}{2} \mu_n - \frac{C}{2} \to \infty \quad \text{as } n \to \infty.
\]
Consequently,
\[
\lim_{n \to \infty} \frac{\langle (E''(\phi_g) - \lambda_g \mathcal{I}) \text{Proj}_{\phi_g}^{L^2} v_n, \text{Proj}_{\phi_g}^{L^2} v_n \rangle}{\langle (E''(\phi_g) - \lambda_g \mathcal{I}) \text{Proj}_{\phi_g}^{L^2} v_n, \text{Proj}_{\phi_g}^{L^2} v_n \rangle + \sigma_0 \|\text{Proj}_{\phi_g}^{L^2} v_n\|_{L^2}^2} = 1.
\]
This proves that $L = 1$, independent of $\sigma_0$. We further address the lower bound $\mu = \frac{\lambda_3 - \lambda_g}{\lambda_3 - \lambda_g + \sigma_0}$. First, by the monotonicity of the function $x \mapsto \frac{x}{x + \sigma_0}$ for $x > 0$, which is {increasing}, we immediately obtain that for any $v \in N_{\phi_g}\mathcal{M}\backslash\{0\}$,
\begin{align*}
	\frac{\langle E''(\phi_g) v, v \rangle/\|v\|^2_{L^2}- \lambda_g }{\langle E''(\phi_g)  v, v \rangle/\|v\|^2_{L^2} - \lambda_g + \sigma_0}&=\frac{Q_{\phi_g}(v)- \lambda_g }{Q_{\phi_g}(v) - \lambda_g + \sigma_0}\\ 
	&\ge 
	\frac{\min\limits_{v \in N_{\phi_g}\mathcal{M}\backslash\{0\}} Q_{\phi_g}(v)- \lambda_g}{\min\limits_{v \in N_{\phi_g}\mathcal{M}\backslash\{0\}} Q_{\phi_g}(v) - \lambda_g + \sigma_0}=\frac{\lambda_3 - \lambda_g}{\lambda_3 - \lambda_g + \sigma_0}.
\end{align*}
Above, we utilized the property that the infimum of $Q_{\phi_g} $ on $N_{\phi_g}\mathcal{M}$ is achievable. This has been proven in {\bf Proposition \ref{Prop2}}. Therefore, the lower bound is obtained
\[
\mu = \frac{\lambda_3 - \lambda_g}{\lambda_3 - \lambda_g + \sigma_0}.
\]


\begin{thebibliography}{10}
	\bibitem{1988Manifold}
	R. Abraham, J. E. Marsden, and T. Ratiu, Manifolds, Tensor Analysis, and Applications, 2rd ed., Applied Mathematical Sciences, Vol. 75, Springer-Verlag, New York, 1988.
	
	\bibitem{2026JCAM}
	{ Y. Ai, P. Henning, M. Yadav, and S. Yuan}, { Riemannian conjugate Sobolev gradients and their application to compute ground states of BECs}, J. Comput. Appl. Math., 473 (2026), article 116866.
	
	\bibitem{2021The}
	{ R. Altmann, P. Henning, and D. Peterseim}, { The $J$-method for the Gross-Pitaevskii eigenvalue problem}, Numer. Math., 148 (2021), pp.~575--610.
	
	\bibitem{2024Riem}
	{ R. Altmann, M. Hermann, D. Peterseim, and T. Stykel}, { Riemannian optimization methods for ground states of multicomponent Bose-Einstein condensates}, IMA J. Numer. Anal., (2025), draf046.
	
	\bibitem{1995Observation}
	{ M. H. Anderson, J. R. Ensher, M.~R. Matthews, C.~E. Wieman, and E.~A. Cornell},
	{ Observation of Bose-Einstein condensation in a dilute atomic vapor}, Sci., 269 (1995), pp.~198--201.
	
	\bibitem{2014Robust}
	{ X. Antoine and R. Duboscq}, { Robust and efficient preconditioned Krylov spectral solvers for computing the ground states of fast rotating and strongly interacting Bose-Einstein condensates}, J. Comput. Phys., 258 (2014), pp.~509--523.
	
	\bibitem{2017Efficient}
	{ X. Antoine, A. Levitt, and Q. Tang}, { Efficient spectral computation of the stationary states of rotating Bose-Einstein condensates by the preconditioned nonlinear conjugate gradient method}, J. Comput. Phys., 343 (2017), pp.~92--109.
	
	\bibitem{2003Computing}
	{ W. Bao and Q. Du}, { Computing the ground state solution of Bose-Einstein condensates by a normalized gradient flow}, SIAM J. Sci. Comput., 25 (2004), pp.~1674--1697.
	
	\bibitem{2006Efficient}
	{ W. Bao, I. Chern, and F. Lim}, { Efficient and spectrally accurate numerical methods for computing ground and first excited states in Bose-Einstein condensates}, J. Comput. Phys., 219 (2006), pp.~836--854.
	
	
	
	\bibitem{2013Mathematical}
	{ W. Bao and Y. Cai}, { Mathematical theory and numerical methods
		for Bose-Einstein condensation}, Kinet. Relat. Models, 6 (2013), pp.~1--135.
	
	\bibitem{2014Intro}
	{ C. F. Barenghi, L. Skrbek, and K. R. Sreenivasan}, { Introduction to quantum turbulence}, PNAS, 111 (2014), pp.~4647--4652.
	
	\bibitem{1988BB}
	J. Barzilai, J. M. Borwein, Two-point step size gradient methods, IMA J. Numer. Anal. 8 (1988), pp. 141--148.
	
	\bibitem{1954Bott}
	{ R. Bott}, { Nondegenerate critical manifolds}, Ann. of Math., 60 (1954), pp.~248-–261.
	
	\bibitem{N.B.appear}
	{ N. Boumal}, { An Introduction to Optimization on Smooth Manifolds}, Cambridge University Press, to appear, http://www.nicolasboumal.net/book.
	
	
	\bibitem{2010Numerical}
	{ E. Canc\'es, R. Chakir, and Y. Maday}, { Numerical analysis of nonlinear eigenvalue problems}, J. Sci. Comput., 45 (2010), pp.~90--117.
	
	
	\bibitem{2013Quan}
	{ I. Carusotto and C. Ciuti}, { Quantum fluids of light}, Rev. Mod. Phys., 85 (2013), pp.~299--366.
	
	\bibitem{2003Semi}
	{ T. Cazenave}, { Semilinear Schr\"odinger Equations}, Courant Lect. Notes Math., 10, Amer. Math. Soc., Providence, R.I., 2003.
	
	\bibitem{2023Second}
	{ H. Chen, G. Dong, W. Liu, and Z. Xie}, { Second-order flows for
		computing the ground states of rotating {B}ose-{E}instein condensates},
	J. Comput. Phys., 475 (2023), article 111872.
	
	\bibitem{2024On}
	{ Z. Chen, J. Lu, Y. Lu, and X. Zhang}, { On the convergence of Sobolev gradient flow for the Gross-Pitaevskii eigenvalue problem}, SIAM J. Numer. Anal., 62 (2024), pp.~667--691.
	
	\bibitem{2000Ground}
	{ M. Chiofalo, S. Succi, and M. Tosi}, { Ground state of trapped interacting
		{B}ose-{E}instein condensates by an explicit imaginary-time algorithm},
	Phys. Rev. E, 62 (2000), pp.~7438--7444.
	
	
	\bibitem{2010A}
	{ I. Danaila and P. Kazemi}, { A new {S}obolev gradient method for
		direct minimization of the {G}ross-{P}itaevskii energy with rotation}, SIAM
	J. Sci. Comput., 32 (2010), pp.~2447--2467.
	
	\bibitem{2017Computation}
	{ I. Danaila and B. Protas}, { Computation of ground states of the
		{G}ross-{P}itaevskii functional via {R}iemannian optimization}, SIAM J. Sci. Comput., 39 (2017), pp.~B1102--B1129.
	
	\bibitem{K1995Bose}
	{ K. B. Davis, M. Mewes, and M. R. Andrews}, { Bose-{E}instein condensation
		in a gas of sodium atoms}, Phys. Rev. Lett., 75 (1995), pp.~3969--3973.
	
	\bibitem{2007Dion}  
	{ C. M. Dion and E. Canc\'es}, { Ground state of the time-independent Gross-Pitaevskii equation}, Comput. Phys. Commun., 177 (2007), pp.~787--798.
	
	\bibitem{2021Rota}
	{ L. Dong and Y. V. Kartashov}, { Rotating multidimensional quantum droplets}, Phys. Rev. Lett., 126 (2021), article 244101. 
	
	\bibitem{2017Convergence}
	{ E. Faou and T. J\'ez\'equel}, { Convergence of a normalized gradient algorithm for computing ground states}, IMA J. Numer. Anal., 38 (2017), pp.~360--376.
	
	
	\bibitem{2020L-S}
	{ P. M. Feehan and M. Maridakis},  { \L ojasiewicz-Simon gradient inequalities for analytic and Morse-Bott functions on Banach spaces}, J. Reine Angew. Math., 765 (2020), pp.~35--67 
	
	\bibitem{2001Optimi}
	{ J. J. Garc\'ia. Ripoll and V. M. P\'erez-Garc\'ia}, { Optimizing Schr\"odinger functionals using Sobolev gradients: Applications to quantum mechanics and nonlinear optics}, SIAM J. Sci. Comput., 23 (2001), pp.~1316--1334.
	
	\bibitem{2020Sobolev}
	{ P. Henning and D. Peterseim}, { Sobolev gradient flow for the
		{G}ross-{P}itaevskii eigenvalue problem: global convergence and computational
		efficiency}, SIAM J. Numer. Anal., 58 (2020), pp.~1744--1772.
	
	\bibitem{2023The}
	{ P. Henning}, { The dependency of spectral gaps on the convergence of the inverse iteration for a nonlinear eigenvector problem}, Math. Mod. Meth. Appl. S., 33 (2023), pp.~1517--1544.
	
	\bibitem{2024Ondis}
	{ P. Henning and M. Yadav}, { On discrete ground states of rotating Bose-Einstein condensates}, Math. Comp., 94 (2025), pp.~1--32.
	
	\bibitem{2024Convergence}
	{ P. Henning and M. Yadav}, { Convergence of a Riemannian gradient method for the
		Gross-Pitaevskii energy functional in a rotating frame}, ESAIM Math. Model. Numer. Anal., 59 (2025), pp. 1145--1175.
	
	\bibitem{2000Fuzzy}
	{ W. Hu, R. Barkana, and A. Gruzinov}, { Fuzzy cold dark matter: the wave properties of ultralight particles}, Phys. Rev. Lett., 85 (2000), pp.~1158--1161.
	
	\bibitem{2014An}
	{ E. Jarlebring, S. Kvaal, and W. Michiels}, { An inverse iteration method for eigenvalue problems with eigenvector nonlinearities}, SIAM J. Sci. Comput., 36 (2014), pp.~A1978--A2001.
	
	\bibitem{2010Tackling}
	{ P. Kazemi and M. Eckart}, { Minimizing the {G}ross-{P}itaevskii
		energy functional with the {S}obolev gradient-analytical and numerical
		results}, Int. J. Comput. Meth., 7 (2010), pp.~453--475.
	
	\bibitem{Klaers2010Bose}
	{ J. Klaers, J. Schmitt, F. Vewinger, and M. Weitz}, { 
		Bose-{E}instein condensation of photons in an optical microcavity},
	Nat., 468 (2010), pp.~545-548.
	
	\bibitem{1995Lang}
	S. Lang, Differential and Riemannian Manifolds, 3rd ed., Graduate Texts in Mathematics, Vol. 160, Springer, New York, 1995.
	
	\bibitem{2006Deri}
	{ E. H. Lieb and R. Seiringer}, { Derivation of the Gross-Pitaevskii equation for rotating Bose gases}, Commun. Math. Phys., 264 (2006), pp.~505--537 .
	
	\bibitem{2021Normalized}
	{ W. Liu and Y. Cai}, { Normalized gradient flow with {L}agrange
		multiplier for computing ground states of {B}ose-{E}instein condensates},
	SIAM J. Sci. Comput., 43 (2021), pp.~B219--B242.
	
	\bibitem{2010Sobolev}
	{ J. W. Neuberger}, { Sobolev Gradients and Differential Equations}, Springer Lecture Notes in Mathematics, 1670 (2010).
	
	\bibitem{2011An}
	{ L. Nicolaescu}, { An invitation to Morse theory}, New York, Springer, 2011.
	
	\bibitem{2006Numerical}
	{ J. Nocedal and S. J. Wright}, { Numerical Optimization}, New York, Springer, 2006.
	
	\bibitem{2020Supp}
	{ E. Shamriz, Z. Chen, and B. A. Malomed}, { Suppression of the quasi-two-dimensional quantum collapse in the attraction field by the Lee-Huang-Yang effect}, Phys. Rev. A., 101 (2020), article 063628.
	
	\bibitem{2019Rota}
	{ M. N. Tengstrand, P. St\"urmer, E. \"O. Karabulut, and S. M. Reimann}, { Rotating binary Bose-Einstein condensates and vortex clusters in quantum droplets}, Phys. Rev. Lett., 123 (2019), article 160405. 
	
	\bibitem{2017A}
	{ X. Wu, Z. Wen, and W. Bao}, { A regularized newton method for
		computing ground states of {B}ose-{E}instein condensates}, J. Sci. Comput.,
	73 (2017), pp.~303--329.
	
	\bibitem{2024Anew}
	{ T. Zhang and F. Xue}, { A new preconditioned nonlinear conjugate gradient method in real arithmetic for computing the ground states of rotational Bose-Einstein condensate}, SIAM J. Sci. Comput., 46 (2024), pp.~A1764--A1792.
	
	\bibitem{2022Exp}
	{ Z. Zhang}, { Exponential convergence of Sobolev gradient descent for a class of nonlinear eigenproblems}. Commun. Math. Sci., 20 (2022), pp.~377--403.
	
	\bibitem{2019Efficient}
	{ Q. Zhuang and J. Shen}, { Efficient {SAV} approach for imaginary time
		gradient flows with applications to one- and multi-component {B}ose-{E}instein {C}ondensates}, J. Comput. Phys.,
	396 (2019), pp.~72--88.
	
	
	
	
\end{thebibliography}
\end{document}